\documentclass[10pt]{amsart}
\usepackage{amsthm,amssymb,amsfonts,amsmath,mathrsfs}
\usepackage{appendix}
\usepackage[english]{babel}
\usepackage{nicefrac}  % stylish fraction

\usepackage{multirow} 
\usepackage{stackrel} %stack things
\usepackage{setspace}
\usepackage{enumerate,xspace}

\usepackage{dsfont}      %allows \mathds{E 1}, with sans is more heavy  %allows \mathds{E 1} and frak letters

\usepackage{verbatim} %comments and other help

\usepackage[colorlinks]{hyperref}
\DeclareMathAlphabet{\mathpzc}{OT1}{pzc}{m}{it}

\newcommand{\cA}{{\mathcal A}}
\newcommand{\cB}{{\mathcal B}}
\newcommand{\cC}{{\mathcal C}}
\newcommand{\cD}{{\mathcal D}}
\newcommand{\cE}{{\mathcal E}}

\newcommand{\cG}{{\mathcal G}}

\newcommand{\cL}{{\mathcal L}}
\newcommand{\cO}{{\mathcal O}}

\newcommand{\cN}{{\mathcal N}}
\newcommand{\cS}{{\mathcal S}}

\newcommand{\cV}{{\mathcal V}}
\newcommand{\cJ}{{\mathcal J}}

\newcommand{\cK}{{\mathcal K}}

\newcommand{\bC}{{\mathbb C}}

\newcommand{\bF}{{\mathbb F}}

\newcommand{\bH}{{\mathbb H}}

\newcommand{\bN}{{\mathbb N}}

\newcommand{\bR}{{\mathbb R}}

\newcommand{\bT}{{\mathbb T}}
\newcommand{\bZ}{{\mathbb Z}}
\newcommand{\bV}{{\mathbb V}}
\newcommand{\bW}{{\mathbb W}}
\newcommand{\ve}{\varepsilon}

\newcommand{\vf}{\varphi}
\newcommand{\param}{{\gamma}}
\newcommand{\bRp}{{\bR_{\scriptscriptstyle >}}}
\newcommand{\bRpe}{{\bR_{\scriptscriptstyle \geq}}}
\newcommand{\bNp}{{\bN_{\scriptscriptstyle >}}}

\newcommand{\Id}{{\mathds{1}}}

\newcommand{\balpha}{\boldsymbol \eta}
\newcommand{\bbeta}{\boldsymbol \sigma}

\newcommand{\hbomega}{\hat{\boldsymbol \omega}}
\newcommand{\bell}{\boldsymbol \ell}
\newcommand{\be}{\boldsymbol e}
\newcommand{\bx}{\boldsymbol x}
\newcommand{\Acc}{A}
\newcommand{\Bcc}{B}
\newcommand{\Dcc}{E}
\newcommand{\hAcc}{\widehat A}
\newcommand{\hBcc}{\widehat B}
\newcommand{\hDcc}{\widehat E}
\newcommand{\hL}{\widehat\cL}
\newcommand{\Const}{C_\#}
\newcommand{\const}{c_\#}

\newcommand{\supp}{\operatorname{supp}}

\newcommand\htop{h_{\textrm{top}}}

\newcommand\hV{\widehat V}
\newcommand\bg{{\boldsymbol g}}
\newcommand\bh{{\boldsymbol h}}
\newcommand\boO{{\boldsymbol O}}
\newcommand\oH{{\overline H}}
\newcommand\bgg{{\boldsymbol {\mathfrak g}}}
\newcommand\bphi{{\boldsymbol \phi}}
\newcommand\bpi{{\boldsymbol \pi}}
\newcommand\betas{\beta_{\textrm{ess}}}
\newcommand\pf{Q}
\newlength{\strikelength}

\numberwithin{equation}{section}

\newtheorem{thm}{Theorem}[section]
\newtheorem{lem}[thm]{Lemma}
\newtheorem{sublem}[thm]{Sub-Lemma}
\newtheorem{prop}[thm]{Proposition}
\newtheorem{cor}[thm]{Corollary}
\newtheorem{conj}[thm]{Conjecture}
\newtheorem{defin}[thm]{Definition}

\newtheorem{rem}[thm]{Remark}

\newcommand{\bi}{{\boldsymbol i}}

\newcommand{\bd}{{\boldsymbol d}}
\newcommand{\bv}{{\boldsymbol v}}
\newcommand{\bw}{{\boldsymbol w}}
\newcommand{\nablas}{\widehat\nabla}

\newcommand{\Leb}{\operatorname{Leb}}

\begin{document}

\title[Parabolic dynamics]{Parabolic dynamics and Anisotropic Banach spaces}
\author[P. Giulietti]{Paolo Giulietti}
\address{Paolo Giulietti \\
Scuola Normale Superiore - Centro di Ricerca Matematica Ennio de Giorgi\\ 
Piazza dei Cavalieri 7\\
56126 Pisa, Italy}
\email{{\tt paologiulietti.math@gmail.com}}
\author[C. Liverani]{Carlangelo Liverani}
\address{Carlangelo Liverani\\
Dipartimento di Matematica\\
II Universit\`{a} di Roma (Tor Vergata)\\
Via della Ricerca Scientifica, 00133 Roma, Italy.}
\email{{\tt liverani@mat.uniroma2.it}}
%\date{ \today. {\bf File: {\jobname}.tex.}}% eliminate in final version
\thanks{L.C. gladly thanks Giovanni Forni for several discussions through the years. Such discussions began in the very far past effectively starting the present work and lasted all along, in particular the last version of Lemma \ref{lem:fuck} owns to Forni ideas.  P.G. thanks L. Flaminio for explaining parts of his work and his hospitality at the university of Lille. Also we would like to thank Viviane Baladi, Alexander Bufetov, Oliver Butterley, Ciro Ciliberto, Livio Flaminio, Fran\c cois Ledrappier, Fr\'ed\'eric Naud, Ren\'e Schoof  and Corinna Ulcigrai for several remarks and suggestions that helped to considerably improve the presentation. We thank the anonymous referee for his hard work and thorough suggestions which forced us to substantially improve the paper.  This work was supported by the European Advanced Grant Macroscopic Laws and Dynamical Systems (MALADY) (ERC AdG 246953). P.G. was also partially supported by CNPq-Brazil project number 407129/2013-8 and by Instituto de Matem\'atica, Universidade Federal do Rio Grande do Sul, Porto Alegre, RS - Brazil. L.C. thanks the IHES for the kind hospitality during the revision of the paper.}

\begin{abstract}
We investigate the relation between the distributions appearing in the study of ergodic averages of parabolic flows (e.g. in the work of Flaminio-Forni) and the ones appearing in the study of the statistical properties of hyperbolic dynamical systems (i.e. the eigendistributions of the transfer operator). In order to avoid, as much as possible, technical issues that would cloud the basic idea, we limit ourselves to a simple flow on the torus. Our main result is that, roughly, the growth of ergodic averages (and the characterization of coboundary regularity) of a parabolic flows is controlled by the eigenvalues of a suitable transfer operator associated to the renormalizing dynamics. The conceptual connection that we illustrate is expected to hold in considerable generality.  
\end{abstract}

\maketitle

\section{Introduction} \label{sec:intro}

In the last decade, {\em distributions} have become increasingly relevant both in parabolic and hyperbolic dynamics. On the {\em parabolic dynamics} side consider, for example, the work of Forni and Flaminio-Forni \cite{Forni97, Forni02, FlaminioForni03, FlaminioForni06, FlaminioForni07} on ergodic averages and cohomological equations for horocycle flows or of Bufetov \cite{Bufetov14bis} on translation flows; on the {\em hyperbolic dynamics} side it suffices to mention the study of the transfer operator through anisotropic spaces, started with \cite{BKL02}.\footnote{ But see, e.g., \cite{ruelle-sullivan, Rugh96, Liverani95,  Kitaev99} for earlier related results.}

Since a frequent approach to the study of {\em parabolic dynamics} is the use of renormalization techniques,\footnote{ Typical examples are circle rotations \cite{SinaiKhanin89,SinaiKhanin92}, interval exchange maps via Teichmuller theory \cite{Viana06,ForniMatheus11}, horocycle flow \cite{Forni02, FlaminioForni03}.} where the renormalizing dynamics is often a {\em hyperbolic dynamics},
several people have been wondering on a possible relation between such two classes of distributions. Early examples of such line of thought can be found in Cosentino \cite[Section 3]{Cosentino2005} and Otal \cite{Otal1998}.

In this paper we argue that the distributional obstructions discovered by Forni and the distributional eigenvectors of  certain transfer operators are tightly related, to the point that, informally, one could say that they are exactly the same.

In order to present our argument in the simplest possible manner, instead of trying to develop it for the horocycle flow versus the geodesic flow (which would require a much more technical framework), we consider a very simple example that, while preserving the main ingredients of the horocycle-geodesic flow setting, allows to easily illustrate the argument. Yet, our example is not rigid (morally it corresponds to looking at manifolds of non constant negative curvature). So, notwithstanding its simplicity, it shows the flexibility of our approach, which has the potential of being greatly generalized. On the other hand, we cover only the case of periodic renormalization. Indeed, if the renormalizing dynamics are non linear, then it is not very clear how to define a good moduli space on which to act. The extension of our approach to the non periodic case remains an open problem.

Let us describe a bit more precisely our setting (see Section  \ref{sec:proof} for the exact, less discursive, description). As parabolic dynamics, we consider a flow $\phi_t$, over $\bT^2 =\nicefrac{\bR^2}{\bZ^2}$, generated by a vector field $V\in \cC^{1+\alpha}(\bT^2, \bR^2)$, $\alpha \in \bRp$, such that, for all $x \in\bT^2 $, $V(x) \neq 0$. As hyperbolic dynamics,  we consider a transitive Anosov map $F\in \cC^r(\bT^2,\bT^2)$, $r > 1 + \alpha $. By definition of Anosov map for all $ x \in \bT^2$ we have  $T_x \bT^2 =  E^s(x) \oplus E^u(x)$, where we used the usual  notation for the stable and unstable invariant distributions.\footnote{ Here ``distribution" refers to a field of subspaces in the tangent bundle and has nothing to do with the meaning of ``distribution" as generalized functions previously used. This is an unfortunate linguistic ambiguity for which we bear no responsibility.} Since we want the latter system to act as a renormalizing dynamics for the former, we require,
\begin{equation}\label{eq:renorm-hp}
  \forall x \in\bT^2,\; V(x)\in E^s(x).
\end{equation}

One might wonder which kind of flows admit the property \eqref{eq:renorm-hp} for some Anosov map $F\in\cC^r$. Here is a partial answer whose proof can be found in Appendix \ref{sec:classy}.

\begin{lem}\label{lem:classify} If a $\cC^{1+\alpha}$, $\alpha>0$, flow $\phi_t$, without fixed points, satisfies \eqref{eq:renorm-hp} for some Anosov map $F\in\cC^{r}$, $r\geq 1+\alpha$, then it is topologically conjugated to a rigid rotation with rotation number $\omega$ such that 
\begin{equation}\label{eq:diofanto}
b\,\omega^2+(a-d)\omega-c=0
\end{equation}
 for some $a,b,c,d\in\bZ$ such that $ad-cb=1$.
 
Each $\cC^{1+\alpha}$, $\alpha\geq 1$, flow $\phi_t$ without fixed points, or periodic orbits,  is topologically conjugated to a rigid rotation. If the rotation number satisfies \eqref{eq:diofanto} and $\alpha\geq 2$, then $\phi_t$ satisfies \eqref{eq:renorm-hp} for some Anosov map $F\in\cC^{\beta}$, for each $\beta<\alpha$.
\end{lem}

Note that the condition \eqref{eq:diofanto} can be restated by saying that $ \omega = r_1 + r_2  \sqrt{D}$ where $r_1 , r_2 \in  \mathbb{Q}, D \in \mathbb{N}, r_2 \neq 0$.
We stated the lemma in the above form because it connects better to the example worked-out in Section \ref{sec:examples}.

\begin{rem}\label{eq:examples} Even though the above Lemma shows that it is always possible to reduce our setting to a linear model by a  conjugation, such  conjugation is typically of rather low regularity. We will see shortly that requiring $F$, and related objects, to be of high regularity is essential for the questions we are interested in. It is not obvious to us how to characterize the flows for which \eqref{eq:renorm-hp} holds for very smooth $F$. Yet, such a condition clearly singles out some smaller class of flows (compared to Lemma \ref{lem:classify}) to which our theory applies.
Note however that there are plenty of examples, see Section \ref{sec:examples}.
\end{rem}

Equation \eqref{eq:renorm-hp} implies that the trajectories $\{\phi_t(x)\}_{t\in(a,b)}$ are pieces of the stable manifolds for the map $F$. Thus, if we assume $r \geq 2 + \alpha$,  we can define implicitly a function $\nu_n\in\cC^{1+\alpha}(\bT^2,\bR)$ such that
\begin{equation} \label{eq:defvn}
D_xF^n V(x)=\nu_n(x)V(F^n(x)), 
\end{equation}
where ${|\nu_n|< \Const \lambda^{-n}}$ for some $\lambda >1$. Without loss of generality we assume that $F$ preserves the orientation of the invariant splitting, i.e. $\nu_n>0$ (if not, use $F^2$). 

Given the hypothesis \eqref{eq:renorm-hp} it is natural to ask, at least, that, for each $x\in\bR^2$, the flow is regular with respect to the time coordinate i.e.  
\begin{equation}\label{eq:smooth-hp}
\phi_{(\cdot)}(x)\in\cC^{r}.
\end{equation}
In fact, we will use slightly stronger hypotheses, see Definition \ref{def:flowmapsetup} and Remark \ref{rem:repara}.

The reader may complain that the parabolic nature of the dynamics $\phi_t$ it is not very apparent. Indeed, a little argument is required to show that $\|D_x\phi_t\|$ can grow at most polynomially in $t$, see Section \ref{subsec:parabolic}, and some more work is needed to show that there are cases in which it is truly unbounded, see Lemma \ref{lem:true-parabolic}.

Let us detail an easy consequence of \eqref{eq:defvn}. If, for each $n\in\bN$, we define $\eta_n\in\cC^{1+\alpha}(\bT^2\times\bR,\bT^2)$ as $\eta_n(x,t)=F^n(\phi_t(x))$, then we have
\[
\left\{ \begin{array}{l}
\frac d{dt}\eta_n(x,t)=D_{\phi_t(x)}F^nV(\phi_t(x))=\nu_n(\phi_t(x)) V(\eta_n(x,t))\\
\eta_n(0)=F^n(x).
\end{array} \right. 
\]
It is then natural to define the time change\footnote{ By construction, for each $x\in\bT^2$ and $n\in\bN$, $\tau_n(x,t)$ is a strictly increasing function of $t$, and hence globally invertible. We will use the, slightly misleading, notation $\tau_n^{-1}(x,\cdot)$ for the inverse.}
\begin{equation}\label{eq:param-tau}
\tau_n(x, t)=\int_0^t ds \, \nu_n(\phi_s(x)),
\end{equation}
and introduce the function $\zeta_n\in\cC^{1+\alpha}$ by $\zeta_n(\tau_n(x,t),x)=\eta_n(x,t)$. Then
\begin{equation}\label{eq:renorm-glob}
\left\{ \begin{array}{l}
\frac{d}{dt}\zeta_n(x,t)=V(\zeta_n(x,t))\\
\zeta_n(x,0)=F^n(x).\\
\end{array} \right. 
\end{equation}
By the uniqueness of the solution of the above ODE, it follows, for all $t\in\bRp$
\[
\phi_{t}(F^n(x))=\zeta_n(x,t)=\eta_n(\tau_n^{-1}(x,t),x)=F^n(\phi_{\tau_n^{-1}(x,t)}(x)).
\]
In other words, the image under $F^n$ of a piece of trajectory, is the reparametrization of a (much shorter) piece of trajectory:
\begin{equation} \label{eq:commute}
F^n  (\phi_t (x)) = \phi_{\tau_{n}(x,t)}(F^n(x)).
\end{equation}
The above is the basic renormalization equation for the flow $\phi_t$ that we will use in the following. 

Note that, by Lemma \ref{lem:classify} and Furstenberg \cite{Furstenberg61}, the flow is uniquely ergodic, since its  Poincar\'e map  is uniquely ergodic. Let $\mu$ be the unique invariant measure. In addition, the flow is also minimal since it is topologically conjugated to a minimal flow (the linear one).

By unique ergodicity, given $g\in\cC^0(\bT^2,\bR)$, $\frac 1t \int_0^t ds \, g\circ \phi_s(x)$ converges uniformly to $\mu(g)$.
We have thus naturally arrived at our
\vskip.2cm
{\bf First question}: How fast is the convergence to the ergodic average?
\vskip.2cm
\noindent The question is equivalent to investigating the precise growth of the functionals $H_{x,t}:\cC^r\to\bR$ defined by
\begin{equation} \label{eq:htau}
H_{x,t}(g) := \int_0^t ds \, g\circ \phi_s(x) .
\end{equation}
Of course, if $\mu(g)\neq 0$, then $H_{x,t}(g)\sim \mu(g)t$, but if $\mu(g)=0$, then we expect a slower growth.
\begin{rem} Note that the growth rate of an ergodic integral for functions of a given smoothness it is not a topological invariant, hence the fact that our systems can be topologically conjugated to a linear model, as stated in Lemma \ref{lem:classify}, is not of much help.
\end{rem}

In the work of Flaminio-Forni \cite{FlaminioForni03} is proven that the functionals \eqref{eq:htau}, there defined for the horocycle flow on a surface of constant negative curvature, have a polynomial growth with exponent determined by a countable number of obstructions. That is, the growth is slower if the function $g$ belongs to the kernel of certain set of functionals. The remarkable discovery of Forni (going back to \cite{Forni97, Forni02}) is that the possible power growths form a discrete set and that the associated obstructions cannot be expected, in general, to be measures: they are {\em distributions}.\footnote{ Apart, of course, for the first that, as above, is the invariant measure $\mu$.}  See Remark \ref{rem:horo} for further details.

In analogy with the above situation, one expects that in our simple model there exist a finite number\footnote{ For an explanation of ``finitely many versus countably many" see, again, Remark  \ref{rem:horo}.} of functionals $\{O_i\}_{i=1, \ldots, N_1}\subset \cC^r(\bT^2,\bR)'$, and a corresponding set $\{\alpha_i \}_{i=1, \ldots, N_1}$ of decreasing numbers $\alpha_i\in [0,1]$ such that if $O_j(g)=0$ for all $j<i$ and $O_i(g)\neq 0$, then $H_{x,t}(g)=\cO(t^{\alpha_i})$. As we mentioned just after equation \eqref{eq:htau}, $O_1(g)=\mu(g)$ with $\alpha_1=1$.

Next, suppose that $O_i(g)=0$ for all $i\leq N_1$ and $\alpha_{N_1}=0$. That is, $H_{x,t}(g)$ remains bounded. By Gottschalk-Hedlund theorem \cite{GH55} this implies that $g$ is a continuous coboundary for the flow (since the flow is minimal). To investigate the regularity of the coboundary it is convenient to start with an alternative proof of Gottschalk-Hedlund theorem (limited to our context). For each $T\in\bRp$, consider the new functionals ${\oH_{T} :\cC^r \to\cC^{1+\alpha}}$ defined by, for all $x\in\bT^2$,
\begin{equation} \label{def:eraverage}
\oH_{T}(g) (x):=  -\int_0^T dt  \,\chi\circ \tau_{n_T}(x,t) g\circ \phi_t(x)= -\int_{\bRpe} dt  \,\chi\circ \tau_{n_T}(x,t) g\circ \phi_t(x), 
\end{equation}
where $n_T+1=\inf\{n\in\bN\;:\; \inf_x\tau_n(x,T)\leq 1\}$ and $\chi\in\cC^r(\bRpe,[0,1])$ 
is a fixed function such that $\chi(s)=1$ for all $s\leq 1/2$ and $\chi(s)=0$ for all $s\geq 1$. Such a function $\chi$ can be thought as a ``smoothing'' of $\overline \chi\circ\tau_{n_T}:=\max\{0,\frac {T-s}{T}\}$. Unfortunately, we cannot use $\overline \chi$  because such a choice would create serious difficulties later on (e.g., in the decomposition carried out in equation \eqref{eq:newstart1}), yet the reader can substitute $\overline \chi$ to $\chi$ to have an intuitive idea of what is going on.

The next Lemma restates Gottschalk-Hedlund and will be proved in Section \ref{sec:measurable-cob}.
\begin{lem} \label{lem:cob-prelim}
For each $g\in\cC^r(\bT^2,\bR)$, $r$ large enough,\footnote{ The condition on $r$ is the same as in Theorem \ref{thm:main}.} such that $\cO_i(g)=0$ for all $i\in\{1,\dots, N_1\}$, we have that $\oH_{T}(g)$ converges uniformly, as $T\to\infty$, to some continuous function $h$ such that
\begin{equation} \label{eq:homology0}
 h\circ\phi_t(x)-h(x) = \int_0^t ds \, g\circ \phi_s(x) . 
\end{equation}
\end{lem}

It follows that  $h$ is weakly differentiable in the flow direction and
\begin{equation} \label{eq:homology}
g(x) =\langle V(x),\nabla h(x)\rangle.
\end{equation}
That is, $g$ is a continuous coboundary.

The existence of a continuous coboundary is of some interest; but much more interesting is the existence of more regular solutions of \eqref{eq:homology} since this plays a role in establishing many relevant properties (see \cite[Sections 2.9, 19.2]{KH}). Hence our 
\vskip.2cm
{\bf Second question}: How regular are the solutions of the cohomological equation \eqref{eq:homology}?
\vskip.2cm
\noindent Following Forni again, we expect that there exist finitely many distributional obstructions $\{O_i\}_{i=N_1+1, \ldots, N_2}\subset \cC^r(\bT^2,\bR)'$ and a set of increasing numbers \newline $\{r_i\}_{i=N_1+1, \ldots, N_2}$, $r_i \in (0,1+\alpha)$ such that, if $O_j(g)=0$ for all $j<i$ and $O_i(g)\neq 0$, then $h\in \cC^{r_i}$.
\begin{rem}\label{rem:low-reg} Note that in the present context, as the flow is only $\cC^{1+\alpha}$, it is not clear if it makes any sense to look for coboundaries better than $\cC^{1+\alpha}$. This reflects the fact that if one looks at the horocycle flows on manifolds of non constant negative curvature, then the associate vector field is, in general, not very regular. On the other hand rigidity makes not so interesting our simple example when both foliations are better than $\cC^2$, \cite[Corollary 3.3]{ghys93}. We will therefore limit ourself to finding distributions that are obstruction to Lipschitz coboundaries, i.e. if $O_i(g)=0$ for all $i\leq N_2$, then $h$ is  Lipschitz (see Theorem \ref{thm:maintwo}). We believe this to be more than enough to illustrate the scope of the method.
\end{rem}

The goal of this paper is to prove the above facts by studying transfer operators associated to $F$, acting on appropriate spaces of distributions. 
In fact, we will show that the above mentioned obstructions $\{O_i\}$ can be obtained from the eigenvectors of appropriate transfer operators associated to $F$. As announced, this discloses the connection between the appearance of distributions in two seemingly different fields of dynamical systems.

\begin{rem} 
As already mentioned, in our model $\phi_t$ plays the role of the horocycle flow, while $F$ the one of the geodesic flow. It is important to notice that most of the results obtained for the horocycle flows (and Flaminio-Forni's results in particular) rely on representation theory, thus requiring constant curvature of the space.  In our context, this would correspond to the assumption that $F$ is a toral automorphism and $\phi_t$ a rigid translation. One could then do all the needed computations via Fourier series (if needed, see Section \ref{rem:Fouriercomputation} for details). It is then clear that extending our approach to more general parabolic flows, e.g. horocycle flows, (which should be quite possible using the results on flows by \cite{GLP13,Faure-Tsujii13, Dyatlov-Zworski}) would allow to treat cases of variable negative curvature, and, more generally, cases where the tools of representation theory are not available or effective, whereby greatly extending the scope of the theory. To our knowledge the only other approach trying generalize the theory in such a direction is contained in the papers \cite{Bufetov, Bufetov14}. However Bufetov's strategy relays on a coding of the system. Hence, it seems to suffer from the same limitations that affect the Markov partition approach to the study of hyperbolic systems. In particular, using coding techniques only a small portion of the transfer operator spectrum is accessible. These are exactly the limitations that the techniques used in this paper were designed to overcome.
It would therefore be very interesting, and (we believe) possible, to extend the present approach to the setting described in \cite{Bufetov14}.
\end{rem}

 The plan of the paper is as follow: in section \ref{sec:proof} we state our exact assumptions, outline our reasoning and state precisely our results, assuming lemmata and constructions which are explained later on. Section \ref{sec:growth} is devoted to our first question and proves our Theorem \ref{thm:main} concerning the distributions arising from the study of the ergodic integrals. Section \ref{sec:cob} deals with our second question and proves our Theorem \ref{thm:maintwo} dealing with the distributions arising from the study of the regularity of the cohomological equation.
Section \ref{sec:examples} is devoted to the discussion of examples. This is a rather long section, yet we think it is important as it explicitly shows how the abstract theory can be concretely applied and gives a clear idea of the type of work involved.
In Section \ref{rem:Fouriercomputation} we work out explicitly the simplest possible situation: a linear flow renormalized by a toral automorphism. Note that in this case 
all the obstructions generate by our scheme reduce, as it should be, to the Lebesgue measure. Yet, even in this simple situation, our scheme unveils much more structure than expected.  In Section \ref{sec:modelp} we present a more general (non linear) one parameter family of systems and we prove that it behaves not so differently from the linear model: either the ergodic integral of a function $g$ grows linearly in time or $g$ is at least a $\cC^{\frac 12- \const \param}$ coboundary, where $\param$ is the perturbative parameter of the family of examples (see Lemmata \ref{thm:lip}, \ref{lem:fuck} and Corollary \ref{cor:cob} for details). In sections \ref{sub:perturb}, \ref{sec:trulyp} we use perturbation theory to show that the flow is, generically, truly parabolic, that is the derivative has a polynomial growth (see Lemma \ref{lem:true-parabolic}).  In the Appendix \ref{sec:classy} we provide the details for some facts mentioned in the introduction without proof. In the Appendices \ref{sec:norms} and \ref{sec:norms-bis} we recall the definition of the various functional spaces needed in the following (adapted to the present setting).
\vskip.2cm
{\bf Notation.} When convenient, we will use $\Const, \const$ to designate a generic constant, depending only on $F$ and $\phi_1$, and $C_{a,b,\dots}$ for a generic constant depending also from $a,b,\dots$. Be advised that the actual value of such constants may change from one occurrence to the next.

\section{Definitions and main Results} \label{sec:proof}
In this section we will introduce rigorously the model loosely described in the introduction and explicitly state our results. Unfortunately, this requires quite a bit of not so intuitive notations and constructions, which call for some explanation. The experienced reader can jump immediately to Theorems \ref{thm:main}, \ref{thm:maintwo} but we do not recommend it in general.

Let $\alpha, r\in\bRp$ with $r>2+\alpha$.
 
We start by recalling the definition of $\cC^r$ Anosov map of the torus.

\begin{defin} \label{def:AnosovMap0}
Let $ F\in \cC^r(\bT^2,\bT^2)$ where $\bT^2 =\nicefrac{\bR^2}{\bZ^2}$. The map is called Anosov if there exists two continuous closed nontrivial transversal cone fields $C^{u,s}:\bT^2\to\bR^2$ which are strictly $DF$-invariant. That is, for each $x\in\bT^2$, 
\begin{equation}\label{eq:cone-inv}
\begin{split}
&D_xF C^u(x)\subset \text{Int } C^u(F(x))\cup\{0\}\\
&D_xF^{-1} C^s(x)\subset  \text{Int } C^s(F^{-1}(x))\cup\{0\}.
\end{split}
\end{equation}
In addition, there exists $C>0$ and $\lambda>1$ such that, for all  $n \in\bN$, 
\begin{equation}\label{eq:asplit1} 
\begin{array}{lll}
\|D_xF^{-n} v\| &> C \lambda^n \|v\|   \quad&\text{ if } 
v \in C^s(x) ; \\ 
\|D_xF^{n} v\|&>  C \lambda^{n}  \|v\|    &\text{ if } v
\in C^u(x).  \\
\end{array}
\end{equation} 
\end{defin}

It is well known that the above implies the following, seemingly stronger but in fact equivalent \cite{KH}, definition
\begin{defin} \label{def:AnosovMap}
Let $ F\in \cC^r(\bT^2,\bT^2)$. The map is called {\em Anosov} if there exists a $DF$-invariant 
$\cC^{1+\alpha}$, $r-1\geq \alpha>0$, splitting $T_x M = E^s(x) \oplus E^u(x)$ and constants $C, \lambda > 1 $
such that for $n \geq 0$ 
\begin{equation}\label{eq:asplit} 
\begin{array}{lll}
\|DF^n v\| &< C \lambda^{-n} \|v\|   \quad&\text{ if } 
v \in E^s ; \\ 
\|DF^{n} v\|& <  C \lambda^{n}  \|v\|    &\text{ if } v
\in E^u.  \\
\end{array}
\end{equation} 
\end{defin}
As already mentioned we assume that the stable distribution $E^s$ is orientable and that $F$ preserves such an orientation.
Further note that, since $F$ is topologically conjugated to a toral automorphism \cite[Theorem 18.6.1]{KH}, $F$ is topologically transitive.

Next, we consider a flow $\phi_t$  generated by a vector field $V$ satisfying the following properties.
\begin{defin} \label{def:flowmapsetup}
Let  the vector field $ V$ be such that
\begin{enumerate}[(i)]
 \item $ V \in \cC^{1+\alpha}(\bT^2,\bR^2)$;
 \item $ \|V\| \in \cC^{r}(\bT^2,\bRp)$;\label{enum:2}
 \item for all $x\in\bT^2$, $ V(x) \neq 0  $; \label{enum:3}
 \item for all $x\in\bT^2$, $ V(x) \in E^s(x) $ \label{enum:4}.
\end{enumerate}
\end{defin}

\begin{rem} \label{rem:repara}
Note that Definitions \ref{def:flowmapsetup}-\eqref{enum:2} and \eqref{enum:4} imply condition \eqref{eq:smooth-hp} since, being $F\in\cC^r$, so are the stable leaves \cite{KH}. In fact, Definition \ref{def:flowmapsetup}-\eqref{enum:2}  essentially implies that we are just considering $\cC^r$ time reparametrizations of the case $\|V\|=1$. Hence, we are treating all the $\cC^r$ reparametrizations on the same footing. This is rather convenient although not so deep in the present context. Yet, it could be of interest if the present point of view could be extended to the study of the mixing speed  of the flow. Indeed, there is a scarcity of results on reparametrization of parabolic flows (see \cite{Forni-Ulcigrai} for recent advances).
\end{rem}
Remembering  \eqref{eq:commute}, \eqref{eq:param-tau} and using the definition \eqref{eq:htau},
\begin{equation}\label{eq:preliminary}
\begin{split}
H_{x,t}(g)&=\int_0^t ds \,g\circ F^{-n}\circ\phi_{\tau_n(x,s)}(F^n(x))\\
&=\int_0^{\tau_n(x,t)} ds_1\, \left(\frac{g}{\nu_n}\right)\circ F^{-n}\circ\phi_{s_1}\circ F^n(x).
\end{split}
\end{equation}
It is then natural to introduce the {\em transfer operator} $\cL_F \in L(\cC^0,\cC^0)$,\footnote{ Given a map $F$, in general a {transfer operator} associated to $F$ has the form $\vf\to\vf\circ F^{-1}e^{\phi}$ for some function $\phi$. Normally, the factor $e^\phi$ is called the {\em weight} while $\phi$ is the {\em potential},  \cite{Baladi00}.} 
 \begin{equation} \label{eq:operator}
 \begin{split}
 \cL_F (g) :=&  (\nu_1\circ F^{-1})^{-1} g\circ F^{-1}= g \circ F^{-1}\frac{\|V\|}{\|DFV\|\circ F^{-1}} \\ 
 =&g \circ F^{-1}\frac{\|DF^{-1}V\|}{\|V\|\circ F^{-1}},
 \end{split}
 \end{equation}
where we have used \eqref{eq:defvn}.

 We can now write
 \begin{equation}\label{eq:basic-step}
 H_{x,t}(g)=\int_0^{\tau_n(x,t)}\hskip-.6cm ds_1\; (\cL_F^n g)(\phi_{s_1}(F^n(x)))=H_{F^n(x),\tau_n(x,t)}(\cL_F^ng).
 \end{equation}
 
 The above formula is quite suggestive: for each $x\in\bT^2$ and $t\in\bRp$, if we fix $n=n_t(x)$ such that $\tau_{n_t}(x,t)$ is of order one, then  $H_{x,t}(g)$ is expressed in terms of very similar functionals of $\cL_F^{n_t} g$. Note that such functionals are uniformly bounded on $\cC^0$ with respect to $(x,t)$, that is: they can be seen as measures and, as such, they have uniform total variation. One can then naively imagine that to address the questions put forward in the introduction it suffices to understand the behavior of $\cL_F^n$ for large $n$. This obviously is determined by the spectral properties of $\cL_F$.

 Unfortunately, it is well known that the spectrum of $\cL_F$ depends strongly on the Banach space on which it acts. For example, in the trivial case when $F$ is a toral automorphism and $\phi_t$ a rigid translation with unit speed, $e^{-\htop}\cL_F$ acting on $L^2$ is an isometry,\footnote{ As usual $\htop$ stands for the topological entropy of the map $F$ we are considering.} hence the spectrum of $\cL_F$ consists of the circle of radius $e^{\htop}$. The spectrum on $\cC^0$ it is not much different. On the contrary, if we consider $\cL_F$ acting on $\cC^r$, then the spectral radius will be given by $e^{(r+1)\htop}$. 
  
This seems to render completely hopeless the above line of thought.

Yet, as mentioned in the introduction, it is possible to define norms $\| \cdot \|_{p,q}$ and associated anisotropic Banach spaces $\cC^{p+q}\subset\cB^{p,q}\subset (\cC^q)'$,  $p \in \bN^*, q \in \bRp$, $p+q \leq r $, such that each transfer operator with $\cC^r$ weight  can be continuously extended to $\cB^{p,q}=\overline{\cC^r}^{\|\cdot\|_{p,q}}$.  The above are spaces of distributions (a fact that the reader might find annoying) but, under mild hypotheses on the weight used in the operators, several remarkable properties hold true\footnote{ Before \cite{BKL02} it was unclear if spaces with such properties existed at all. Nowadays there exists a profusion of possibilities. We use the ones stemming from \cite{GouezelLiverani06, GouezelLiverani08} because they seem particularly well 
suited for the task at hand, but any other possibility (e.g.  \cite{BaladiTsujii00, BaladiTsujii08}) should do.} 
\begin{enumerate}[i)]
\item a transfer operator (with $\cC^r$ weight) extends by continuity from $\cC^r$ to a bounded operator on $\cB^{p,q}$; 
\item such a transfer operator is a quasi-compact operator with a simple maximal eigenvalue; 
\item the essential spectral radius of the transfer operator decreases exponentially with $\inf\{q,p\}$;
\item the point spectrum is stable with respect to deterministic and random perturbations.
\end{enumerate}

The possibility to make the essential spectrum arbitrarily small, by increasing $p$ and $q$, will play a fundamental role in our subsequent analysis. Unfortunately, a further problem now arises: the weight of $\cL_F$ contains the vector field $V$ which, by hypothesis, is only $\cC^{1+\alpha}$. Hence $\cL_F$ leaves $\cC^r$ invariant only for $r\leq 1+\alpha$ (exactly the range in which we are not interested). Again it seems that we cannot use our strategy in any profitable manner. 

Yet, such a problem has been overcame as well, e.g., in \cite{GouezelLiverani08}. The basic idea is to extend the dynamics $F$ to the oriented Grassmannian. Indeed, looking at \eqref{eq:operator}, it is clear that the weight can be essentially interpreted as the expansion of a volume form. The simplest idea would then be to let the dynamics act on one forms on $\bT^2$. Unfortunately, the length cannot be written exactly as a volume form on $\bT^2$, hence the convenience of being a bit more sophisticated: the weight of $\cL_F$ can be written as the expansion of a one dimensional volume form on the vector space containing $V$. As $V$ is exactly the tangent vector to the curves along which we 
integrate, we are led, as in \cite{GouezelLiverani08}, to consider functions on the Grassmannian made by one dimensional subspaces. However, in the simple case at hand, the construction in \cite{GouezelLiverani08} can be considerably simplified. Namely, we can limit ourselves to considering the compact set $\Omega_*=\{(x,v)\in \bT^2\times\bR^2\;:\; \|v\|=1,\, v\in \overline{C^s(x)}\}$.\footnote{ Note that $\Omega_*$ is a subset of the unitary tangent bundle of $\bT^2$.} Moreover, since we have assumed that the stable distribution is orientable, then $\Omega_*$ is the disjoint union of two sets (corresponding to the two possible orientations). Let $\Omega$ be the connected component that contains the elements $(x,\hV(x))$, where $\hV(x)=\|V(x)\|^{-1} V(x)$. In addition, since we have also assumed that $F$ preserves the orientation of the stable distribution, calling $\bF$ the lift to the unitary tangent bundle of $F$ we have $\Omega_0=\bF^{-1}(\Omega)\subset \Omega$. 

Thus we have that $\bF:\Omega_0\subset\Omega\to\Omega$ is defined as 
\[
\begin{split}
&\bF(x,v)=(F(x), \|D_xFv\|^{-1}D_xFv),\\
&\bF^{-1}(x,v)=(F^{-1}(x), \|D_xF^{-1}v\|^{-1}D_xF^{-1}v).
\end{split}
\]
Also note that
\begin{equation}\label{eq:inv-stable-vec}
\bF^{-1}(x,\hV(x))=(F^{-1}(x),\hV(F^{-1}(x))).
\end{equation}
Hence, if we define the natural extension $\bphi_t(x, v)=(\phi_t(x), \|D_x\phi_tv\|^{-1}D_x\phi_tv)$, remembering \eqref{eq:commute}, we have\footnote{ See equation \eqref{eq:a-palle0} if details are needed.} 
\[
\begin{split}
&\bF^n  (\bphi_s (x,v)) = (F^n  (\phi_t (x)),\|D_x[F^n \circ \phi_t] v\|^{-1} D_x[F^n \circ \phi_t] v)\\
&=\left(\phi_{\tau_{n}(x,t)}(F^n(x)),\frac{D_{F^n(x)}\phi_{\tau_n(x,t)}D_xF^nv+V(\phi_{\tau_{n}(x,t)}(F^n(x))\langle\nabla\tau_n(x,t),v\rangle}{\|D_{F^n(x)}\phi_{\tau_n(x,t)}D_xF^nv+V(\phi_{\tau_{n}(x,t)}(F^n(x))\langle\nabla\tau_n(x,t),v\rangle\|}\right).
\end{split}
\]
The above formula does not look very nice, however we will be only interested in integrations along the flow direction. Accordingly, by \eqref{eq:defvn} and 
\begin{equation} \label{eq:flipvphi}
 D_x \phi_s (V(x)) = D_x \phi_s  \left. \frac{d}{d\tau} \phi_\tau (x) \right|_{\tau = 0} =  \left. \frac{d}{d\tau} \phi_{s+\tau} (x) \right|_{\tau = 0}=  V(\phi_s(x)),\end{equation}
we have $D_{F^n(x)}\phi_{\tau_n(x,t)}D_xF^n V(x)=\nu_n(x)V(\phi_{\tau_{n}(x,t)}(F^n(x))$.
Hence, limited to $v=\hV(x)$, we recover an analogue of \eqref{eq:commute}:
\begin{equation}\label{eq:commute2}
\bF^n  (\bphi_s (x,\hV(x))) = \bphi_{\tau_{n}(x,s)}(\bF^n(x, \hV(x))).
\end{equation}
\begin{rem}\label{rem:hyp-ext}
Note that $\bF$ is itself a uniformly hyperbolic map with the two dimensional repellor $\{(x,v)\in\Omega\;:\; v=\widehat{V}(x)\}$ and it has an invariant splitting of the tangent space with two dimensional unstable distribution and one dimensional stable.
\end{rem}
Next, we define the transfer operator associated to $\bF: \cC^0(\Omega_0,\bR) \to \cC^0(\Omega,\bR)$ as
\begin{equation}\label{eq:new-to} 
\cL_{\bF} \bg(x,v)=\bg\circ \bF^{-1}(x,v) \frac{\|D_xF^{-1}v\|\,\|V(x)\|}{ \|V\circ F^{-1}(x)\|}.
\end{equation}

The key observation is that $\bpi\circ \bF^{-1}=F^{-1}\circ \bpi$, where we have introduced the projection $\bpi(x,v)=x$. Hence, for each function $g\in \cC^r(\bT^2,\bR)$, if we define $\bg=\bpi^*g:=g\circ \bpi$, then $\bg\in\cC^r(\Omega,\bR)$ and we have, for all $n\in\bN$,
\begin{equation}\label{eq:LtoL}
\cL_{\bF}^n\bg(x,\hV(x))=\cL_F^n g(x).
\end{equation}
The above shows that understanding the properties of $\cL_\bF$ allows to control $\cL_F$. In addition, from definition \eqref{eq:new-to} it is apparent that $\cL_\bF(\cC^{r-1}(\Omega,\bR))\subset \cC^{r-1}(\Omega,\bR)$.\footnote{ Recall Definition \ref{def:flowmapsetup}-\eqref{enum:2}.} We have thus completely eliminated the above mentioned regularity problem.

Accordingly, we define, for each $\bg\in\cC^0(\Omega,\bR)$, $t\in\bRp$ and $x\in\bT^2$, the new functional
\begin{equation}\label{eq:HH-func}
\bH_{x,t}(\bg)=\int_0^t ds \, \bg\left(\phi_s(x), \hV(\phi_s(x))\right)
\end{equation}
and easily obtain an analogue of \eqref{eq:basic-step} for the operator $\cL_\bF$.
\begin{lem}\label{lem:adjust} 
For each $\bg\in\cC^0(\Omega,\bR)$ and $g\in\cC^0(\bT^2,\bR)$, $n\in\bN$, $t\in\bRp$ and $x\in\bT^2$ we have
\[
\begin{split}
&\bH_{x,t}(g\circ \bpi)=H_{x,t}(g)\\
&\bH_{x,t}(\bg)=\bH_{F^n x,\tau_n(x,t)}(\cL_{\bF}^n\bg) .
\end{split}
\]
\end{lem}
\begin{proof}
The proof of the first formula is obvious by the definition, the second follows by direct computation using \eqref{eq:commute2}:
\[
\begin{split}
\bH_{x,t}(\bg)&=\int_0^t ds \, \bg\left(\phi_s(x), \hV(\phi_s(x))\right) =\int_0^t ds \, \bg\circ\bF^{-n}\circ\bF^n\circ\bphi_s(x, \hV(x))\\
&=\int_0^{\tau_n(x,t)} ds_1\, \nu_n(\bpi\circ\bF^{-n}\circ\bphi_{s_1}(\bF^n(x,\hV(x))))^{-1}\bg\circ \bF^{-n}\circ\bphi_{s_1}\circ\bF^n(x, \hV(x))\\
&=\int_0^{\tau_n(x,t)} ds_1(\cL_{\bF}^n \bg)(\bphi_{s_1}(F^n(x),\hV(F^n(x))))=\bH_{F^n x,\tau_n(x,t)}(\cL_{\bF}^n\bg) .
\end{split}
\]
\end{proof}

As already mentioned, the basic fact about the operator $\cL_{\bF}$ is that there exists Banach spaces $\cB^{p,q}$,\footnote{ These are more general with respect to the previously mentioned ones. We use the same name to simplify notation and since no confusion can arise.} detailed in Appendix \ref{sec:norms}, to which $\cL_\bF$ can be continuously extended.\footnote{ By a slight abuse of notations we will still call $\cL_\bF$ such an extension.} Moreover, in  Appendix \ref{sec:norms} we prove the following result.
\begin{prop}\label{th:base}
Let $F \in \cC^r(\bT^2,\bT^2)$ be an Anosov map. Let $p \in \bN^*$ and $q \in \bR$ such that 
 $p + q  \leq  r$ and  $q > 0 $. 
Let $\rho= \exp(\htop)$ where $\htop$ is the topological entropy of $F$.
Then the spectral radius of $\cL_{\bF}$ on $\cB^{p,q}$ is $\rho$ and its essential spectral radius is at most $\rho \lambda^{-\min\{p,q\}}$. In addition, $\rho$ is a simple eigenvalue of $\cL_{\bF}$ and all the other eigenvalues are strictly smaller in norm.\footnote{ The conditions on $ p, q$ are not optimal. The lack of optimality begin due to the fact that we require $p\in \bN$. See \cite{Baladi-Tsujii08}, and reference therein, for different approaches that remove such a constraint.}
 \end{prop}

This last result has finally made precise our original naive idea: now the operator $\cL_\bF$ has a nice spectral picture and Lemma \ref{lem:adjust} shows that $H_{x,t}(g)$ can be written in terms of similar functionals acting on $\cL_\bF^n\bg$. Yet, a last difficulty appears: the $\bH_{x,s}(\cdot)$, $s\leq \Const$, although uniformly bounded as functionals on $\cC^0$, are not uniformly bounded on $\cB^{p,q}$, in fact when acting on $\cB^{p,q}$, for $p>0$, they are not even continuous functionals\,!\footnote{ This is due to the sharp cut-off of the test function at zero and $s$, see Lemma \ref{lem:boundedfunctional}.} This last obstacle can be dealt with by a more sophisticated representation of $\bH_{x,t}(\cdot)$ in terms of uniformly bounded elements of $(\cB^{p,q})'$ plus a measure with total variation uniformly bounded in $x, t$. Such a representation is achieved in Lemma \ref{eq:molli} which provides the last ingredient needed to close the argument.

Before being able to state precisely our first result we need another little bit of notation. Let $\{\boO_{i,j}\}_{j=1}^{\bd_i}\subset (\cB^{p,q})'$ be the elements of a base of the eigenspaces associated to the discrete eigenvalues $\{\rho_i\}_{i\geq 1}$, $|\rho_i|> \exp(\htop)\lambda^{-\min\{p,q\}}$, of $\cL_\bF'$ when acting on $(\cB^{p,q})'$, $p+q\leq r-1$.\footnote{ Remark that the compact pat of the spectrum of $\cL_{\bF}$ and $\cL_{\bF}'$ coincide (see \cite[Remark 6.23]{Kato}).} Since $\cC^{p+q}\subset \cB^{p,q}$, we have $(\cB^{p,q})'\subset(\cC^r)'$. Hence $\{\boO_{i,j}\}\subset \cC^r(\Omega,\bC)'$. We then define $\tilde O_{i,j}=\bpi_*\boO_{i,j}$, clearly $\tilde O_{i,j}\in\cC^r(\bT^2,\bC)'$. Note that $\bpi_*$ is far from being invertible, so many different distributions could be mapped to the same one. Thus the dimension $d_j$ of the span of $\{\tilde O_{i,j}\}_{j=1}^{\bd_i}$ will be, in general, smaller than $\bd_j$ (see Section \ref{rem:Fouriercomputation} for an explicit example). Let us relabel a subset of the $\tilde O_{i,j}$ so that the $\{\tilde O_{i,j}\}_{j=1}^{d_i}$ are all linearly independent and set $D_k=\sum_{i\leq k}d_i$. For convenience, let us relabel our distributions $\{O_i\}$, by $O_i=\tilde O_{k,l}$ for $i\in [D_k+1, D_{k+1}]$ and $l=i-D_k$.
\begin{thm}\label{thm:main}
Provided $r$ is large enough,\footnote{ For example, $e^{\htop}\lambda^{-r/2}< 1$ suffices. We refrain from giving a more precise characterization of the minimal $r$ since, in the present context, it is not very relevant.} there exists $N_1$ such that the distributions ({\em obstructions}) $\{O_i\}_{i=1}^{N_1}\subset \cC^r(\bT^2,\bC)'$ have the following properties. For each $i\leq N_1$, let  $\bV_i=\{g\in\cC^r(\bT^2,\bC)\;:\;O_j(g)=0\; \forall j< i\;; O_i(g)\neq 0\}$. Then there exists $C,\delta>0$ such that, for all $g\in \bV_i$, there exists functions $\hat\ell_{k,j}\in L^\infty(\bT^2\times \bRp)$ such that, for all $t\in\bR_>$ and $x\in\bT^2$, we have
\[ 
\left|  H_{x,t}(g) - t^{\alpha_k}\sum_{j=0}^{b_k-i+D_{k-1}}(\ln t)^{j} \hat\ell_{k,j}(x,t)\right| \leq\begin{cases}& C\, t^{\alpha_k-\delta} \|g\|_{\cC^r}\textrm{ if } \alpha_k > 0 \\ 
&C\, \|g\|_{\cC^r} \textrm{ if } \alpha_k =  0, 
\end{cases}
\]
where $i\in (D_{k-1},D_k]$, $\alpha_k=\frac{\ln|\rho_k|}{\htop}$ and $b_k= d_k$ if $\alpha_k>0$ and $b_k=d_k+1$ if $\alpha_k=0$. Also $\alpha_1=1$, $b_1=0$ and $\alpha_{N_1}=b_{N_1}=0$.
\end{thm}
The above Theorem will be proven in Section \ref{sec:growth}. 
\begin{rem} Note that in Theorem \ref{thm:main} it could happen $N_1=1$, that is: either the integral grows like $t$ or it is bounded. This is indeed the situation, for example, in the linear case (see Section \ref{rem:Fouriercomputation}) and hence for small smooth perturbations of the linear case as well (see Section \ref{sec:non-lin-ex}). In such an event the result might seem less interesting, nonetheless it provides a relevant information.
\end{rem}
\begin{rem} A natural question is how to obtain a more explicit identification of the above mentioned distributions. In particular, the analogy with the situations studied by Flaminio-Forni would suggest $(\phi_t)_*O_j=O_j$, that is the distributions are invariant for the flow. However, note that in the present context, for $j>1$, we know only that $O_j\in(\cC^s)'$, for some $s\geq 2$, while in general $\phi_t\not\in\cC^2$, so $(\phi_t)_*O_j$ is, in principle, not even defined. Nevertheless, we expect some $O_j$ to be invariant distribution for the flow, but the proof is not so obvious. We therefore limit ourselves, in the general case, to discussing $O_1$, that is known to be a measure (see however Lemma \ref{lem:fuck} for a more in depth discussion of a specific class of examples).
\end{rem}
\begin{lem}\label{lem:invar-mea} The distribution $O_1$ is proportional to the unique invariant measure $\mu$ of $\phi_t$.
\end{lem}
\begin{proof}
By Proposition \ref{th:base} it follows that, for all $\bg\geq0$, 
\[
0\leq\lim_{n\to \infty}\rho^{-n}\cL_\bF^n(\bg)=\bh_1\boO_1(\bg),
\]
where $\bh_1$ and $\boO_1$ are the right and left eigendistributions of $\cL_\bF$ associated to the eigenvalue $\rho$, respectively.
Accordingly, $\boO_1$ is a positive distribution, and hence a measure, thus also $O_1$ is a measure.
By the ergodic theorem $H_{x,t}(g)$ grows proportional to $t$ unless $g\in\bV_0=\{g\;:\;\mu(g)=0\}$. By Theorem \ref{thm:main} it follows that $\operatorname{Ker}(O_1)\subset\bV_0$. On the other hand the kernel of $O_1$ must be a codimension one closed subspace, hence  $\operatorname{Ker}(O_1)=\bV_0$. It follows that the two measures must be proportional.\footnote{ Remark that this implies that $O_1$ is invariant for the flow $\phi_t$. In fact, by using  judiciously \eqref{eq:commute2} one could have proven directly that $\boO_1$ is invariant for $\bphi_t$. It is possible that such a proof would work also for eigendistributions with eigenvalues with modulus sufficiently close to one. Yet, for smaller eigenvalues the aforementioned regularity problems seem to kick in.} 
\end{proof}
The next step is to study, in the case $O_i(g)=0$ for all $i\in\{1,\dots,N_1\}$, the regularity of the coboundary. As already mentioned (see Remark \ref{rem:low-reg}) it is natural to consider only $r_i\leq 1+\alpha$. To study exactly the H\"older regularity would entail either to use a more complex Banach space or an interpolating argument.  In the spirit of giving ideas rather than a complete theory, we content ourselves with considering Lipschitz regularity. To do so we have only to consider the derivative of $\oH_{T}$ with respect to $x$. To study the growth of such a derivative, it is necessary to introduce new adapted transfer operators $\cL_{\hat A,\bF}$ and  $\widehat{\cL}_\bF$ the second of which is now defined on one forms (see equation \eqref{eq:new-new} for the precise definitions) and acts on different Banach  spaces $ \widehat \cB^{p,q}$ (see Appendix \ref{sec:norms-bis}).
The Banach spaces $\widehat \cB^{p,q}$ are a bit more complex than the $\cB^{p,q}$ used in Theorem \ref{thm:main} insofar they are really spaces of currents rather then distributions (one has to think of $dg$, rather than $g$, as an element of the Banach space).  Apart from this, the proof of our next result, to be found in Section \ref{sec:cob}, follows the same logic of the first proof.

\begin{thm}\label{thm:maintwo}
Provided $r$ is large enough,\footnote{ Here $r$ needs to be much larger than in the previous Theorem. A precise estimate is implicit in the proof, but the reader may be better off assuming $F\in\cC^\infty$ and not worrying about this issue.} there exist distributions (that we often call {\em obstructions}) $\{O_i\}_{i=N_1+1}^{N_2}\subset \cC^r(\bT^2,\bR)'$ such that if $O_i(g)=0$ for all $i\in\{1,\dots,N_2\}$, then $g$ is a Lipschitz coboundary. More precisely, for appropriate $p,q$, $p+q\leq r-2$, there exists a potential $\hAcc$, an operator $\cL_{\bF,\hAcc}$ acting on $\cB^{p,q}$ and a Banach spaces $\widehat \cB^{p,q}$ with a transfer operator $\widehat \cL_{\bF}$ (depending on action of the map $\bF$ on one forms) acting on it,\footnote{ See \eqref{eq:new-new} for the exact definition of such operators.} such that the distributions $\{O_i\}_{i=N_1+1}^{N_2}$ are described in terms of a base of the eigenspaces associated to the discrete eigenvalues of the operators $\cL_{\bF}, \cL_{\bF,\hAcc}$ and $\widehat \cL_\bF$. 
\end{thm}
\begin{rem} Note that, in principle, $g$ could be a Lipschitz coboundary even if it is not in the kernel of the distributions $\{O_i\}_{i=1}^{N_2}$. Indeed, the Theorem provides only sufficient conditions. However, we believe the conditions to be generically also necessary, but to prove this quite some more work seems necessary. We limit ourselves to discussing such issue in a class of examples (see Lemma \ref{thm:lip}).
\end{rem}
The next sections of the paper are devoted to the proof of the above claims.

Last we would like to conclude this section with the following considerations.

\begin{conj}
 The natural analogues of Theorems \ref{thm:main} and \ref{thm:maintwo} hold in the case of the horocycle flow on a surface of variable strictly negative curvature, where the renormalizing dynamics is the geodesic flow, with the only modification of having an infinite countable family of obstructions. 
\end{conj}

\begin{rem}\label{rem:horo} The difference between finitely many and countably many obstructions comes from the different spectrum of the transfer operators for maps and flows. In the former, the discrete spectrum is always finite. In the latter, one has a one parameter families of operators and it is then more natural to look at  the spectrum of the generator. It turns out that such a spectrum is discrete on the right of a vertical line whose location depends on the flow regularity. Yet, it can have countably many eigenvalues (as the laplacian on hyperbolic surfaces), hence the countably many obstructions (see \cite{ButterleyLiverani07, ButterleyLiverani13, Faure-Tsujii13} for more details).
\end{rem}

\section{Growth of the ergodic average}\label{sec:growth}
As already explained, there is one further, and luckily last, conceptual obstacle preventing the naive implementation our strategy: the functionals $\bH_{x,t}$ are, in general, not continuous (let alone uniformly continuous with respect to $(x,t)$) on the spaces $\cB^{p,q}$ that are detailed in Appendix \ref{sec:norms}.  That is, they do not belong to $(\cB^{p,q})'$, for $p\neq 0$.\footnote{ The problem comes from the boundary in the domain of the integral defining them. There, in some sense, the integrand jumps to zero and cannot be considered smooth in any effective manner.}
In fact, it is possible to introduce different Banach spaces on which the transfer operator is quasi-compact and the functionals $\bH_{x,t}$ are continuous (this are spaces developed to handle piecewise smooth dynamics such as \cite{DemersLiverani, BaladiGouezel10, DemersZhang, BaladiLiverani11}) but the essential spectral radius of our transfer operators on such spaces would always be rather large. Hence we would be able to obtain in this way, at best, only the very firsts among the relevant distributions we are seeking, whereby nullifying the appeal of our approach.

Before providing the proof of Theorem \ref{thm:main} we must thus circumvent such a problem. To this end we introduce, for each $x\in\bT^2$ and $\vf\in L^\infty(\bRp,\bR)$, the new ``mollified" functional
\begin{equation}\label{eq:molli}
\bH_{x,\vf}(\bg)=\int_{\bR}\vf(t)\cdot \bg\circ\bphi_t(x, \hV(x))\, dt.
\end{equation}
It is proven in Appendix \ref{sec:norms} that $\bH_{x,\vf}\in (\cB^{p,q})'$ provided $\vf\in \cC_0^{p+q}(\bR,\bR)$.

\subsection{Proof of our first main result}\ \newline\label{subsec:proof-one}
Our key claim is that the functionals \eqref{eq:molli} suffice for our purposes. To be more precise let us fix $t>0$ and define the sets
$\cD^s_{r,C}=\{\vf\in \cC^r([0,t],\bRpe)\;:\;\|\vf\|_{\cC^r}\leq C\}$ and $\cD_{r,C}=\{\vf\in \cC_0^r([0,t],\bRp)\;:\;\|\vf\|_{\cC^r}\leq C\}$, note that such sets are locally compact in $\cC^{r-1}([0,t],\bR)$ and $\cC_0^{r-1}([0,t],\bR)$, respectively.\footnote{ Up to now the exact definition of the $\cC^r$ norms was irrelevant, now instead it does matter. We make the choice $\|\vf\|_{\cC^r}=\sum_{k=0}^r 2^{r-k}\|\vf^{(k)}\|_{L^\infty}$. It is well known that  with such a norm $\cC^r$ is a Banach algebra. Also, as usual, for a $\cC^r$ function on a closed set, we mean that there exists an extension on some larger open set.}
\begin{lem}\label{lem:suffice}
There exists $C_*>0$ such that, for each $n\in\bN$, $t\in\bRp$, $x\in\bT^2$ and $\bg\in\cC^{r-1}(\Omega,\bR)$, there exists $K\in \bN$, $\{n_i^\pm\}_{i=1}^K\subset \bN$, $n_K^\pm=0$, $\Const>n_i^\pm- n_{i+1}^\pm\geq 0$, $n_i^-+n_i^+>n_{i+1}^-+n_{i+1}^+$, and $\{\vf_i^\pm\}_{i=1}^K\subset \cD_{r,C_*}$, $\{\vf^\pm\}\subset \cD^s_{r,C_*}$ such that
\[
\bH_{x,t}(\bg)=\sum_{\sigma\in\{+,-\}} \left( \sum_{i=1}^K\bH_{F^{n_i^\sigma}(x), \vf_i^\sigma}(\cL_{\bF}^{n_i^\sigma}\bg)+\bH_{x,\vf^\sigma}(\bg) \right).
\]
Moreover,  $\max\{|\supp \vf^\pm|,|\supp\vf^\pm_i|\}\leq 1$. \newline
Finally $\vf_1^-=\vf_1^+$ and $n_1^\pm= n_t$ where $n_t=\inf\{n\in\bN\;:\; \tau_n(x,t)< 1\}$ satisfies the bounds
\begin{equation}\label{eq:nt-bound}
\frac{\ln t}{\htop}-\Const\leq n_t\leq \frac{\ln t}{\htop}+\Const.
\end{equation}
\end{lem}

Before proving Lemma \ref{lem:suffice}, let us use it and prove our first main result.
\begin{proof}[{\bf Proof of Theorem \ref{thm:main}}]
By Proposition \ref{th:base} we have
\begin{equation} \label{eq:quasicompactness}
\cL_{\bF}= \sum_{j=0}^m\left(  \rho_j \Pi_j +Q_j\right) + R_{p,q}
\end{equation}
where $m$ is a finite number, $\rho_j$, $|\rho_{j+1}|\leq |\rho_j|\leq e^{\htop}$, are complex eigenvalues of $\cL_{\bF}$, $\Pi_j$ are finite rank projectors, $Q_j$ are  nilpotent operators. That is, there exists $\{d_j\}_{j=1}^m$  such that $Q_j^{d_j}=0$ and, if $d_j>1$, then $Q^{d_j-1}\neq 0$. Finally, $R_{p,q}$ is a linear operator with spectral radius at most $ e^{\betas}$ where $ e^{\betas}  =\lambda^{-\min(p,q)} e^{\htop}$. In addition,
$\Pi_j  R_{p,q} = R_{p,q} \Pi_j =Q_j  R_{p,q} = R_{p,q} Q_j =  0 $. Moreover, for each $i\neq j$, $[\Pi_i, \Pi_j] =[\Pi_i,Q_j]=[Q_i, Q_j] = 0$ and $\Pi_i^2=\Pi_i$, $\Pi_iQ_i=Q_i\Pi_i=Q_i$.  In other words the operator $\cL_\bF$ is quasi compact and it has a spectral decomposition in  Jordan Block of size $d_j$  plus a non compact part of small spectral radius. Note as well that $d_1=1$, $Q_1=0$ and $\Pi_1$ is a one dimensional projection corresponding to the eigenvalue $e^{\htop}$ which is the only eigenvalue of modulus $e^{\htop}$. Finally, set 
\[
\alpha_j=\frac{\ln|\rho_j|}{\htop}\;;\quad \widetilde N_1=\sum_{j=1}^m d_j.
\]
If $r$ is large enough, we can choose $p,q$, $m$ and $\ve$ such that $|\rho_j|\geq 1$ for all $j\leq m$, $\betas<0$  and $e^{\betas}+\ve<1$, hence $\sup_n \|R_{p,q}^n\|_{p,q}\leq\Const$. 
Then, setting $\bg=g\circ\bpi$, Lemmata \ref{lem:adjust}, \ref{lem:suffice} and \ref{lem:boundedfunctional}, together with the spectral decomposition \eqref{eq:quasicompactness}, imply
\[
\begin{split}
\left|H_{x,t}(g)-\sum_{j=0}^m \sum_{\sigma\in\{+,-\}}  \sum_{i=1}^K\bH_{F^{n_i^\sigma}(x), \vf_i^\sigma}((\rho_j\Pi_j+Q_j)^{n_i^\sigma} \bg)\right|&\leq \Const\|g\|_{L^\infty}+\Const \|R_{p,q}^n \bg\|_{p,q}\\
&\leq\Const \|\bg\|_{p,q}.
\end{split} 
\]
On the other hand, setting
\[
\ell_j(x,t,g)=\begin{cases}\rho_j^{-n_t}n_t^{-d_j+1}\sum_{\sigma\in\{+,-\}}  \sum_{i=1}^K\bH_{F^{n_i^\sigma}(x), \vf_i^\sigma}((\rho_j\Pi_j+Q_j)^{n_i^\sigma} \bg)\;&\textrm{if } |\rho_j|>1\\
n_t^{-d_j}\sum_{\sigma\in\{+,-\}}  \sum_{i=1}^K\bH_{F^{n_i^\sigma}(x), \vf_i^\sigma}((\rho_j\Pi_j+Q_j)^{n_i^\sigma} \bg)\;&\textrm{ if } |\rho_j|=1,
\end{cases}
\]
we have, in the first case,\footnote{ Note that the $n_i^\pm$ in Lemma \ref{lem:suffice} cannot be more than $n_t$, hence $K\leq n_t$.}
\[
\begin{split}
\left|\ell_j(x,t,g)\right|&\leq \Const \sum_{n=1}^{n_t}\rho_j^{-n_t}n_t^{-d_j}\|(\rho_j\Pi_j+Q_j)^n \bg\|_{p,q}\\
&\leq\Const \sum_{n=1}^{n_t}\rho_j^{-n_t}n_t^{-d_j}\rho_j^{n}n^{d_j}\|\bg\|_{p,q}\leq \Const\|\bg\|_{p,q}
\end{split}
\]
and the same estimate holds in the second case. Note that the function $\ell_j$ have a natural decomposition $\ell_j=\sum_{k=0}^{d_j-1}n_t^{k}\rho_j^{n_t}\ell_{j,k}$. Collecting the above yields
\begin{equation}\label{eq:parte-unobis}
\left|H_{x,t}(g)-\sum_{j=0}^m \rho_j^{n_t}\sum_{k=0}^{d_j-1}n_t^{k}\ell_{j,k}(x,t,g)\right|\leq \Const \|\bg\|_{p,q}.
\end{equation} 
To conclude note that $\Pi_{j}=\sum_{i=1}^{d_j}h_{j,i}\otimes \boO_{j,i}$ with $h_{j,i}\in\cB^{p,q}$ and $\boO_{j,i}\in (\cB^{p,q})'\subset (\cC^r(\Omega,\bR)'$. Finally, since $\bpi^*:\cC^r(\bT^2,\bR)\subset\cC^r(\Omega,\bR)$, we have that $\tilde O_{j,i}:=\bpi_*\boO_{i,j}\in (\cC^r(\bT^2,\bR))'$, and $\boO_{j,i}(\bg)=\tilde O_{j,i}(g)$. Note that it might happen $\bpi_*\boO_{j,i}=\bpi_*\boO_{j',i'}$ or $\bpi_*\boO_{j,i}=0$ (see Section \ref{rem:Fouriercomputation}). Let $N_1$ be the cardinality of the set $\{\bpi_*\boO_{j,i}\}$. Then, by construction, if $g\in \bV_{i}$, then $\ell_j(x,t,g)\equiv 0$ for all $j$ such that $i\leq D_{j-1}$. Hence the Theorem follows.
\end{proof} 

\subsection{Decomposition in proper functionals}\ \newline
This section is devoted to showing that the functionals $H_{x,t}$ can be written in terms of well behaved functionals plus a bounded error.
\begin{proof}[{\bf Proof of Lemma \ref{lem:suffice}}]
Fix $x\in\bT^2$ and $t\in\bRp$. By definition $\tau_{n_t}(x,t)\in (\Lambda^{-1}, 1)$ for some fixed $\Lambda>1$.

Let $\delta\in(0,\Lambda^{-1}/4)$ small and $C_*>0$ large enough to be fixed later.
We can now fix $n_1=n_t$. Note that the claimed bound on $n_t$ follows directly by \cite[Lemma C.3]{GLP13}. Next, chose $\psi\in\cD_{r,C_*/2}$ such that $\supp\psi\subset (\delta,\tau_{n_1}-\delta)$, $\psi|_{[2\delta,\tau_{n_1}-2\delta]}=1$. Set $\psi^-=(1-\psi)\Id_{[0,\tau_{n_1}/2]}$, $\psi^+=(1-\psi)\Id_{[\tau_{n_1}/2, \tau_{n_1}]}$. Then $\psi^\pm\in \cD^s_{r,C_*}$ and we can use Lemma \ref{lem:adjust} to write
\[
\begin{split}
\bH_{x,t}(\bg)&=\bH_{F^{n_1}(x),\tau_{n_1}(x,t)}(\cL_{\bF}^{n_1}\bg)\\
&=\bH_{F^{n_1}(x), \psi^-}(\cL_{\bF}^{n_1}\bg)+\bH_{F^{n_1}(x), \psi}(\cL_{\bF}^{n_1}\bg)+\bH_{F^{n_1}(x), \psi^+}(\cL_{\bF}^{n_1}\bg).
\end{split}
\]
We are happy with the middle term which, by Lemma \ref{lem:boundedfunctional}, is a continuous functional of $\cL_{\bF}^{n_1}\bg$, not so the other two terms. We have thus to take care of them.
A computation analogous to the one done in Lemma \ref{lem:adjust} yields, for each $n\in\bN$,
\begin{equation}\label{eq:change-time}
\bH_{x,\vf\circ \tau_n(x, \cdot)}(\bg)=\bH_{F^n(x),\vf}(\cL_{\bF}^n\bg).
\end{equation}
We will use the above to prove inductively the formula
\begin{equation}\label{eq:H-ind}
\begin{split}
\bH_{x,t}(\bg)=&\bH_{F^{n^-_k}(x), \psi_k^-}(\cL_{\bF}^{n^-_k}\bg)+\bH_{F^{n^+_k}(x), \psi_k^+}(\cL_{\bF}^{n^+_k}\bg)\\
&+\sum_{\sigma\in\{+,-\}}\sum_{i=1}^k\bH_{F^{n^\sigma_i}(x), \vf^\sigma_i}(\cL_{\bF}^{n^\sigma_i}\bg)
\end{split}
\end{equation}
where $\psi^\pm_1=\psi^\pm$, $\vf^{\pm}_1=\frac 12\psi$, $n_1^\pm=n_1$, $\psi^\pm_k\in\cD^s_{r,C_*}$, $\{\vf_i^\pm\}\subset \cD_{r,C_*}$, $\|\psi^\pm_k\|_{L^\infty}\leq 1$, $\|\vf^\pm_k\|_{L^\infty}\leq 1$, $(b_k^\pm,b_k^\pm\mp\delta)\subset\supp{\psi^\pm_k}\subset(b_k^\pm,b_k^\pm\mp2\delta)$,  $\supp{\vf_k^\pm}\subset (b_k^\pm,b_k^\pm\mp 1)$, $b_k^-=0$ and $b_k^+\in[0, \Lambda^{n^+_k}]$, $b_1^+=t$.

Let us consider the first term on the right hand side of the first line of \eqref{eq:H-ind} (the second one can be treated in total analogy). Let $\supp(\psi_k^-)=[0,a_k]$ and define $m+1=\inf\{n\in\bN\;:\; \tau_n^{-1}(F^{n_k^-}(x),a_k)\geq 1\}$. Note that, by construction, there exists a fixed $\overline m\in\bN$ such that $m\geq \overline m$, also $\overline m$ can be made large by choosing $\delta$ small. Hence $\widehat{\psi}^-_{k}(s)=\psi_k^-\circ \tau_{m}(F^{n_k^-}(x),s)$ is supported in the interval $[0,1)$ and the support contains $[0,\Lambda^{-1}]$. 

Next, we need an estimate on  the norm of $\widehat{\psi}^-_{k}$. We state it in a sub-lemma so the reader can easily choose to skip the, direct but rather tedious, proof.
\begin{sublem}\label{sublem:sigh}
Provided we choose $\delta$ small and $C_*$ large enough, we have 
\[
\widehat{\psi}^-_{k}\in\cD_{r,C_*/2},
\] 
where $n_{k+1}^-=n_k^--m$.
\end{sublem}
\begin{proof}
First of all $\|\psi^-_k\|_{L^\infty}\leq 1$, and\footnote{ Here we use the formula $\|f\circ g\|_{\cC^r}\leq\sum_{k=0}^r 2^{r-k}\|f\|_{\cC^k}\|Dg\|_{\cC^{r-1}}\|Dg\|_{\cC^{r-2}}\cdots\|Dg\|_{\cC^{r-k}}$, that can be verified by induction.}
\begin{equation}\label{eq:hoihoi}
\begin{split}
\|\widehat{\psi}^-_{k}\|_{\cC^r}&\leq \sum_{j=0}^r 2^{r-j}\|\psi^-_k\|_{\cC^j}\|\tilde\nu_{z_k,m}\|_{\cC^{r-1}}\|\tilde\nu_{z_k,m}\|_{\cC^{r-2}}\cdots\|\tilde\nu_{z_k,m}\|_{\cC^{r-j}}\\
&\leq 2^r+C_*\sum_{j=1}^r \|\tilde\nu_{z_k,m}\|_{\cC^{r-1}}\|\tilde\nu_{z_k,m}\|_{\cC^{r-2}}\cdots\|\tilde\nu_{z_k,m}\|_{\cC^{r-j}}
\end{split}
\end{equation}
where $z_k=F^{n_k^-}(x)$ and, for each $j\in\bN$ and $z\in\bT^2$, $\tilde \nu_{z,m}(s)=\nu_m(\phi_s(z))$, where $\nu_m$ is defined in \eqref{eq:defvn}. Note that, although $\nu_m$ is, in general, only $\cC^{1+\alpha}$, by \eqref{eq:smooth-hp} it follows that the map $s \to \tilde \nu_{z,m}(s)\in\cC^{r-1}$ and hence, for all $z\in\bT^2$, $\langle V,\nabla \tilde \nu_{z}\rangle\circ \phi_{(\cdot)}\in\cC^{r-2}$.
We can thus continue and compute
\[
\begin{split}
\frac{d}{ds}\tilde \nu_{z_k,m}(s)&=\frac{d}{ds}\prod_{j=0}^{m-1}\nu_1(F^j(\phi_s( z_k)))\\
&=\tilde \nu_{z_k,m}(s)\sum_{l=0}^{m-1} \frac{\langle \nabla \nu_1(F^l\circ \phi_s(z_k)), V(F^l\circ\phi_s(z_k))\rangle}{\nu_1(F^l\circ\phi_s(z_k))} \tilde \nu_{z_k,l}(s)\\
&=\tilde \nu_{z_k,m}(s)\sum_{l=0}^{m-1} \left[\frac{\langle V, \nabla \nu_1\rangle}{\nu_1}\right]\circ\phi_{\tau_l(z_k,s)}(F^l(z_k)) \tilde \nu_{z_k,l}(s).
\end{split}
\]
The above, by induction, implies that there exist increasing constants $A_q\geq 1$ such that $\|\tilde\nu_{z_k,m}\|_{\cC^{q}}\leq A_q\|\tilde\nu_{z_k,m}\|_{\cC^0}$. Indeed,  
$[\frac{\langle V, \nabla \nu_1\rangle}{\nu_1}]\circ \phi_{\cdot}\in\cC^{r-2}$, and
\[
\begin{split}
\left\|\left[\frac{\langle V, \nabla \nu_1\rangle}{\nu_1}\right]\circ\phi_{\tau_l(z_k,\cdot)}(F^l(z_k))\right\|_{\cC^q}&\leq \sum_{i=0}^q2^{q-i}\left\|\left[\frac{\langle V, \nabla \nu_1\rangle}{\nu_1}\right]\circ\phi_{\cdot}(F^l(z))\right\|_{\cC^q}\|\tilde\nu_{F^l(z_k),l}\|_{\cC^{q-1}}^i\\
&\leq \Const \sum_{i=0}^q A_{q-1}^{i}\lambda^{-i l}\leq \Const A_{q-1}^q.
\end{split}
\]
Thus,
\[
\begin{split}
\|\tilde\nu_{z_k,m}\|_{\cC^{q+1}}&=2^q \|\tilde\nu_{z_k,m}\|_{\cC^0}+\|\frac{d}{ds}\tilde \nu_{z_k,m}\|_{\cC^{q}}\\
&\leq 2^q \|\tilde\nu_{z_k,m}\|_{\cC^0}+\|\tilde \nu_{z_k,m}\|_{\cC^{q}}\sum_{l=0}^{m-1}\Const A_{q-1}^q A_{q-1}\lambda^{-l}\\
&\leq \left[2^q+A_q \Const A_{q-1}^{q+1}\right] \|\tilde\nu_{z_k,m}\|_{\cC^0}=:A_{q+1} \|\tilde\nu_{z_k,m}\|_{\cC^0}.
\end{split}
\]
We did not try to optimize the above computation since the only relevant point is that the $A_q$ do not depend on $m$.
Accordingly, if we choose $\delta$ small (and hence $m$ large) enough, we have $\|\tilde\nu_{z_k,m}\|_{\cC^{q}}\leq \frac 14$ for all $q\leq r-1$.
Using the above fact in \eqref{eq:hoihoi} yields
\[
\|\widehat{\psi}^-_{k}\|_{\cC^r} \leq 2^r+C_*\sum_{j=1}^r 4^{-j}=2^r+\frac 13 C_*
\]
which implies the Lemma provided we choose $C_*$ large.
\end{proof}
By \eqref{eq:change-time}, we have
\[
\bH_{F^{n_k^-}(x), \psi_k^-}(\cL_{\bF}^{n_k^-}\bg)=\bH_{F^{n_{k+1}^-}(x), \widehat{\psi}_{k}^-}(\cL_{\bF}^{n_{k+1}^-}\bg). 
\]
Again we can write  $\psi^-_{k+1}=(1-\psi)\widehat{\psi}_{k}^{-}\Id_{[0,2\delta]}$ and $\vf_{k+1}^-=\widehat{\psi}_{k}^{-}-\psi^-_{k+1}$. Then,
\[
\sup\{\|\psi^-_{k+1}\|_{\cC^{r}([0,1],\bR)},\|\vf_{k+1}^-\|_{\cC_0^r(\bR,\bR)}\} \leq C_*
\]
and $[0,\delta]\subset\supp\psi^-_{k+1}\subset [0,2\delta]$. Accordingly
\[
\bH_{F^{n_k^-}(x), \psi_k^-}(\cL_{\bF}^{n_k^-}\bg)=\bH_{F^{n_{k+1}^-}(x), \psi_{k+1}^-}(\cL_{\bF}^{n_{k+1}^-}\bg)+\bH_{F^{n_{k+1}^-}(x), \vf_{k+1}^-}(\cL_{\bF}^{n_{k+1}^-}\bg).
\]
The Lemma is thus proven by taking $k=K$, so that $n^\pm_K=0$.\footnote{ If more steps are needed on one side, say the plus side, one can simply set $n^-_{k+1}=n^-_k$, $\psi^-_{k+1}=\psi^-_{k}$ and $\vf^-_k=0$ for all the extra steps.}
\end{proof}
%%%%%%%%%%%%%%%%%%%%%%%%%%%%%%%%%%%%%
\section{Coboundary regularity}\label{sec:cob}
We first prove several claims stated in the introduction and set up some notation. Then we prove our main results concerning coboundary regularity.
\subsection{Parabolic}\ \newline\label{subsec:parabolic}
In the introduction we called the flow $\phi_t$ {\em parabolic}, but no evidence was provided for this name. It is now time to substantiate such an assertion.
\begin{rem}\label{rem:oliver} Note that the following Lemma shows only that the differential cannot grow more than polynomially, yet the possibility remains open that it does not grow at all, as in the linear model (or when the map is Lipschitz conjugated to the linear model). In such a case the flows should be more properly called {\em elliptic}. This is not always the case, as one can see in an explicit class of examples worked out in Section \ref{sec:non-lin-ex} (see Lemma \ref{lem:true-parabolic}).
\begin{lem}\label{lem:parabolic}
There exists $C,\beta>0$ such that, for all $x\in \bT^2$ and $t\in\bR$, letting $\xi(s)=D_x\phi_{s}$, we have
\[
\|\xi\|_{\cC^{r-1}((0,t),GL(2,\bR))}\leq C |t|^\beta.
\]
\end{lem}
\begin{proof}
Since $\phi_{-t}=\phi^-_t$, where $\phi^-_t$ is the flow generated by $-V$, and the following argument is insensitive to orientation, it suffices to consider the case $t\geq 0$.
It turns out to be convenient to define $V^\perp (x)$ as the perpendicular vector to $V(x)$ such that $\|V^\perp(x)\|=\|V(x)\|^{-1}$. In this way we can use $\{V(x),V^\perp(x)\}$ as basis of the tangent space at $x$, and the changes of variable are uniformly bounded, with  determinant one and $\cC^r $ in the flow direction. In such coordinates we have
\begin{equation}\label{eq:expmat}
D_x\phi_t=\begin{pmatrix}1&a(x,t)\\0&b(x,t)\end{pmatrix}.
\end{equation}
To have a more precise understanding of the above matrix elements, we have to use the knowledge that the dynamics is renormalizable.
To start with we must differentiate \eqref{eq:commute}:
\begin{equation}\label{eq:a-palle0}
\begin{split}
D_{\phi_t (x)}F^n  \cdot D_x\phi_t &= D_{F^n(x)}\phi_{\tau_{n}(x,t)}\cdot D_xF^n+V(\phi_{\tau_{n}(x,t)}(F^n(x)))\otimes\nabla \tau_n(x,t)\\
&=D_{F^n(x)}\phi_{\tau_{n}(x,t)}\cdot D_xF^n\left[\Id+\nu_n(x)^{-1}V(x)\otimes\nabla \tau_n(x,t)\right],
\end{split}
\end{equation}
where we have used \eqref{eq:defvn} and \eqref{eq:flipvphi}.
Hence, setting 
\begin{equation}\label{eq:a-palle}
A_{x,t,n}=\Id+\nu_n(x)^{-1}V(x)\otimes\nabla \tau_n(x,\tau_n^{-1}(x,t)),
\end{equation}
we have
\begin{equation}\label{eq:renorm-deriv}
\begin{split}
D_x\phi_t =&D_{\phi_{\tau_{n}(x,t)}\circ F^n (x)}F^{-n}\cdot D_{F^n(x)}\phi_{\tau_{n}(x,t)}\cdot D_xF^n \cdot A_{x,\tau_n(x,t),n}.
\end{split}
\end{equation}
Thus, by equations \eqref{eq:expmat} and \eqref{eq:renorm-deriv} we have, for each $n\in\bN$,
\[
b(x,t)=\langle V^\perp(\phi_t(x)), D_{\phi_{\tau_{n}(x,t)}\circ F^n (x)}F^{-n}\cdot D_{F^n(x)}\phi_{\tau_{n}(x,t)}\cdot D_xF^n\cdot A_{x,\tau_n(x,t),n}V^\perp(x)\rangle
\]
\begin{equation}\label{eq:uff}
\hskip-.7cm=\langle V^\perp(\phi_t(x)), D_{\phi_{\tau_{n}(x,t)}\circ F^n (x)}F^{-n}\cdot D_{F^n(x)}\phi_{\tau_{n}(x,t)}\cdot D_xF^nV^\perp(x)\rangle.
\end{equation}
Choose $n$ so that $\tau_n\in [\Lambda^{-1}, 1]$, hence $n$ is proportional to $\ln t$. By compactness it follows that $\|D_{F^n(x)}\phi_{\tau_{n}(x,t)}\|\leq \Const$. Hence, there exists $\beta_0>0$ such that
\[
\sup_{x\in\bT^2}|b(x,t)|\leq \Const t^{\beta_0}.
\]
On the other hand, by the semigroup property, for each $m\in\bN$,
\[
D_x\phi_m=\prod_{i=0}^{m-1}\begin{pmatrix} 1&a(\phi_{i}(x),1)\\ 0&b(\phi_i(x),1)\end{pmatrix}=\begin{pmatrix} 1&\sum_{j=0}^{m-1}a(\phi_j(x),1)b(x,j) \\0&b(x,m)\end{pmatrix}.
\]
Since, again by compactness, $|a(x,1)|\leq \Const$, it follows
\[
|a(x,m)|\leq \Const\sum_{j=0}^{m-1}|b(x,j)|\leq \Const\sum_{j=0}^{m-1}j^{\beta_0}\leq \Const m^{\beta_0+1}.
\]
Hence 
\[
\|\xi\|_{\cC^0((0,t), GL(2,\bR))}\leq \Const t^\beta
\]
 with $\beta=\beta_0+1$.

To estimate the derivatives notice that $\dot \xi(s)=D_{\phi_s(x)}V\xi(s)$. To understand the regularity of the above equation, recall that the stable foliation can be expressed in local coordinates by $(x_1,G(x_1,x_2))$, where  $G(\cdot, x_2) \in \cC^r$, $G(0,x_2)=x_2$, so that $\{(x_1,G(x_1,x_2))\}_{x_1\in\bR}$ is the leaf through the point $x=(x_1,x_2)$, and $(1, \partial_{x_1}G(x ))=V(x)$. It is known that, in such coordinates, $\partial_{x_2} G(\cdot,x_2)\in\cC^{r-1}$ uniformly, see \cite{Hasselblatt97} and references therein. Then, by Schwarz Theorem \cite{Rudin}, if follows that $\partial_{x_2}\partial_{x_1}G(\cdot, x_2)\in \cC^{r-2}$. Hence, $DV\circ \phi_t$ is a $\cC^{r-2}$ function of $t$, with uniformly bounded norm. Accordingly, for each $k\in\{0, \dots,r-2\}$,
\[
\begin{split}
\|\xi\|_{\cC^{k+1}((0,t), GL(2,\bR))}&\leq \|\dot \xi\|_{\cC^{k}((0,t), GL(2,\bR))}+2^{k+1}\|\xi\|_{\cC^{0}((0,t), GL(2,\bR))}\\
&\leq C_k\|\xi\|_{\cC^{k}((0,t), GL(2,\bR))},
\end{split}
\]
from which the Lemma readily follows.
\end{proof}

\end{rem}

\subsection{Some preliminary facts}\ \newline\label{subsec:palle}
In this section we establish some facts and formulae needed in the following.

First of all  recall that given a one form $\omega(x)=\sum_{i=1}^2 a_i(x) dx_i$ and a diffeomorphism $G\in\cC^1(\bT^2,\bT^2)$ the pullback of the form is given by
\begin{equation}\label{eq:pullback}
G^*\omega(x)= a_i(G(x))(D_xG)_{ij}dx_j,
\end{equation}
where we have used the usual convention on the summation of repeated indexes; moreover recall that for a vector field $v$ the pushforward is given by
\[
G_*v(x)=D_{G^{-1}(x)}G\cdot v(G^{-1}(x)). 
\]
Next, we spell out the cocycle properties of $\tau_n$.
\begin{lem}\label{lem:Theta}
For each $x\in\bT^2$, $n,m\in\bN$ we have
\[
\tau_m(F^n(x),\tau_n(x,s))=\tau_{n+m}(x,s)\\
\]
\end{lem}
\begin{proof}
By definition
\[
\begin{split}
\tau_m(F^n(x),\tau_n(x,t))&=\int_0^{\tau_n(x,t)}\nu_m(\phi_s(F^n(x))) ds\\
&=\int_0^t \nu_m(\phi_{\tau_n(x,s)}(F^n(x)))\nu_n(\phi_{s} (x))ds\\
&=\int_0^t \nu_m(F^n\circ \phi_{s}(x))\nu_n(\phi_{s} (x))ds\\
&=\int_0^t \nu_{n+m}(\phi_{s} (x))ds=\tau_{n+m}(x,t),
\end{split}
\]
where we have used \eqref{eq:commute}. 
\end{proof}
By Lemma \ref{lem:Theta}, and using \eqref{def:eraverage}, \eqref{eq:commute}, we can write, for all $n\in\{0,\dots, n_T\}$,\footnote{ We have used the fact that, by definition, $\tau_{n_T}(x,T)\geq 1$ while $\supp\chi\subset [0,1]$.}
\begin{equation}\label{eq:newstart}
\oH_{T}(g)(x)=-\int_{\bRpe}\chi\circ \tau_{n_T-n}(F^n(x),s) \left( \frac{g}{\nu_{n}}\right)\circ F^{-n}\circ\phi_s\circ F^n(x) ds.
\end{equation}
As we are now aware of the fact that the discontinuity of the test function $\chi$ at zero will create problems,\footnote{ The integral will not belong to the dual of the appropriate Banach space.} we take care of the problem right away. Given $\varpi\in (0,1/4)$, small enough, let $T>0$ and $n_*\in\bN$ be large enough and such that $\sup_{z\in\bT^2}\tau_{n_*}(z,1)\leq \varpi$ and $n_T\geq n_*$. Then, we can chose $n=n_T-n_*$ and write 
\[
\begin{split}
\chi\circ \tau_{n_*}(F^{n_T-n_*}(x),s) &=\chi(s)\chi\circ \tau_{n_*}(F^{n_T-n_*}(x),s)+(1-\chi(s))\chi\circ \tau_{n_*}(F^{n_T-n_*}(x),s)\\
&=\chi(s)+(1-\chi(s))\chi\circ \tau_{n_*}(F^{n_T-n_*}(x),s).
\end{split}
\]
Thus, setting $\chi_*(z,s)=(1-\chi(\tau_{n_*}^{-1}(F^{-n_*}(z),s)))\chi(s)$, we can write \eqref{eq:newstart} as
\begin{equation}\label{eq:newstart1}
\begin{split}
\oH_{T}(g)(x)=&-\int_{\bR}\chi_*(F^{n_T}(x),s) \left( \frac{g}{\nu_{n_T}}\right)\circ F^{-n_T}\circ\phi_s\circ F^{n_T}(x) ds\\
&-\int_{\bRpe} \chi\circ \tau_{n_T-n_*}(x,s) g\circ \phi_s(x) ds\\
&=:\oH_{T}^{\star}(g)(x)-\int_{\bRpe} \chi\circ \tau_{n_T-n_*}(x,s) g\circ \phi_s(x) ds.
\end{split}
\end{equation}
Note that the first term contains now a smooth compactly supported test function, while the second term is exactly of the same initial form, that is \eqref{def:eraverage}, (apart  from the fact that $n_T$ is replaced by $n_T-n_*$). Hence, it suffices to study $\oH_{T}^{\star}(g)(x)$. Before doing so let us show that $\oH_{T}(g)$ is really the right quantity to consider.

%%%%%%%%%%%%%%%%%%%%%
\subsection{Continuous coboundary}\ \newline \label{sec:measurable-cob}
In Section \ref{subsec:proof-one} we have seen that if $g$ belongs to the kernel of enough distributions $O_i$, then the $H_{x,t}(g)$ are all uniformly bounded. In the introduction we claimed  that this implies that $g$ is a continuous coboundary, now is the time to prove it.
\begin{proof}[{\bf Proof of Lemma \ref{lem:cob-prelim}}]
Setting, as before, $\bg=g\circ\bpi$, by equations \eqref{def:eraverage}, \eqref{eq:molli},  \eqref{eq:change-time} and using repeatedly formula \eqref{eq:newstart1} we have
\[
\begin{split}
\oH_{T}(g)(x)&=-\bH_{x,\chi\circ\tau_{n_T}(x,\cdot)}(\bg)\\
&=-\sum_{l=0}^{K_T}\bH_{F^{n_*l}(x),\chi_*(F^{n_*l}(x),\cdot)}(\cL_{\bF}^{n_*l}\bg)-\int_{\bRpe} \chi(s) g\circ \phi_s(x) ds,
\end{split}
\]
where $K_T n_*=n_T$.
By hypothesis $\cO_i(g)=0$ for all $i\in\{1,\dots, N_1\}$,  hence $\|\cL_{\bF}^n\bg\|_{p,q}\leq C_g \theta^n$ for some $\theta\in (0,1)$ (see the proof of Theorem \ref{thm:main} for details). Thus, by Lemma \ref{lem:boundedfunctional}, it follows that also $|\bH_{F^{n}(x),\chi_*(F^{n_*l}(x),\cdot)}(\cL_{\bF}^{n}\bg)|\leq C_g \theta^n$. Hence, for all $T'\geq T>0$,
\begin{equation}\label{eq:c0conv}
\|\oH_{T}(g)-\oH_{T'}(g)\|_{\cC^0}\leq C_g \theta^{n_{T}},
\end{equation}
which, recalling \eqref{eq:nt-bound}, implies the existence of the limit.

To prove the second statement of the Lemma, observe that, recalling the properties of $\chi$ specified after \eqref{def:eraverage},
\[
\begin{split}
 &\langle V(x) , \nabla \oH_{T}(g)(x) \rangle  = - \int_0^T dt \,\chi\circ \tau_{n_T}(x,t) \langle D_x \phi_t V(x), (\nabla g) \circ \phi_t(x) \rangle \\
&\phantom{=}-  \int_0^T dt \,\chi'\circ \tau_{n_T}(x,t)\left[\int_0^t \langle D_x \phi_s V(x), (\nabla \nu_{n_T}) \circ \phi_s(x) \rangle ds\right] g\circ\phi_t(x) dt\\
&  = -\int_0^T dt \,\chi\circ \tau_{n_T}(x,t) \langle V(\phi_t(x)), (\nabla g) \circ \phi_t(x) \rangle \\
&\phantom{=}-  \int_0^T dt \,\chi'\circ \tau_{n_T}(x,t)\left[\int_0^t \langle V(\phi_s(x)), (\nabla \nu_{n_T}) \circ \phi_s(x) \rangle ds\right] g\circ\phi_t(x) dt
\end{split}
\]
where we have used \eqref{eq:flipvphi} and the notation of the previous section. Hence,
\begin{equation} \label{eq:homology2} 
\begin{split}
 &\langle V(x) , \nabla \oH_{T}(g)(x) \rangle  =-   \int_0^T dt \,\chi\circ \tau_{n_T}(x,t)  \left(\frac{d}{dt} g\circ \phi_t(x)\right)  \\
&\phantom{=}-  \int_0^T dt \,\chi'\circ \tau_{n_T}(x,t)\left[ \nu_{n_T} \circ \phi_t(x) -\nu_{n_T}(x)\right] g\circ\phi_t(x) dt\\
& =-   \int_0^T dt \, \frac{d}{dt}  \left(\chi\circ \tau_{n_T}(x,t)g\circ \phi_t(x)\right) +\nu_{n_T}(x) \int_0^T dt \,\chi'\circ \tau_{n_T}(x,t)  g\circ \phi_t(x)\\
& =    g(x)+\nu_{n_T}(x) \int_0^T dt \,\chi'\circ \tau_{n_T}(x,t)  g\circ \phi_t(x).
\end{split} \end{equation}
 On the other hand
\[
\int_0^T dt \,\chi'\circ \tau_{n_T}(x,t)  g\circ \phi_t(x)=\bH_{F^n(x),\chi'}(\cL_{\bF}^n\bg),
\]
which, by the same argument as before, is uniformly bounded. Integrating \eqref{eq:homology2} along the flow, yields, for all $t\in\bR$,
\[
\begin{split}
&\oH_{T}(g)(\phi_t(x))-\oH_{T}(g)(x)=\int_0^t\frac{d}{ds} \oH_{T}(g)(\phi_s(x)) ds\\
&=\int_0^t\langle V(\phi_s(x)) , \nabla \oH_{T}(g)(\phi_s(x)) \rangle ds
= \int_0^tg\circ \phi_s(x) ds+\cO(\lambda^{-n_T} t).
\end{split}
\]
The Lemma follows remembering \eqref{eq:nt-bound} and taking the limit $T\to\infty$.
\end{proof}
To study the coboundary regularity we will investigate the regularity of $\oH_{t}(g)$. In reality, we will investigate only the first derivative, see Remark \ref{rem:low-reg} for a discussion of this choice.
%%%%%%%%%%%%%%%%%%%%
\subsection{An explicit formula for the coboundary derivative}\ \newline
Our goal here is to establish an explicit formula for the derivative of \eqref{eq:newstart1}.

For each vector field  $\bv\in\cC^0(\bT^2,\bR^2)$, noticing that\footnote{  Given a matrix $A$ we use $A^*$ to designate the transpose.}
\[
\nabla_z \chi_*(z,s)=-\chi(s)\chi'(\tau_{n_*}^{-1}(F^{-n_*}(z),s))\frac{(D_zF^{-n_*})^*\nabla_z \tau_{n_*}(F^{-n_*}(z),\tau_{n_*}^{-1}(F^{-n_*}(z),s))}{(\partial_t\tau_{n_*})(F^{-n_*}(z),\tau_{n_*}^{-1}(F^{-n_*}(z),s))},
\]
and setting, ${\boldsymbol s}_{*}(x,s)=\tau_{n_*}^{-1}(F^{-n_*}(x),s)$,
\begin{equation}\label{eq:vartheta}
\vartheta_{j}(x,s)=-\chi(s)\chi'({\boldsymbol s}_{*}(x,s))\frac{\langle\nabla_x \tau_{n_*}(F^{-n_*}(x),{\boldsymbol s}_{*}(x,s)), D_xF^{j-n_*}\bv\rangle}{\nu_{n_*}\circ\phi_{{\boldsymbol s}_{*}(x,s)}(F^{-n_*}(x))}
\end{equation}
we have that
\[
\begin{split}
 &\langle \bv (x), \nabla \oH_{T}^{\star}(g)(x) \rangle  =- \int_{\bR} ds \,\vartheta_{n_T}(F^{n_T}(x),s) \left(\frac{g}{\nu_{n_T}}\right) \circ F^{-{n_T}}\circ \phi_s\circ F^{n_T}(x)\\
&+ \int_{\bR} ds \,\chi_*(F^{n_T}(x),s)\langle D_{x}( F^{-{n_T}}\circ \phi_s \circ F^{n_T})\bv, \left[\frac{g}{\nu_{n_T}^2}\nabla \nu_{n_T} \right]\circ F^{-{n_T}}\circ \phi_s\circ F^{n_T}(x) \rangle \\
& - \int_{\bR} ds \,\chi_*(F^{n_T}(x),s) \langle D_{x}( F^{-n_T}\circ \phi_s \circ F^{n_T})\bv, \frac{\nabla g}{\nu_{n_T}} \circ F^{-{n_T}}\circ \phi_s\circ F^{n_T}(x) \rangle .
\end{split}
\]
Recall that
\begin{equation}\label{eq:hoi-hoi-vei}
\begin{split}
&\nabla\nu_{n}=\nabla\prod_{j=0}^{n-1}\nu_{1}\circ F^j=\sum_{j=0}^{n-1}\nu_{n}(DF^j)^*\left[\frac{\nabla \nu_{1}}{\nu_{1}}\circ F^j\right].
\end{split}
\end{equation}
Setting (with a mild abuse of notation) $\chi_*(s)=\chi_*(F^{n_T}(x),s)$, we can write
\begin{equation}\label{eq:deriv-firts}
\begin{split}
 &\langle \bv (x), \nabla \oH_{T}^{\star}(g)(x) \rangle  =- \int_{\bR} dt \,\vartheta_{n_T}(F^{n_T}(x),t) \left(\frac{g}{\nu_{n_T}}\right) \circ F^{-{n_T}}\circ \phi_t\circ F^{n_T}(x)\\
&\phantom{=}  - \int_{\bR} dt \,\chi_* (t) \frac{ \left[(F^{-n_T}\circ \phi_t)^*d g\right](F^{n_T}_*\bv)}{\nu_{n_T}\circ F^{-n_T}\circ \phi_t} \circ  F^{n_T}(x) \\
&\phantom{=}+ \sum_{j=0}^{n_T-1}\int_{\bR} dt \,\chi_*(t) \frac{\left[(F^{-n_T+j}\circ \phi_t)^*(\cL_F^j g \cdot d\ln\nu_1)\right] (F^{n_T}_*\bv)}{\nu_{n_T-j}\circ F^{-n_T+j}\circ \phi_t} \circ  F^{n_T}(x) dt.
\end{split}
\end{equation}
Next, we need an explicit formula for $d\ln\nu_1$. To this end notice that
\begin{equation}\label{eq:V-deriv}
\partial_{x_k}V=p^*_{k} V+p_{k}\hat V^\perp,
\end{equation} 
where $(\hat V_1,\hat V_2)=\hat V=\|V\|^{-1}V$ and $\hat V^\perp=(-\hat V_2,\hat V_1)$. Then, differentiating $\|V\|^2$ and $DF V=\nu_1 V\circ F$, respectively, we have 
\begin{equation}\label{eq:odio0}
\begin{split}
&p^*_{k}=\partial_{x_k}\ln\|V\|\in C^{r-1}\\
&(\partial_{x_k}DF) V+p_{k} DF \hat V^\perp=-p^*_{k}\nu_1 V\circ F+\partial_{x_k} \nu_1  V\circ F\\
&\phantom{(\partial_{x_k}DF) V+p_{k} DF \hat V^\perp=}
+\nu_1 \sum_j \partial_{x_k}F_j\left[p^*_{j} V+p_{j}\hat V^\perp\right]\circ F.
\end{split}
\end{equation}
Multiplying the latter by $\hat V^\perp\circ F$, since $DF^*( \hat V^\perp\circ F)=\frac{\|V\|\det DF }{\nu_1\|V\circ F\|} \hat V^\perp$, yields
\[
\langle \hat V^\perp\circ F, \partial_{x_k}DF V\rangle+p_{k}\frac{ \|V\|\det DF}{ \|V\|\circ F\nu_1}=\nu_1 \sum_j \partial_{x_k}F_jp_{j}\circ F.
\]
Note that, due to the condition that $F$ is uniformly hyperbolic we can assume (eventually using a power of $F$ instead of $F$) 
\begin{equation}\label{eq:hyp-one-step}
\frac{ \|V\|\det DF}{ \|V\|\circ F\nu_1}>\lambda>1>\lambda^{-1}>\nu_1;
\end{equation}
hence, setting 
\begin{equation}\label{eq:gamma-def}
\Gamma(x,v)_k=\langle (D_xFv)^\perp, \partial_{x_k}D_xF v\rangle \frac{\|V(x)\|}{\det D_xF},
\end{equation}
we have\footnote{ By $\Gamma(\hat V)$ we mean the function $\Gamma(\cdot,\hat V(\cdot))$ and $p=(p_1,p_2)\in\bR^2$.}
\begin{equation}\label{eq:diff-reg0}
p=\frac{\nu_1^2\|V\|\circ F}{\|V\|\det DF} DF^* p\circ F-\Gamma(\hat V).
\end{equation}
It is then natural to set 
\begin{equation}\label{eq:A-def}
\Acc=\frac{\nu_1^2\|V\|\circ F}{\|V\|\det DF}
\end{equation}
 and write\footnote{ The series is convergent due to \eqref{eq:hyp-one-step}.}
\begin{equation}\label{eq:diff-reg}
p=-\sum_{m=0}^\infty\left[\prod_{j=0}^{m-1}\Acc\circ F^{j}\right](DF^m)^*\Gamma(\hat V)\circ F^m.
\end{equation}
In the same way, but multiplying the second of \eqref{eq:odio0} by $\hat V\circ F$, we obtain
\begin{equation}\label{eq:diff-nu0}
\begin{split}
\partial_{x_k}\ln \nu_1&=p_{k}\Dcc+\Bcc_k\\
\Bcc_k&=\frac{\langle \hat V\circ F, (\partial_{x_k}DF)\hat V\rangle\|V\|}{\nu_1\|V\|\circ F}+p^*_{k}-(DF^* p^*\circ F)_k\\
\Dcc&=\frac{\langle \hat V\circ F, DF\, \hat V^\perp\rangle}{\nu_1\|V\|\circ F}.
\end{split}
\end{equation}
Note that, by equations \eqref{eq:diff-nu0} and \eqref{eq:diff-reg},
\begin{equation}\label{eq:E-split}
\begin{split}
&(DF^{-n_T+j})^*(\nabla \ln \nu_1)\circ F^{-n_T+j}=(DF^{-n_T+j})^*\Bcc\circ F^{-n_T+j}\\
&-\Dcc\circ F^{-n_T+j}\sum_{m=0}^{n_T-j}\left[\prod_{l=0}^{m-1}\Acc\circ F^{l-n_T+j}\right](DF^{-n_T+j+m})^*\Gamma(\hat V)\circ F^{-n_T+j+m}\\
&-\Dcc\circ F^{-n_T+j}\sum_{m=1}^{\infty}\left[\prod_{l=-n_T+j}^{m-1}\Acc\circ F^{l}\right](DF^{m})^*\Gamma(\hat V)\circ F^{m}.
\end{split}
\end{equation}
We can thus express the last line of \eqref{eq:deriv-firts} in terms of the transfer operators (acting on one forms $\omega=\langle\bar\omega(x), dx\rangle$ and functions $g$, respectively)
\begin{equation}\label{eq:echecavolo}
\begin{split}
&(\hL_F\,\omega)_x=\langle (D_xF^{-1})^*(\nu_1^{-1}\bar\omega)\circ F^{-1}(x), dx\rangle\\
&\cL_{F,A}g=\cL_F(A g).
\end{split}
\end{equation}
Indeed, using the above notation, equation \eqref{eq:E-split}, and setting $\omega_B=\langle B,dx\rangle$ and $\omega_\Gamma=\langle \Gamma(\hat V), dx\rangle$, allows to rewrite \eqref{eq:deriv-firts} as
\begin{equation}\label{eq:deriv-second-00}
\begin{split}
& \langle \bv (F^{-n_T}(x)), \nabla \oH_{T}^{\star}(g)(F^{-n_T}(x)) \rangle 
= - \int_{\bR} dt \,\vartheta_{n_T}(x,t) (\cL_F^{n_T}g)\circ \phi_t(x)\\
&- \int_{\bR} dt \,\chi_*(x,t)\{[\phi_t^*\hL_F^{n_T}dg](F^{n_T}_*\bv)\} (x) \\
&+ \sum_{j=0}^{n_T-1}\int_{\bR} dt \,\chi_*(x,t)  \{[\phi_t^*\hL_F^{n_T-j}((\cL_F^j g)\cdot \omega_B) ](F^{n_T}_*\bv)\}(x) \\
&-\sum_{j=0}^{n_T-1}\sum_{m=0}^{n_T-j}\int_{\bR} dt \,\chi_*(x,t)  \{[\phi_t^*(\hL_{F}^{n_T-j-m}(\cL_{F,A}^m E \cL_F^{j} g)\cdot \omega_\Gamma) ](F^{n_T}_*\bv)\}(x) \\
&-\sum_{j=0}^{n_T-1}\sum_{m=1}^{\infty}\int_{\bR} dt \, \Psi_m(x,t) (\cL_{F,A}^{n_T-j} E\cL_F^{j} g)\circ \phi_t (x)\;;\\
& \Psi_m(x,t)= \chi_*(x,t)\prod_{l=0}^{m-1} A\circ F^{l}\circ \phi_t(x)\cdot [(F^m\circ\phi_t)^*\omega_{\Gamma}(F_*^{n_T}\bv)(x)].
\end{split}
\end{equation}
Next, we need bounds on $\vartheta_m$ and $\Psi_m$.\footnote{ Remark that the point of the next Lemma is that the bounds do not depend on $r$, apart from an irrelevant multiplicative constant.} Let us call $\nu^u_n$ the maximal eigenvalue of $DF^n$, then $\|D_xF^n\|\leq \Const \nu^u_n(x)$.
\begin{lem}\label{lem:test-function}
For each $m\in\bN$ and $x\in\bT^2$, we have, 
\[
\begin{split}
&\|\vartheta_m(x,\cdot)\|_{\cC^{r-1}(\bRp,\bR)}\leq C_{r,n_*} \nu^u_{m}(x)\|\bv\| \\
&\|\Psi_m(x,\cdot)\|_{\cC^{r-1}(\bRp,\bR)}\leq C_{r,n_*} \nu_m(x)\nu^u_{n_T}(x)\|\bv\| .
\end{split}
\]
\end{lem}
\begin{proof}
First of all note that, by the smoothness of the stable manifolds of an Anosov map (see \cite{Hasselblatt97} and references therein) it follows that $\nu_1\circ\phi_{(\cdot)}\in\cC^{r-1}(\bR, \cC^{1+\alpha}(\bT^2))$. Hence, by using repeatedly Schwarz theorem \cite{Rudin}, we have that, for each $p< r$, $\nabla(\partial_t^p\nu_1)\circ\phi_{t}(x)=\partial^p_t(\nabla\nu_1)\circ\phi_{t}(x)$. This implies $\sup_{x\in\bT^2}\|\nabla\nu_1\circ\phi_{(\cdot)}(x)\|_{\cC^{r-1}}\leq C_r$. Also, since $\cC^{r-1}$ is a Banach algebra and
\[
\begin{split}
&\partial_t\nu_n\circ \phi_t(x)=\sum_{k=0}^{n-1}\nu_n\circ \phi_t(x)\nu_k\circ\phi_t(x)\langle V,\nabla(\ln\nu_1)\rangle\circ F^k\circ \phi_t(x)\\
&\partial_tD_{\phi_{t}(x)}F^n=\sum_{k=0}^{n-1}\sum_{j=1}^2\left[D_{\phi_{t}(x)}F^{n-k-1}\partial_{x_j}D_{F^k\circ\phi_{t}(x)}F D_{\phi_t(x)}F^k\right]\nu_{j}\circ \phi_t(x)V_j\circ F^k\circ \phi_t(x)
\end{split}
\]
we have, by induction on $n$ and $r$, $\|\nu_n\circ \phi_{(\cdot)}(x)\|_{\cC^{r-1}((0,1),\bR)}\leq C_r \|\nu_n(\phi_{(\cdot)}x)\|_{\cC^{0}((0,1),\bR)}$ and $\|D_{\phi_{(\cdot)}(x)}F^n\|_{\cC^{r-1}((0,1),\bR)}\leq C_r\|\nu^u_n\circ \phi_{(\cdot)}(x)\|_{\cC^{0}((0,1),\bR)}$, thus, recalling \eqref{eq:hoi-hoi-vei},  $\|\nabla\nu_n\circ \phi_{(\cdot)}(x)\|_{\cC^{r-1}((0,1),\bR)}\leq C_r \nu_n(x)\nu^u_n(x)$.\footnote{ Here we have used Gr\"onwall's inequality to prove $\|\nu_n\circ \phi_{(\cdot)}(x)\|_{\cC^{0}((0,1),\bR)}\leq\Const \nu_n(x)$, and the same for $\nu^u_n$.}

In addition, by \eqref{eq:param-tau} we have
\[
\nabla \tau_n(x,t)=\int_0^tD_{x}\phi_s^*\nabla \nu_n\circ \phi_s(x) ds.
\]
It follows, using Lemma \ref{lem:parabolic} and since $\cC^r$ is an algebra, that
\begin{equation}\label{eq:tauder}
\|\nabla\tau_n\circ\phi_{(\cdot)}(x)\|_{\cC^{r}((0,1),\bR^2)}\leq C_r \nu_n(x)\nu^u_n(x).
\end{equation}
The first inequality in the Lemma follows remembering the definition \eqref{eq:vartheta}.

Let us prove the second. By equations \eqref{eq:gamma-def}, \eqref{eq:A-def} and the smoothness of the stable manifolds of an Anosov map (see \cite{Hasselblatt97} and references therein) it follows that $\Gamma(\hat V)\circ \phi_t$ and  $A\circ \phi_t$  are uniformly (in $x$) $\cC^{r-1}$-bounded functions of $t$.
For each $x\in \bT^2$ we have
\[
\partial_t\prod_{l=0}^{m-1} A\circ F^{l}\circ \phi_{t}(x)=\prod_{l=0}^{m-1} A\circ F^{l}\circ \phi_{t}(x) \cdot\sum_{l=0}^{m-1}\frac{ \nu_l\circ \phi_t (\partial_s A\circ \phi_s|_{s=0})\circ F^{l}\circ \phi_{t}(x)}{A\circ F^{l}\circ \phi_{t}(x)}.
\]
Since each further derivative of a function composed with $F^{l}\circ \phi_{t}$ produces the multiplicative factor $\nu_l$, it follows
\[
\left\|\sum_{l=0}^{m-1}\frac{ \nu_l\circ \phi_t (\partial_s A\circ \phi_s|_{s=0})\circ F^{l}\circ \phi_{t}(x)}{A\circ F^{l}\circ \phi_{t}(x)}\right\|_{\cC^{r-2}}\leq \Const.
\]
On the other hand, notice that $\det D_xF=\nu_1(x)\nu^u_1(x)\frac{\theta\circ F(x)}{\theta(x)}$, where $\theta(x)$ depends only on the angle between the stable and unstable direction at $x$ and on $\|V(x)\|$. Accordingly,
\[
\left\|\prod_{l=0}^{m-1} A\circ F^{l}\circ \phi_{(\cdot)}(x)\right\|_{\cC^0}\leq \Const \frac{\nu_m(x)}{\nu^u_m(x)}.
\]
Thus, we have, by induction,
\[
\begin{split}
\left\|\prod_{l=0}^{m-1} A\circ F^{l}\circ \phi_{(\cdot)}(x)\right\|_{\cC^{r-1}}{\hskip -12pt}&\leq 2^{r-1}\left\|\prod_{l=0}^{m-1} A\circ F^{l}\circ \phi_{(\cdot)}(x)\right\|_{\cC^{0}}+\left\|\partial_t\prod_{l=0}^{m-1} A\circ F^{l}\circ \phi_{(\cdot)}(x)\right\|_{\cC^{r-2}}\\
&\leq 2^{r-1}\left\|\prod_{l=0}^{m-1} A\circ F^{l}\circ \phi_{(\cdot)}(x)\right\|_{\cC^{0}}+\Const \left\|\prod_{l=0}^{m-1} A\circ F^{l}\circ \phi_{(\cdot)}(x)\right\|_{\cC^{r-2}}\\
&\leq C_r\left\|\prod_{l=0}^{m-1} A\circ F^{l}\circ \phi_{(\cdot)}(x)\right\|_{\cC^{0}}\leq C_r  \frac{\nu_m(x)}{\nu^u_m(x)}.
\end{split}
\]
Analogously, $\|D_xF^m\circ \phi_{(\cdot)}\|_{\cC^{r-1}}\leq C_r\nu^u_m(x)$, from which the Lemma follows.
\end{proof}
As in section \ref{sec:growth} we are left with one last problem: the potentials may be non smooth. Such a problem can be solved in the same way as before:  extending all the objects to a subset $\Omega$ of the unitary tangent bundle.

Recall that $(x,v) \in \Omega$ is a three dimensional subset of $\mathbb{T}^2 \times \mathbb{R}^2 $, thus we can naturally write vectors in $T\Omega$ as $(w,\eta)$, $w\in T\bT^2$ and $\eta\in\bR^2$. Accordingly,  a one form $\bgg$ on $\Omega$ at a point $(x,v)$ acts on a vector $(w,\eta)$ as $\bgg_{(x,v)}((w,\eta))$.

Next, let us define
\begin{equation}\label{eq:stop}
\begin{split}
\hAcc\circ \bF^{-1}(x,v)&=\frac{\|V\circ F^{-1}(x)\|\det D_{x}F^{-1}}{\| D_xF^{-1} v\|^2\|V(x)\|}\\
\hBcc_k(x,v)&=\frac{\langle D_xFv, (\partial_{x_k}D_xF)v\rangle}{\|D_xF v\|^2}+p^*_{k}(x)-(D_xF^* p^*\circ F(x))_k\\
\hDcc(x,v)&=\frac{\langle D_xFv, D_xF\, v^\perp\rangle}{\|D_xF v\|^2\|V(x)\|}.
\end{split}
\end{equation}
We can then define the operators, acting, respectively, on functions $\bg$ and on one forms $\bgg$ defined on $\Omega$ by
\begin{equation}\label{eq:new-new} 
\begin{split}
&\cL_{\bF,\hAcc}\bg=\cL_{\bF}(\hAcc \bg)\\
&\left[\widehat\cL_{\bF} \bgg\right]_{(x,v)}= \frac{\|D_xF^{-1}v\|\,\|V(x)\|}{ \|V\circ F^{-1}(x)\|} \left[(\bF^{-1})^*\bgg\right]_{(x,v)}.
\end{split}
\end{equation}
The relation with the previously defined operators is given by the following Lemma.
\begin{lem}\label{lem:LtoL}
For each $x\in\bT^2$ and $w\in\bR^2$ we have\footnote{ See \eqref{eq:echecavolo} for the definition of $\cL_{F,A}$ and $\widehat \cL_F$.}
\[
\begin{split}
&[\cL_{\bF,\hAcc} (g\circ\bpi)](x,\hat V(x))=\cL_{F,A}( g)(x)=(\cL_{F,A}( g))\circ \bpi(x, \hat V(x))\\
&\left[\widehat\cL_{\bF} \bpi^*dg\right]_{(x,\hat V(x))}\!\!(w,0)= \left[\widehat\cL_F dg\right]_{x}\!\!(w)=\left[\bpi^*(\widehat\cL_F dg)\right]_{(x,\hat V(x))}\!\!(w,0) .
\end{split}
\]
\end{lem}
\begin{proof}
By direct computation $\hAcc(x,\hat V(x))=A(x)$ and the first statement of the Lemma follows. The second follows directly from the definition since $(\bF^{-n})^* \bpi^* dg=\bpi^* (F^{-n})^*dg$.
\end{proof}
Recalling equation \eqref{eq:LtoL} and Lemma \ref{lem:LtoL}, and setting $\bx=(x, \hat V(x))$,  $\bg=g\circ\bpi$, we can rewrite \eqref{eq:deriv-second-00} as:
\begin{equation}\label{eq:deriv-second-0}
\begin{split}
& \langle \bv (F^{-n_T}(x)), \nabla \oH_{T}^{\star}(g)(F^{-n_T}(x)) \rangle 
= - \int_{\bR} dt \,\vartheta_{n_T}(x,t) (\cL_{\bF}^{n_T}\bg)\circ \bphi_t(\bx)\\
&- \int_{\bR} dt \,\chi_*(x,t)\{[\bphi_t^*\hL_{\bF}^{n_T}\bpi^*dg](\bF^{n_T}_*(\bv,0))\} (\bx) \\
&+ \sum_{j=0}^{n_T-1}\int_{\bR} dt \,\chi_*(x,t)  \{[\bphi_t^*\hL_{\bF}^{n_T-j}((\cL_{\bF}^j \bg)\cdot \hbomega_B) ](\bF^{n_T}_*(\bv,0))\}(\bx) \\
&-\sum_{j=0}^{n_T-1}\sum_{m=0}^{n_T-j}\int_{\bR} \hskip-2pt dt \,\chi_*(x,t)  \{[\bphi_t^*(\hL_{\bF}^{n_T-j-m}(\cL_{\bF,\hAcc}^m \hDcc \cL_{\bF}^{j} \bg)\cdot \hbomega_\Gamma) ](\bF^{n_T}_*(\bv,0))\}(\bx) \\
&-\sum_{j=0}^{n_T-1}\sum_{m=1}^{\infty}\int_{\bR} dt \, \Psi_m(x,t) (\cL_{\bF,\hAcc}^{n_T-j} \hDcc\cL_{\bF}^{j} \bg)\circ \bphi_t (\bx)\\
&\hbomega_B=\langle (\hBcc, 0),(dx, dv)\rangle;\;\quad \hbomega_\Gamma=\langle (\Gamma(v), 0),(dx, dv)\rangle.
\end{split}
\end{equation}
Last, for each time dependent function $\vf\in L^\infty(\bR^3,\bR)$, each  form $\bgg$ on $\Omega$ and time dependent vector field  $w\in L^\infty(\bR^3, \bR^2)$, with compact support in $\bRp$, we define\footnote{ The first is just a slight generalization of \eqref{eq:molli}.}
\begin{equation}\label{eq:bH1-def}
\begin{split}
&\bH_{x,\vf}(\bg)=\int_{\bR}\vf(x,s)\cdot \bg\circ\bphi_s(x, \hV(x))\, ds\\
&\bH^1_{x,w}(\bgg)=\int_{\bR} \bgg_{\bphi_s(x,\hat V(x))}((D_x\phi_s w(x,s), 0)) ds.
\end{split}
\end{equation}
With such a notation  we can finally write \eqref{eq:deriv-second-0} as
\begin{equation}\label{eq:deriv-second}
\begin{split}
 \langle \bv (x),& \nabla \oH_{T}^{\star}(g)(x) \rangle 
= - \bH_{F^{n_T}(x),\vartheta_{n_T}} (\cL_{\bF}^{n_T}\bg)\\
&- \bH^1_{F^{n_T}(x),\chi_*\bF^{n_T}_*(\bv,0)}(\hL_{\bF}^{n_T}\bpi^*dg) \\
&+ \sum_{j=0}^{n_T-1}\bH^1_{F^{n_T}(x),\chi_*\bF^{n_T}_*(\bv,0)}(\hL_{\bF}^{n_T-j}((\cL_{\bF}^j \bg)\cdot \hbomega_B) ) \\
&-\sum_{j=0}^{n_T-1}\sum_{m=0}^{n_T-j} \bH^1_{F^{n_T}(x),\chi_*\bF^{n_T}_*(\bv,0)}(\hL_{\bF}^{n_T-j-m}(\cL_{\bF,\hAcc}^m \hDcc \cL_{\bF}^{j} \bg)\cdot \hbomega_\Gamma)  \\
&-\sum_{m=1}^{\infty}\sum_{j=0}^{n_T-1}\bH_{F^{n_T}(x),\Psi_m} (\cL_{\bF,\hAcc}^{n_T-j} \hDcc\cL_{\bF}^{j} \bg)=:\cJ_{n_T}(F^{n_T}(x)).
\end{split}
\end{equation}
It follows that, if $n_T=n_*K_T$, the derivative of \eqref{eq:newstart1} takes the form
\begin{equation}\label{eq:preliminary2}
\begin{split}
&\langle \bv, \nabla \overline H_{T}(g)\rangle=\sum_{l=1}^{K_T}\cJ_{ln_*}(F^{l n_*}(x))-\int_{\bRpe} \chi(s) [\phi_s^*dg(\bv)](x) ds.
\end{split}
\end{equation}
The above corresponds to Lemma \ref{lem:suffice} in the present context. 

\subsection{ Lipschitz coboundary}\ \newline\label{subsec:second-proof}
Having shown that the problem can be cast in a setting completely analogous to the one already discussed in Section \ref{sec:growth}, we are now ready to conclude.
\begin{proof}[{\bf Proof of Theorem \ref{thm:maintwo}}]
This is the same argument carried out in the proof of Theorem \ref{thm:main}, only now we also need the spectral picture for the operator $\cL_{\bF,\hAcc}$ on $\cB^{p,q}$ and of the operator $\widehat\cL_{\bF}$ acting on an appropriate (new) space $\widehat \cB^{p,q}$, $p+q\leq r-2$. 
Indeed, by the arguments in appendix \ref{sec:norms} it follows that $\cL_{\bF,\hAcc}$ is quasi compact on $\cB^{p,q}$ and in appendix \ref{sec:norms-bis} we show that there exists a Banach space $\widehat \cB^{p,q}$ on which $\widehat\cL_{\bF}$ has spectral radius $\rho>0$ and essential spectral radius bounded by $\lambda^{-\min\{p,q\}}\rho$. Also  the functionals $\bH^1_{x,w}$, for $w\in\cC^r_0$, are bounded by
\[
|\bH^1_{x,w}(\bgg)|\leq \Const\|w\|_{\cC^{r-2}}\|\bgg\|_{\widehat\cB^{p,q}}.
\]

As before we notice that in \eqref{eq:preliminary2} the last term is uniformly bounded, hence we have to worry only about the terms $\cJ_{ln_*}$. Looking at \eqref{eq:deriv-second} we see that each $\cJ_{ln_*}$ consists of five terms.

By Lemma \ref{lem:test-function} we see that the first term is bounded only if $p,q$ (and hence $r$) are large enough and $g$ belongs to the kernel of enough eigenprojectors of $\cL_{\bF}$ so that the essential spectral radius of $\cL_\bF$, when restricted to the invariant subspace to which $g$ belongs, is smaller than $\|\nu^u_1\|_\infty^{-n}$. Analogously, the second term is bounded if $\bpi^* dg$ belongs to the kernel of enough eigenprojectors of $\hL_{\bF}$ (again the spectral radius of the remainder must be smaller of $\|\nu^u_1\|_\infty^{-n}$). The third term has a bit more complex structure. First of all note that the multiplication by a smooth function is a bounded operator from $\cB^{p,q}$  to itself, while the multiplication by a smooth one-form is a bounded operator from $\cB^{p,q}$ to $\widehat\cB^{p,q}$ (to verify it just use the norms definitions \eqref{eq:def-norms} and \eqref{eq:def-norms-bis}). Next, assuming that $g$ belongs to the above subspaces and using the spectral decomposition of $\hL_{\bF}$ we can write,\footnote{ That is $\hL_{\bF}^j=\sum_{k=0}^{m_*}[ \rho_k \Pi_k +Q_k]^j+\cO(\|\nu^u_1\|_\infty^{-j})$, where $Q_k^{d_k}=0$ and $\Pi_k Q_k=Q_k\Pi_k=Q_k$ and $\Pi_k^2=\Pi_k$. Also we use the useful convention $Q_k^0=\Pi_k$.} for some $m_*\in\bN$, $\rho_*\in (0,1)$ and all $l\leq n_T$,
\[
\begin{split}
&\sum_{k=0}^{m_*}\sum_{j=0}^{l-1}\bH^1_{F^{n_T}(x),\chi_*\bF^{n_T}_*(\bv,0)}( [ \rho_k \Pi_k +Q_k]^{l-j}((\cL_{\bF}^j \bg)\cdot \hbomega_B) )+\cO(\rho_*^l)\\
&=\sum_{k=0}^{m_*}\sum_{j=0}^{l-1}\sum_{p=0}^{\min\{l-j,d_k-1\}}\binom{l-j}{p}\bH^1_{F^{n_T}(x),\chi_*\bF^{n_T}_*(\bv,0)}( \rho_k^{l-j-p} Q_k^{p}((\cL_{\bF}^{j} \bg)\cdot \hbomega_B) )\\
&\phantom{=}+\cO(\rho_*^l).
\end{split}
\]
On the other hand, for $p\leq l-j$,
\[
\begin{split}
&\sum_{j=0}^{l-1}\binom{l-j}{p}\bH^1_{F^{n_T}(x),\chi_*\bF^{n_T}_*(\bv,0)}( \rho_k^{l-j-p} Q_k^{p}((\cL_{\bF}^{j} \bg)\cdot \hbomega_B) )\\
&=\rho_k^{l-p}\sum_{s=0}^pc_{s,p}l^{p-s}\bH^1_{F^{n_T}(x),\chi_*\bF^{n_T}_*(\bv,0)}( Q_k^{p}(\cK_{s} \bg)\cdot \hbomega_B) )+\cO(\rho_*^l)
\end{split}
\]
where $\cK_{s}=\sum_{j=0}^{\infty}j^s(\rho_k^{-1}\cL_{\bF})^{j}$ and $\sum_{s=0}^pc_{s,p} l^{p-s}j^s=\binom{l-j}{p}$. By identifying the range of $\Pi_k$ with $\bR^{d_k}$ we can then define the functional $\ell_{k,s}: \cB^{p,q}\to\bR^{d_k}$ as $\ell_{k,s}(\bg)=\Pi_k( \cK_{s}(\bg)\cdot \hbomega_B)$. Accordingly, if $\ell_{k,s}(\bg)=0$ for all $k\leq m_*$ and $s\leq d_k$, we have that also the third term is uniformly bounded.

Similar arguments hold for the fourth and fifth term.
This implies that if $g$ belongs to and appropriate finite codimensional subspace (determined by the above eigendistributions) then the $\oH_{T}(g)$ are equicontinuous functions of $x$. Hence there exists $\{T_j\}$ such that $\oH_{T_j}(g)$ converges uniformly to a Lipschitz function. We have thus shown that $\oH_{T}(g)$ has a convergent subsequence to a Lipschitz function $h$ which, for all $t\in\bR$, satisfies
\[
h\circ\phi_t(x)-h(x)=\int_0^tg\circ \phi_s(x) ds
\]
Hence, $h$ satisfies \eqref{eq:homology} and $g$ is a Lipschitz coboundary.
\end{proof}

%%%%%%%%%%%%%%%%%%%%%%%%%%
\section{Examples}\label{sec:examples}
Lemma \ref{lem:classify} shows that the flows to which our theory applies must necessarily enjoy several properties, the reader might be left wondering if such flows exist at all (apart, of course, for the trivial one consisting of rigid rotations).

To construct examples the simplest strategy is to reverse the logic and start with a $\cC^r$ Anosov map which is orientation preserving.\footnote{ After all, this is what is done for the horocycle flow: one starts from the geodesic flow.} Given such a map, we have an associated stable distribution. If we choose any strictly positive function $\cN\in\cC^r(\bT^2,\bR)$ there are only two fields $V$ such that $V(x)\in E^s(x)$ for all $x\in\bT^2$ and $\|V(x)\|=\cN(x)$, they correspond to the two possible orientations. We can then choose any of the two and we have, at the same time, an example that satisfies all our assumptions and a justification of such assumptions.
Indeed, in general the distribution $E^s$ of a $\cC^r$ Anosov map will be only $\cC^{1+\alpha}$ with $\alpha\in(0,1)$, \cite{KH}. Notice however that it is possible to have situations in which $\alpha>1$ and yet $F$ is not $\cC^1$ conjugated to a toral automorphism \cite[Exercise  19.1.5]{KH}. Of course, in the latter case the unstable foliation will be irregular \cite[Corollary 3.3]{ghys93}.

The above partially clarifies the applicability of our work. Nevertheless, other reasons of unhappiness persist. In particular all our discussion, up to this point, has been a bit abstract as we did not discuss which type of concrete objects it really yields. To gain a better understanding we start by working out the linear case that, surprisingly, is not completely trivial.

\subsection{A ``trivial" example}\label{rem:Fouriercomputation}\ \newline
For the reader convenience we discuss here the case in which $F$ is linear and $\phi_t$ is generated by a constant vector field. As already mentioned in the introduction this is an analogue of the case, for the geodesic-horocycle flow setting, of compact manifolds of constant negative curvature. Hence it can be dealt with directly by representation theory (i.e. Fourier transform, in the present setting), without using the strategy put forward in this paper.

Let $A\in SL(2,\bN)$; let $F_A:\bT^2 \to \bT^2$ be the Anosov map defined by $F_A(\xi) := A\xi\mod 1.$
Since $\det(A) = 1 $ the map is invertible, and has eigenvalues $\lambda_A, \lambda^{-1}_A \in \bR $, with  $\lambda_A > 1 $. Let $V_A=(1,\omega)$  be the eigenvector associated to the eigenvalue $\lambda_A^{-1}$. Note that $\omega$ is a quadratic irrational, as in Lemma \ref{lem:classify}. Let $\phi_t(\xi) = \xi + tV_A \mod 1 $. In this case by applying Fourier transform to equation \eqref{eq:homology} we obtain, for $k \in \bZ^2$, calling $ \hat{f}_k $ the Fourier coefficients of $f$,
\[ 
\sum_{k \in \bZ^2}  2\pi i\langle V , k  \rangle   \hat{h}_k  e^{2\pi i\langle k, \xi \rangle}    = 
  \sum_{k \in \bZ^2} \hat{g}_k e^{2\pi i\langle k, \xi \rangle}.
 \]
Note that we have the trivial obstruction $\hat g_0=0$. If this is satisfied, note that $\langle V, k \rangle=k_1+\omega k_2\neq 0$ for all $(k_1,k_2)\in\bZ^2\setminus \{0\}$, since $\omega$ is irrational. Thus, we can write
\begin{equation}\label{eq:cohomology}
 \hat{h}(k) = - i\frac{\hat{g}(k)}{ 2\pi\langle V, k \rangle}.
\end{equation}
Since $\omega$ is a quadratic irrational, it is well known (e.g. by using standard results on continuous fractions) that $|\langle V, k \rangle|\geq \Const \|k\|^{-1}$. Hence, if $g\in W^{r,2}$ (the Sobolev space with the first $r$ derivatives in $L^2$), then $h\in W^{r-1,2}$. In particular, if $g\in\cC^\infty$, then  $h\in\cC^\infty$. 

That is, in this example all the aforementioned distributions do not exist and the only obstruction is the trivial one: the one given by the invariant measure. If such an obstruction is satisfied (i.e. if $\operatorname{Leb}(g)=0$), then the ergodic integrals are bounded and $g$ is a coboundary with the maximal regularity one can expect.

Nevertheless, it is very instructive to apply to this example also our strategy. This will give us a feeling for what might happen in general.  

\subsubsection{\bfseries{Ergodic integrals}}\label{sec:ergo}\ \newline 
 Let us change coordinates in $\Omega$.
One convenient choice is $\theta(\xi,s)=(\xi,v(s))$ with $v(s)=(1, s)(1+s^2)^{-\frac 12}$ and $s<0$. In this co-ordinates we have that the set $\Omega$, defined just before \eqref{eq:inv-stable-vec}, reads $\bT^2\times [-\beta_1,-\beta_2]$ for some $0<\beta_2<\beta_1$. Also, calling $\widehat\bF=\theta^{-1}\circ \bF\circ \theta$, we have 
\[
\begin{split}
&\widehat \bF(\xi,s)=(F(\xi), \psi(s))\\
&\psi(s)=\frac {c+sd}{a+sb}\;;\quad A=\begin{pmatrix} a&b\\c&d\end{pmatrix}.
\end{split}
\]
In addition,
\begin{equation}\label{eq:psinv}
\psi^{-1}(s)=\frac{as-c}{d-bs}.
\end{equation}
The map $\psi^{-1}$ is a contracting map with derivative $(\psi^{-1})'(s)=(d-bs)^{-2}$ and a unique fixed point $\bar s$ in $[-\beta_1,-\beta_2]$. 

The smallest eigenvalue of $A$ is given by $\bar\nu=(d-\bar sb)^{-1}<1$ and the corresponding eigenvector is  $\bar v=(1, \bar s)$. 

Setting, as usual, $\theta^*\bg=\bg\circ \theta$, for $\bg\in\cC^0(\Omega,\bC)$, and introducing the multiplication operator $\Xi\, \bg(\xi,s)=\frac{\sqrt{1+s^2}}{\|V(\xi)\|}\bg(\xi,s)$, let us define, recalling that $\cL_\bF$ is defined in equation \eqref{eq:new-to},
\begin{equation}\label{eq:tra-con}
\cL_{\widehat \bF}=\Xi\, \theta^*\cL_\bF (\theta^*)^{-1}\,\Xi^{-1}.
\end{equation}
By direct computation it follows\footnote{ Remember that $s<0$, hence $d-sb>0$.}
\begin{equation}\label{eq:linear-basic}
\cL_{\widehat\bF}\, \bg(\xi,s)=\bg\circ \widehat\bF^{-1}(\xi,s) (d-bs).
\end{equation}
By Theorem \ref{thm:main} the obstruction are determined by the spectrum of $\cL_\bF$ on the Banach spaces $\cB^{p,q}$ (see Appendix \ref{sec:norms}). 
Note that \eqref{eq:tra-con} implies that the operator $\cL_\bF$ is conjugated to the operator $\cL_{\widehat\bF}$ on the space $\Xi\, \theta^*(\cB^{p,q})$ which, to simplify notation, we still call $\cB^{p,q}$ since the identification is trivial; hence they have the same spectrum.
Accordingly, it suffices to study the spectrum of $\cL_{\widehat\bF}$. We detail it in the next Lemma; see Remark \ref{rem:useless-eig} for the consequences of the Lemma on the existence of the obstructions.
\begin{lem}\label{lem:spectrum-final}
For each $q\geq p\in\bN$, setting $D_{p}= \{z\in\bC\;:\;|z|> \bar\nu^{p-1}\}$,  
\[
\sigma_{\cB^{p,q}}(\cL_{\widehat\bF})\cap D_{p}= \{\bar\nu^{2k-1}\}_{k\in \bN}\cap D_{p},
\]
and the eigenvalues have multiplicity one. 
In addition, there exist explicit formulae for the right and left eigenvectors associated to the point spectrum (see equations \eqref{eq:f-product}, \eqref{eq:f-product2} and \eqref{eq:kernels}).
\end{lem}
\begin{proof}
It is convenient to consider a larger Banach space $\cB^{p,q}_*$ which is defined exactly as $\cB^{p,q}$ with the only difference that the set $\Sigma$ of admissible curves (see Definition \ref{def:leaves}) is smaller, as it consists only of stable curves (that, in this case, are just segments in the direction $\hat V$). Indeed, since the sup is taken on a smaller set, it follows $\|\cdot\|_{\cB^{p,q}_*}\leq \|\cdot\|_{\cB^{p,q}}$ which implies $\cB^{p,q}\subset \cB^{p,q}_*$. In addition, note that if $\cL$ is a bounded operator both on $\cB^{p,q}$ and $\cB^{p,q}_*$ and $R\subset\bC$ is a region in which both $\sigma_{\cB^{p,q}_*}(\cL)$ and $\sigma_{\cB^{p,q}}(\cL)$ consist only of point spectrum with finite multiplicity, then $\sigma_{\cB^{p,q}}(\cL)\cap R\subset \sigma_{\cB^{p,q}_*}(\cL)\cap R$. Indeed, if $h\in \cB^{p,q}$ is an eigenvector of $\cL$, then it is an eigenvector of $\cL$ also when it is viewed as an operator from $\cB^{p,q}_*$ to itself. In fact, \cite[Lemma A.1]{BaladiTsujii08} implies  $\sigma_{\cB^{p,q}}(\cL)\cap R= \sigma_{\cB^{p,q}_*}(\cL)\cap R$.

The advantage in considering $\cB^{p,q}_*$ rests in the fact that the norm (and all the theory) simplifies considerably since the stable and unstable distributions are constant.\footnote{ If the reader wonders why not always use such a space, he should consider that such a space is directly tied to the map and does not work for a small perturbation, as the invariant distributions will be different. Hence the spaces $\cB^{p,q}_*$ would be useless in the following sections.}  This implies in particular that $\partial_\sigma g=\langle\nabla_{x,y} g, \hat V\rangle$, $\partial_ug=\langle\nabla_{x,y} g, \hat V^\perp\rangle$ and $\partial_s$ all commute.\footnote{ Of course, $\partial_\sigma$ is just the derivative in the stable direction and $\partial_u$ the one in the unstable direction of $F$ and $\partial_s$ is the derivative with respect to $s$.} 

Setting $B_{\beta, r}=\{\vf\in \cC^{\beta}_0([-r,r],\bC)\;:\; \|\vf\|_{\cC^{\beta}([-r,r],\bC)}\leq 1\}$, the analogous of definition \eqref{eq:def-norms} is equivalent to
\[
\|g\|_{\cB^{p,q}_*}=\sup_{\substack {x\in \bT^2\\r\in[\delta/2,\delta]}}\sup_{\substack{p_\sigma+p_u+p_s=p_*\\
p_*\leq p}}\sup_{\vf\in B_{q+p_*,r}}
\left|\int_{-r}^r\left[\partial_\sigma^{p_\sigma}\partial_u^{p_u}\partial_s^{p_s} g\right](x+\hat V t, \bar s) \vf(t) dt\right|,
\]
where, for further convenience, we make the choice $\delta\leq \bar\nu$.
In addition, since  $\vf$ is compactly supported in $[-r,r]$,
\[ 
\int_{-r}^r\left[\partial_\sigma g\right](x+\hat V t, \bar s) \vf(t) dt=\int_{-r}^r\frac{d}{dt} g(x+\hat V t, \bar s) \vf(t) dt=-\int_{-r}^rg (x+\hat V t, \bar s) \vf'(t) dt.
\]
Hence, we can write
\[
\|g\|_{\cB^{p,q}_*}=\sup_{x\in \bT^2}\sup_{r\in[\delta/2,\delta]}\;\sup_{p_u+p_s=p_*\leq p}\;\sup_{\vf\in B_{q+p_*,r}}
\left|\int_{-r}^r\left[\partial_u^{p_u}\partial_s^{p_s} g\right](x+\hat V t, \bar s) \vf(t) dt\right|.
\]
Next, we establish a Lasota-Yorke inequality by a super simplified version of the arguments used in \cite[Lemma2.2]{GouezelLiverani06}.
\begin{sublem}\label{sublem:lasota-easy}
For all $p,q>0$ there exist $C>0$ such that, for all $n\in\bN$,
\[
\begin{split}
&\|\cL_{\widehat\bF}^n h\|_{\cB^{p,q}_*}\leq C\bar\nu^{-n}\|h\|_{\cB^{p,q}_*}\\
&\|\cL_{\widehat\bF}^n h\|_{\cB^{p,q}_*}\leq C\bar\nu^{n(\min\{p,q\}-1)}\|h\|_{\cB^{p,q}_*}+C\bar \nu^{-n}\|h\|_{\cB^{p-1,q+1}_*}.
\end{split}
\]
\end{sublem}
\begin{proof}
Note that $\partial_u\cL_{\widehat\bF} h=\bar\nu \cL_{\widehat\bF} \partial_u h$ and 
\begin{equation}\label{eq:deriv-s}
\partial_s\cL_{\widehat\bF} h=(d-bs)^{-2}\cL_{\widehat\bF} \partial_s h-b(d-bs)^{-1}\cL_{\widehat\bF} h .
\end{equation} 
Iterating the above equations yield, for each $p_u+p_s=p_*\leq p$, $\vf\in B_{q+p_*,r}$,
\[
\begin{split}
&\left|\int_{-r}^r\left[\partial_u^{p_u}\partial_s^{p_s} \cL_{\widehat\bF}^ng\right](x+\hat V t, \bar s) \vf(t) dt\right|
\leq \bar\nu^{(2p_s+p_u)n}\left|\int_{-r}^r\left[ \cL_{\widehat\bF}^n\partial_u^{p_u}\partial_s^{p_s}g\right](x+\hat V t, \bar s) \vf(t) dt\right|\\
&+C_{p}\sum_{0\leq p'_s\leq p_*-1-p_u}\left|\int_{-r}^r\left[ \cL_{\widehat\bF}^n\partial_u^{p_u}\partial_s^{p'_s}g\right](x+\hat V t, \bar s) \vf(t) dt\right|.
\end{split}
\]
On the other hand
\[
\int_{-r}^r\left[ \cL_{\widehat\bF}^nh\right](x+\hat V t, \bar s) \vf(t) dt=\int_{-\bar\nu^{-n}r}^{\bar\nu^{-n}r} h(x+\hat V t, \bar s) \vf(\bar\nu^n t) dt.
\]
To compare the above with the integrals defining the norm we introduce a smooth partition of unity $\{\phi_k\}$ of the form $\phi_k(t)=\phi(t-k\delta/2)$, $\supp\phi\subset [-\delta,\delta]$, and write $\vf(\bar\nu^n t)=\sum_k \vf_{k,n}(t-k\delta/2)$ where $\vf_{k,n}(t)=\phi(t)\vf(\bar\nu^{n}(t+k\delta/2))$. Then, setting $x_k=x+k\delta/2 \hat V$,
\[
\int_{-r}^r\left[ \cL_{\widehat\bF}^nh\right](x+\hat V t, \bar s) \vf(t) dt=\sum_k\int_{-\delta}^{\delta} h(x_k+\hat V t, \bar s) \vf_{k,n}(t) dt.
\]
Note that the sum consists of at most $\Const\bar\nu^{-n}$ terms.
This yields the first estimate of the Sub-Lemma since $\|\vf_{k,n}\|_{\cC^{q+p_*}}\leq \Const$. Moreover we have
\[
\begin{split}
\left|\int_{-r}^r\left[\partial_u^{p_u}\partial_s^{p_s} \cL_{\widehat\bF}^ng\right](x+\hat V t, \bar s) \vf(t) dt\right|
\leq& \bar\nu^{p_*n}\sum_k\left|\int_{-\delta}^\delta\left[ \partial_u^{p_u}\partial_s^{p_s}g\right](x_k+\hat V t, \bar s) \vf_{k,n}(t) dt\right|\\
&+C_{p}\bar\nu^{-n}\|g\|_{\cB^{p-1,q+1}_*}.
\end{split}
\]
If $p_*=p$ then
\[
\begin{split}
\left|\int_{-r}^r\left[\partial_u^{p_u}\partial_s^{p_s} \cL_{\widehat\bF}^ng\right](x+\hat V t, \bar s) \vf(t) dt\right|
\leq C_q \bar\nu^{(p-1)n}\|g\|_{\cB^{p,q}_*}+C_{p}\bar\nu^{-n}\|g\|_{\cB^{p-1,q+1}_*},
\end{split}
\] 
 If $p_*<p$, recalling the choice $\delta\leq \bar\nu$,
\[
\begin{split}
&\left\|\vf_{k,n}(t)-\phi(t)\sum_{j=0}^{q-1}\frac{\vf^{(j)}(\bar\nu^{n}k\delta/2)\bar\nu^j}{j!}t^j\right\|_{\cC^{q+p_*}}\leq C_q\bar\nu^q\\
&\left\|\phi(t)\sum_{j=0}^{q-1}\frac{\vf^{(j)}(\bar\nu^{n}k\delta/2)\bar\nu^j}{j!}t^j\right\|_{\cC^{q+p_*+1}}\leq C_q.
\end{split}
\]
Thus
\[
\begin{split}
\left|\int_{-r}^r\left[\partial_u^{p_u}\partial_s^{p_s} \cL_{\widehat\bF}^ng\right](x+\hat V t, \bar s) \vf(t) dt\right|
\leq&C_q \bar\nu^{(p_*-1)n}\bar\nu^{qn}\|g\|_{\cB^{p^*,q}_*}+C_{p}\bar\nu^{-n}\|g\|_{\cB^{p-1,q+1}_*}\\
\leq&C_q \bar\nu^{(q-1)n}\|g\|_{\cB^{p,q}_*}+C_{p}\bar\nu^{-n}\|g\|_{\cB^{p-1,q+1}_*},
\end{split}
\]
from which the second inequality in the statement of the Lemma follows by taking the sup on $x$, $p_u,p_s$ and $\vf$.
\end{proof}
\begin{sublem}\label{sublem:essential0}
For each $q, p\in\bN$ the essential spectral radius of $\cL_{\widehat\bF}$ acting on $\cB^{p,q}_*$ is bounded by $\bar\nu^{\min\{p,q\}-1}$.
\end{sublem}
\begin{proof}
One can check that the unit ball of $\cB^{p,q}_*$ is weakly compact in $\cB^{p-1,q+1}_*$ by a much simplified version of the proof of \cite[Lemma 2.9]{Liverani04}. Then the Lemma follows by the Lasota-Yorke inequality established in Sub-Lemma \ref{sublem:lasota-easy} and Hennion theorem \cite{Hennion93}.
\end{proof}
We can now study the point spectrum. 
We start by establishing a lower bound on its cardinality.
\begin{sublem}\label{lem:discrete-sp}
For each $p,q\in \bN$, $q, p\in\bN$, we have
\[
\sigma_{\cB^{p,q}_*}(\cL_{\widehat\bF})\supset \{\bar\nu^{2k-1}\}_{k\geq 0}.
\]
\end{sublem}
\begin{proof}
We look for $\cL_{\widehat\bF} \bg=\mu\bg$ of the form $\bg(\xi,s)=g(\xi)f(s)$. Then
\[
\mu g(\xi)f(s)=g\circ F^{-1}(\xi) f(\psi^{-1}(s)) (d-bs).
\]
By the computation Fourier modes done at the beginning of section \ref{rem:Fouriercomputation} it follows that $g(\xi)=1$; hence
\[
\mu f(s)= f(\psi^{-1}(s)) (d-bs).
\]
Iterating the above relation yields
\[
f(s)=\mu^{-n}\prod_{k=0}^{n-1}(d-b\psi^{-k}(s)) f(\psi^{-n}(s)).
\]
Since $\psi^{-n}(s)$ converges to $\bar s$ there are two possibilities; either $f(\bar s)\neq 0$ or $f(\bar s)=0$. In the first case we can assume, without loss of generality, that $f(\bar s)=1$. Then
\begin{equation}\label{eq:f-product}
f(s)=\prod_{k=0}^{\infty}\mu^{-1}(d-b\psi^{-k}(s))=\prod_{k=0}^{\infty}\mu^{-1}(d-b\bar s+b[\psi^{-k}(\bar s)-\psi^{-k}(s)]),
\end{equation}
provided the product converges. If we choose $\mu=d-b\bar s=\bar\nu^{-1}$, then, for any $\tau\in( \bar \nu^2,1)$ we have
\[
\prod_{k=0}^{\infty}\mu^{-1}(\mu+b[\psi^{-k}(\bar s)-\psi^{-k}(s)])=e^{\sum_{k=0}^\infty \cO(\tau^k)}
\]
which shows the convergence. Since $f \in\cC^\infty$, we have that $f\in \cB^{p,q}_*$ for all $p,q\in\bN$. We have thus an eigenvector.

Next, consider the case $f(\bar s)=0$. It is natural to
look for solutions of the form $f(s)=(s-\bar s)^p M(s)$, $M(\bar s)=1$, for $p\in\bN$, then, recalling \eqref{eq:psinv},
\[
\begin{split}
\mu(s-\bar s)^pM(s)&=(d-bs)(\psi^{-1}(s)-\psi^{-1}(\bar s))^pM\circ \psi^{-1}(s)\\
&=(d-bs)^{-p+1}(d-b\bar s)^{-p}(s-\bar s)^pM\circ \psi^{-1}(s).
\end{split}
\]
Iterating again we have
\begin{equation}\label{eq:f-product2}
M(s)=\prod_{k=0}^\infty\mu^{-1}(\bar\nu^{-1}+b(\psi^{-k}(\bar s)-\psi^{-k}(s)))^{-p+1}\bar\nu^{p}.
\end{equation}
The above is convergent and non zero provided $\mu=\bar\nu^{2p-1}$. Such a choice yields $M(s)=f(s)^{-p+1}$ and since $1/f\in\cC^\infty$ again we have an eigenvector in $\cB^{p,q}_*$.
\end{proof}
Next, we need a lower bound on the cardinality of the point spectrum.
The idea to achieve it is to investigate the spectrum of $\cL_{\widehat\bF}'$ by considering the distributions
\begin{equation}\label{eq:kernels}
\delta^n_*(h)=\int_{\bT^2} [\partial_s^nh](x,\bar s) dx.
\end{equation}
Note that, for $q\geq p\geq n$, $\delta^*_n\in (\cB^{p,q}_*)'$. Indeed, $\bT^2$ can be partitioned by pieces of stable curves, hence $\delta^n_*$ can be seen as an convex combination of the functionals that define the $\cB^{p,q}_*$ norm, thus it is bounded.
\begin{lem}\label{lem:invariant-p}
The spaces $\bW_n=\{h\in \cB^{p,q}_*\;:\; \delta^0_*(h)=\dots=\delta^n_*(h)=0\}$ are invariant  for $\cL_{\widehat\bF}$.
\end{lem}
\begin{proof}
Note that
\[
\delta^0_*(\cL_{\widehat\bF}g)=\int_{\bT^{2}}g(F^{-1}(x),\bar s)\bar\nu^{-1}dx=\bar\nu^{-1}\int_{\bT^{2}}g(x,\bar s)dx=\bar\nu^{-1}\delta^0_*(g).
\]
Moreover, iterating \eqref{eq:deriv-s} it follows that there exist smooth functions $a_k$ such that
\[
\partial_s^n\cL_{\widehat\bF}g=\sum_{k=0}^n a_k(s)\cL_{\widehat\bF}\partial_s^k g.
\]
Thus,
\[
\delta^n_*(\cL_{\widehat\bF}g)=\sum_{k=0}^{n} a_k(\bar s)\int_{\bT^{2}}[\partial_s^kg](F^{-1}(x),\bar s)\bar\nu^{-1}=\sum_{k=0}^n a_k(\bar s)\bar\nu^{-1}\delta^k_*(g).
\]
\end{proof}
Note that the above implies that $\cL_{\widehat\bF}'$ has eigenvectors formed by linear combinations of $\delta^n_*$.
\begin{sublem}\label{sublem:small-ess}
For each $m\geq 0$, the spectral radius of $\cL_{\widehat\bF}$, when acting on $\bW_{m-1}\cap \cB^{p,q}_*$, $p,q\geq 2m$, is bounded by $\bar\nu^{2m-1}$.
\end{sublem}
\begin{proof}
Setting $A(x,s)=(d-bs)^{-1}$ and $\cL_{\widehat\bF,k}g=A^{2k}\cL_{\widehat\bF}g$, by equation \eqref{eq:deriv-s}
\[
\begin{split}
\partial_s\cL_{\widehat\bF,m}^ng&=(d-bs)^{-2}\cL_{\widehat\bF,m}\partial_s\cL_{\widehat\bF,m}^{n-1} g+b(2m-1)(d-bs)^{-1}\cL_{\widehat\bF,m}^n g\\
&=\cL_{\widehat\bF,m+1}^n\partial_s g+b(2m-1)\sum_{k=0}^{n-1}\cL_{\widehat\bF,m+1}^k A \cL_{\widehat\bF,m}^{n-k}g.
\end{split}
\]
Iterating the above formula yields
\[
\begin{split}
\left|\int_{-r}^r\left[\partial_u^j\partial_s^k(\cL_{\widehat\bF}^nh)\right](x+\hat V t) dt\right|
&\leq \bar\nu^{(2k-1)n}\left|\int_{-r}^r \left[\partial_u^j\cL^nh_{\bar s,k})\right](x+\hat V t) dt\right|\\
&\phantom{\leq}+C_k\sum_{l=0}^{k-1}\bar\nu^{(2l-1)n}\left|\int_{-r}^r \left[\partial_u^j\cL^nh_{\bar s,l})\right](x+\hat V t) dt\right|
\end{split}
\]
where $h_{\bar s,k}(x)=(\partial_s^k h)(x,\bar s)$ and $\cL g(x)=g\circ F^{-1}(x)$ is the transfer operator associated to the SRB measure for $F$ (which, of course, it is just Lebesgue).
It is then natural to consider the Banach spaces $\cB_0^{p,q}$ obtained by closing $\cC^\infty(\bT^2,\bC)$ with respect to the norms
\[
\|g\|_{\cB^{p,q}_0}=\sup_{\substack {x\in \bT^2\\r\in[\delta/2,\delta]}}\;\sup_{p_u\leq p}\;\sup_{\|\vf\|_{\cC_0^{q+p_u}}\leq 1}
\left|\int_{-r}^r\left[\partial_u^{p_u} g\right](x+\hat V t) \vf(t) dt\right|.
\]
Note that $\|h_{\bar s,k}\|_{\cB^{p,q}_0}\leq \|h\|_{\cB^{p,q}_*}$ if $k\leq p$.
Then
\begin{equation}\label{eq:herewego}
\|\cL_{\widehat\bF}^n h\|_{\cB^{p,q}_*}\leq \bar\nu^{(2p-1)n}\|\cL^n h_{\bar s,p}\|_{\cB^{p,q}_0}+C_p\sum_{j=0}^{p-1}\bar\nu^{(2j-1)n}\|\cL^n h_{\bar s,j}\|_{\cB^{p,q}_0}.
\end{equation}
It is well known that the point spectrum of the operator $\cL$ on each $\cB^{p,q}_0$ consists of just $\{1\}$, with the left eigenvector being Lebesgue measure, while the essential spectrum is bounded by $\bar\nu^{\min\{p,q\}}$.\footnote{ The first assertion follows by noticing that the decay of correlations is super exponential for analytic functions (just compute in Fourier transform), the second can be seen by the same arguments (actually, even simpler) we used for the spectrum of $\cL_{\widehat\bF}$.}

The above readily implies that if $h\in \bW_0$, then, by \eqref{eq:kernels}, $h_{\bar s,0}$ is of zero average and, for each $\beta>\bar\nu$ and $q\geq 2$,
\[
\|\cL_{\widehat\bF}^n h\|_{\cB^{2,q}_*}\leq C_\beta \bar\nu^{n}\beta^{n}\| h_{\bar s,0}\|_{\cB^{2,q}_0}\leq C_\beta \bar\nu^{n}\beta^{n}\| h\|_{\cB^{2,q}_*}.
\]
Thus the spectral radius of $\cL_{\widehat\bF}$  on $\bW_0$ must be less than $\bar\nu$, and since $\bW_0$ has codimension one it follows that $\cL_{\widehat\bF}$ can have at most one eigenvalue larger than $\bar \nu$. By our knowledge of the essential spectral radius and  \cite[Lemma A.1]{BaladiTsujii08} it follows that on each $\cB^{p,q}_*$, $p,q\geq 1$, the operator can have at most one eigenvalue larger that $\bar \nu$. 

To bound the spectral radius of $\cL_{\widehat\bF}$, when acting on $\bW_{m-1}\cap \cB^{p,q}_*$, $p,q\geq 2m$, we proceed by induction. We already proved the statement for $m=1$. Suppose it is true for $m$, then, for $q\geq 2m+2$, equation \eqref{eq:herewego} implies, for $h\in\bW_m$ and $\beta>\bar\nu$,
\[
\|\cL_{\widehat\bF}^n h\|_{\cB^{2m+2,q}_*}\leq C_\beta \beta^{(2m+1)n}\sum_{j=0}^{2m+2}\| h_{\bar s,j}\|_{\cB^{2m+2,q}_0},
\]
which proves the claim. 
\end{proof}
We are finally ready to conclude: by Sub-Lemma \ref{sublem:essential0} and Proposition \ref{th:base} we know that the essential spectrum of $\cL_{\widehat\bF}$ acting on $\cB^{p,q}_*$ and on $\cB^{p,q}$, $p\geq q$, are both bounded by $\bar\nu^{p-1}$. Hence \cite[Lemma A.1]{BaladiTsujii08} implies that $\sigma_{\cB^{p,q}_*}(\cL_{\widehat\bF})\cap D_p=\sigma_{\cB^{p,q}}(\cL_{\widehat\bF})\cap D_p$. On the other hand Sub-Lemma \ref{sublem:small-ess} implies that there can be at most $m$ eigenvalues larger than $\bar\nu^{2m-1}$, provided $p> 2m$. Finally Sub-Lemma \ref{lem:discrete-sp} shows that are at least $m$ eigenvalues, implying the Lemma.
\end{proof}
\begin{rem}\label{rem:useless-eig}
Note that, as remarked just before Theorem \ref{thm:main}, the projection $\bpi_*$ is not one-one. In the present case $O_k=\bpi_*{\bell}_k$ are all zero apart when $k=0$ in which case it is the invariant measure of $\phi_t$ (see equation \eqref{eq:kernels}). Hence as we already saw, there are no non trivial obstructions. It is however interesting that this does not imply that the spectrum of $\cL_\bF'$ consists only of zero and $e^{\htop}$.
\end{rem}
\subsubsection{\bfseries Cohomological equation}\label{subsec:cob}\ \newline 
To study the regularity of the coboundary we consider separately the various contributions in equation \eqref{eq:deriv-second}. The first contribution is of the same type of the previous section, only now the test function is much bigger, however we have seen that if $g$ is of zero average, then this term is exponentially small.

To analyze the second contribution we must analyze the operator $\widehat\cL_{\bF} $ defined in \eqref{eq:new-new}. Doing the same type of conjugation than before we can reduce the problem to studying the operator $\widehat\cL_{\widehat\bF} =\Xi\,\theta^*\widehat \cL_{\bF}\,(\theta^*)^{-1}\Xi^{-1}$.\footnote{ Of course, now $\theta^*$ is not the composition operator but rather the pushforward  of one forms while $\Xi$ is a again a multiplication operator but now acting on forms.} A direct computation shows that
\begin{equation}\label{eq:new-new-lin} 
\left[\widehat\cL_{\widehat\bF} \bgg\right]_{(\xi,s)}= (d-bs) \left[(\widehat\bF^{-1})^*\bgg\right]_{(\xi,s)}.
\end{equation}
In this coordinates a one form reads $\bgg=\langle\balpha,  d\xi\rangle+\bbeta ds$. Thus we can identify one forms with vector functions $(\balpha, \bbeta): \bT^2\times [-\beta_1,-\beta_2]\to\bR^3$. If we do such an identification, another direct computation shows that we are reduced to studying the operator
\[
\left[\,\widehat\cL(\balpha, \bbeta)\right](\xi,s)=(d-bs)((D F^{-1})^*\balpha\circ \widehat \bF^{-1}(\xi,s), (d-bs)^{-2}\bbeta\circ \widehat \bF^{-1}(\xi,s)).
\]
In addition, we are interested only in forms such that $\bbeta\equiv 0$ (see \eqref{eq:deriv-second-0}). We are thus reduced to study the transfer operator
\begin{equation}\label{eq:heritis}
\left[\,\cL\balpha\right](\xi,s)=(d-bs) A^{-1}\balpha\circ \widehat \bF^{-1}(\xi,s).
\end{equation}
Since we can write $\balpha(\xi,s)=p_+(\xi,s)v^++p_-(\xi,s)v^-$, it follows that the eigenvalues must be of the form $\balpha(\xi,s)= p(\xi,s) v^{\pm}$ where $v^{-}=V$, $v^+=V^\perp$ are the eigenvectors of $A$. Calling $\lambda_\pm$ the corresponding eigenvalues we have that the eigenvalues of $\cL$ must be eigenvalues of
\[
\cL_\pm p=(d-bs) \lambda^{-1}_\pm p\circ \widehat \bF^{-1}.
\]
The above are simply multiples of the operator $\cL_{\widehat\bF}$ defined in \eqref{eq:linear-basic}. Note that $\lambda_+^{-1}=\lambda_-=\bar \nu<1$.
Accordingly, Lemma \ref{lem:spectrum-final} implies that, outside the essential spectrum, we have $\sigma\!\left(\widehat\cL_{\widehat\bF}\right)=\left\{\bar \nu^{2k-2}\right\}_{k\in\bN}$, where all the eigenvalues have multiplicity two, apart form the largest one which is simple. The eigenvalues that are possibly relevant are the three larger or equal to $\bar\nu$, however the projection of the corresponding eigendistributions that are not identically zero, when applied to $\nabla g$, yield $\operatorname{Leb}(\partial_{v^-} (\|V\|^{-1}g))=\operatorname{Leb}(\partial_{v^+}  (\|V\|^{-1}g))=0$. Hence, also this term is uniformly bounded.
To analyze the other terms note that $\partial_xDF=0=\partial_x\|V\|$, hence $\Gamma=p_0=p^*_{0}=B=E=0$ while $A=\bar\nu^2$, hence the remaining terms are identically zero.

It follows that no nontrivial obstructions exists and, if $\operatorname{Leb}(g)=0$, then $g$ is a Lipschitz coboundary (as we already knew from the simple Fourier transform computation).

%%%%%%%%%%%%%%%%%%%%%%%%%%%%%%%%%%%%%%%%%%%%%
\subsection{Some considerations on the general case}\ \newline\label{sec:non-lin-ex}
Given the previous discussion, a natural question is if there are or not cases in which non trivial obstructions exist. This is a difficult question to answer, here we content ourselves with the discussion of small perturbations of the linear case. See Remark \ref{rem:conjecture} for considerations on the non perturbative case.

We will see that small perturbations of the linear case do not have obstructions to non trivial growth of ergodic integrals or to $\cC^{\frac 12-\epsilon}$ coboundary. Hence, for a small enough perturbation, an ergodic integral either grows linearly in time or, if the function is of zero average, the function is, at least, a $\cC^{\frac 12-\epsilon}$ coboundary (see Lemma \ref{lem:fuck} and Corollary \ref{cor:cob}).

\subsubsection{\bfseries A one-parameter family of examples}\label{sec:modelp}\ \newline 
Let us be more concrete: consider the symplectic maps studied in \cite{LM, Li04}:
\begin{equation}\label{eq:family}
F_\param(x,y)=\left(2x+y-\param\vf(x), x+y-\param\vf(x)\right),
\end{equation}
where  $\vf\in\cC^\infty(\bT,\bR)$, $\int_{\bT}\vf=0$, and $\param\geq 0$. For example one can choose $\vf(x)=\frac{1}{2\pi}\sin2\pi x$.
Also choose $\|V\|=1$ and call $V_\param$ the vector field.
Note that $F_0$ is the linear total automorphism discussed in the previous section. 

Using the same co-ordinates as in the previous section we can reduce ourselves to the study of the map
\begin{equation}\label{eq:alfamap}
\begin{split}
&\widehat\bF_\param (x,y,s)=(F_\param(x,y), \psi_\param(x,s)))\\
&\psi_\param(x,s)=\frac {1-\param\vf'(x)+s}{2-\param\vf'( x)+s},
\end{split}
\end{equation}
and of the transfer operator $\cL_{\widehat \bF_\param}=\Xi\, \theta^*\cL_{\bF_\param} (\theta^*)^{-1}\,\Xi^{-1}$ which reads
\begin{equation}\label{eq:L-reduced}
\cL_{\widehat\bF_\param} \bg(x,y,s)= (1-s)\,\bg\circ \widehat\bF_\param^{-1}(x,y,s).
\end{equation}
For small $\param$ the above operator is a perturbation of $\cL_{\widehat\bF_0}$, the spectrum of which we have computed. So by the perturbation theory in \cite{KellerLiverani99} it will have eigenvalues close to the ones of $\cL_{\widehat\bF_0}$. 
Hence, the second eigenvalue of $\cL_{\bF_\param}$ will be close to $\bar \nu<1$, thus it will not have any influence on the growth of ergodic integrals. Accordingly, for small $\param$ it persists the conclusion that either a function has non zero average, and hence the ergodic integral grows like $t$, or it has zero average and then it is a continuous coboundary. This is a bit disappointing, yet it does provide non trivial information on the flow.

\begin{rem}\label{rem:conjecture}
Note that, for $\param\neq 0$ there is no obvious reason to expect that the projection of the eigendistributions are automatically trivial (as in the linear case). Hence, having ergodic integrals with a growth $t^\beta$, $\beta\not\in\{0,1\}$ seems to be related to having a transfer operator, associated to the measure of maximal entropy, with the second largest eigenvalue outside the disk of radius one. 
Note that, when $\param=1$ the map is no longer Anosov. In  fact, we have a map of the class studied in \cite{LM} where it is shown that the decay of correlations with respect to Lebesgue, is only polynomial. In particular, this shows that the Ruelle transfer operator (associated to Lebesgue) cannot have a spectral gap. This is suggestive, although the relevance of such a fact for the present context is doubtful since we are studying different operators.\footnote{ In fact, in \cite{BCV16} it is proven that the operator $\cL_{\bF_1}$ has a spectral gap. More in general, in the case of area preserving Anosov maps Giovanni Forni gave us an argument showing that there should not be obstructions to the boundedness of the ergodic averages, the general case is however unclear.} In conclusion, it might be possible to have non trivial obstructions to the boundedness of the ergodic integral, although no such example is currently known.\footnote{ Though some hope is given by the construction of generic examples, for the operator associated to the SRB measure, with spectrum different from $\{0,1\}$ by Alexander Adam \cite{AA}. See also the more recent results in \cite{ban-naud} based on special examples described in \cite{sli-ban-J} for which the spectrum can be explicitly computed.}
\end{rem}
Next, let us investigate the coboundary regularity. This amounts to understanding the growth of $\langle \bv,\nabla\overline H_{\param,T}(g)\rangle$ where, as in \eqref{def:eraverage},
\[
\oH_{\param,T}(g) (x)= -\int_{\bRpe} dt  \,\chi\circ \tau_{n_T}(x,t) g\circ \phi_{\param,t}(x).
\]
Almost all the rest of Section \ref{sec:non-lin-ex} is devoted to proving the following claim.
\begin{lem}\label{thm:lip}
For $r$ large enough, there exist $N_2\in\{2,3,4,5\}$, and distributions $\cO_{2,\param},\dots,\cO_{N_2,\param}\in(\cC^{r})'$ such that, if $g\in\cC^r$ belongs to the kernel of  $\cO_{1,\param}$ and $\cO_{2,\param}$, then
 \begin{equation}\label{eq:lip0}
\|\nabla\overline H_{\param,T}(g)(x)\|\leq \Const T^{1+\const\param},
 \end{equation}
while if $g$ belongs to the kernels of all the $\cO_{1,\param},\dots,\cO_{N_2,\param}$, then
 \begin{equation}\label{eq:lip}
\sup_{T\in\bRp}\sup_{x\in\bT^2}\|\nabla\overline H_{\param,T}(g)(x)\|<\infty,
 \end{equation}
 and $g$ is  a Lipschitz coboundary. 
 In addition, there exists $C_*>0$, and a non zero function $\Omega_*^\dagger\in\cC^0(\bT^2,\bR)$ such that, for all $\param\in (0,\param_0)$, $T\geq C_*\ln\param^{-1}$ and $\vf$ $\cC^2$-generic, we have that $\cO_{2,\param}$ is a measure and
\[
 \begin{split}
|\langle V^\perp,\nabla\overline H_{\param,T}(g)(x)\rangle|\geq &\param^2\left|\Omega_*^\dagger(x)\cO_{2,\param} (g)\right|\bar\nu^{-3n_T}e^{-\const\param n_T}-C_*\param\|g\|_{\cC^r}\bar\nu^{-n_T}e^{\const \param n_T}.
\end{split}
 \]
 \end{lem}
 \begin{proof}
 The first claims are just a particular instance of Theorem \ref{thm:maintwo}, apart from the fact that $N_2\in\{2,3,4,5\}$. In other words the distributions are the one constructed in the proof of Theorem \ref{thm:maintwo} while the fact that $N_2\leq 5$ follows by perturbation theory, Lemma \ref{lem:test-function} and formula  \eqref{eq:deriv-second}. Also the inequalities \eqref{eq:lip0}, \eqref{eq:lip} follow by equations \eqref{eq:deriv-second}, \eqref{eq:preliminary2}, the spectral analysis of $\cL_{\bF_0}, \widehat\cL_{\bF_0}$ carried out in sections \ref{sec:ergo}, \ref{subsec:cob} (which shows that at most one eigenvalue of $\cL_{\bF_0}$ and three eigenvalues of $\widehat\cL_{\bF_0}$ can contribute to the growth in \eqref{eq:lip}) and standard perturbation theory for transfer operators \cite{KellerLiverani99}. The assertion that $\cO_{2,\param}$ is a measure and the last inequality require much more work and will be proven in Section \ref{sec:obstruction} (more precisely, the first assertion follows from Remark \ref{rem:measure}, while the inequality follows from equation \eqref{eq:akbar3}). 
 \end{proof}
 \begin{rem} Note that the $\cO_{i,\param}$, which are defined in the proof of Theorem \ref{thm:maintwo}, could be trivial (this is the case for $\cO_{2,\param}$, see Lemma \ref{lem:fuck}). If they were all trivial, then the case $\param\neq 0$ would behave like the linear model: either the ergodic integral of $g$ grows linearly or $g$ is a Lipschitz coboundary. If this is the case or not it remains to be seen.
 \end{rem}
\begin{lem}\label{lem:fuck}
There exists $\param_0>0$ such that for all $|\param|<\param_0$ the obstruction $\cO_{2,\param}$ is invariant for the flow, hence it is identically zero or proportional to $\cO_{1,\gamma}$.
Moreover, for all $g \in\cC^r$, $\lim_{\param\to 0}\cO_{2,\param}(g)=0$. 
\end{lem}
\begin{proof}
Assume that $\cO_{2,\param}(g)\neq 0$ for some $g\in\cC^{1+\alpha}$. Then, for each $\ve>0$ there exists $g_\ve\in\cC^r$ such that $\|g-g_\ve\|_{\cC^0}\leq\Const \ve$, $\|g-g_\ve\|_{\cC^1}\leq\Const \ve^\alpha$ and $\|g\|_{\cC^r}\leq \Const\ve^{-r+1}$. By  \eqref{eq:deriv-second-00} and Lemma \ref{lem:test-function}
\[
|\langle V^\perp,\nabla\overline H_{\param,T}(g-g_\ve)(x)\rangle|\leq\Const \bar\nu^{-3 n_T}e^{\const\param n_T}\ve.
\]
Thus, by Lemma \ref{thm:lip},
\[
\begin{split}
|\langle V^\perp,\nabla\overline H_{\param,T}(g)(x)\rangle|\geq&\param^2\left|\Omega_*^\dagger(x)\cO_{2,\param} (g)\right|\bar\nu^{-3n_T}e^{-\const\param n_T}\\
&-C_*\param\ve^{-r+1}\bar\nu^{-n_T}e^{\const \param n_T}-\Const \bar\nu^{-3 n_T}e^{\const\param n_T}\ve.
\end{split}
\]
Choosing $\ve=e^{-\const\param n_T}$ shows that there exists $c_1>0$ such that
\begin{equation}\label{eq:lower-gen}
\begin{split}
|\langle V^\perp,\nabla\overline H_{\param,T}(g)(x)\rangle|\geq&\param^2\left|\Omega_*^\dagger(x)\cO_{2,\param} (g)\right|\bar\nu^{-3n_T}e^{-c_1\param n_T}\\
&-C_*\param\bar\nu^{-n_T}e^{\const \param n_T}-\Const \bar\nu^{-3 n_T}e^{-2c_1\param n_T}.
\end{split}
\end{equation}

Next, we show that equation \eqref{eq:lower-gen} implies that $g$ cannot be a $\cC^{1+\alpha}$ coboundary if $\cO_{2,\param}(g)\neq 0$.

Note that, differentiating \eqref{eq:deriv-firts} and doing a rough bound on the derivative yields
\[
\left\|D^2 \overline H_{\param,T}(g)\right\|\leq \Const \nu^{-4}_{\param,n_T}\leq \Const T^{4+\const \param}.
\]
Since $\bar\nu^{-3n_T}e^{-c_1\param n_T}\geq T^{3-\const\param}$, by equation \eqref{eq:lower-gen} it follows that there exists a ball of radius at least $C_g T^{-1-\const \param}$ in which the derivative of $\overline H_{\param,T}(g)$ in the $V^\perp$ direction, is larger than $\param^2C_g T^{3-\const\param}$, provided $T\geq C_g\ln\param^{-1}$. This implies the existence of a positive measure set of points $x_0,x_1$, $\|x_0-x_1\|=C_g T^{-3-\const \param}$ such that
\begin{equation}\label{eq:no-go}
|\overline H_{\param,T}(g)(x_0)-\overline H_{\param,T}(g)(x_1)|\geq C_{g,\param} T^{-\const \param}.
\end{equation}
To conclude it suffices to prove that \eqref{eq:no-go} is impossible if $g$ is a $\varpi$-H\"older coboundary with $\varpi> 1/3$. We argue by contradiction, assume that $g$ is a $\varpi$-H\"older coboundary, then there exists a function $h_\param\in\cC^{\varpi}$ such that $h_\param\circ \phi_{\param,t}-h_\param=\int_0^t g\circ \phi_{\param, s}ds$. Accordingly, by \eqref{def:eraverage} and \eqref{eq:param-tau}
\[
\begin{split}
\oH_{\param, T}(g) (x_0)-\oH_{\param, T}(g) (x_1)
=&h_\param(x_0)-h_\param(x_1)\\
&+\int_{\bRpe}h_\param\circ \phi_{\param,s}(x_0)\chi'\circ \tau_{n_T}(x_0,s)\nu_{\param, n_T}(\phi_{\param,s}(x_0))ds\\
&-\int_{\bRpe}h_\param\circ \phi_{\param,s}(x_1)\chi'\circ \tau_{n_T}(x_1,s)\nu_{\param, n_T}(\phi_{\param,s}(x_1)) ds.
\end{split}
\]
By \eqref{eq:tauder} and Lemma \ref{lem:true-parabolic}  it then follows
\[
|\oH_{\param, T}(g) (x_0)-\oH_{\param, T}(g) (x_1)|\leq C_{g,\param} T^{1+\const\param}T^{-3\varpi-\const\param}\leq C_{g,\param} T^{1-3\varpi+\const\param} 
\]
which, as announced, for $\param$ small enough, is incompatible with \eqref{eq:no-go}, provided $T$ is large enough.

It follows that for all $f\in\cC^2$, $\cO_{2,\param}(\langle V_\gamma,\nabla f\rangle)=0$, since $\langle V_\gamma,\nabla f\rangle$ is a $\cC^{1+\alpha}$ coboundary.
But this implies that $\cO_{2,\param}$ is an invariant measure for the flow $\phi_{\param,t}$. Since the invariant measure is unique, either $\cO_{2,\param}$ is identically zero or is proportional to $\cO_{1,\gamma}$.

The last assertion in the Lemma follows from \eqref{eq:sobeit}.
\end{proof} 
By the above Lemma it follows an interesting fact.
\begin{cor}\label{cor:cob}
If $g\in\cC^{r}$ and $\cO_{1,\param}(g)=0$, then $g$ is a $\cC^{\frac 12-\const\param}$ coboundary.
\end{cor}
\begin{proof}
By equation \eqref{eq:lip0} and Lemma \ref{lem:fuck} we have
\[
\|\nabla\overline H_{\param,T}(g)(x)\|\leq C_g T^{1+\const\param}.
\]
By equation \eqref{eq:c0conv}, calling $h_\param$ the coboundary, it follows, for each $T>0$,
\[
\|\oH_{\param, T}(g)-h_\param\|_{\cC^0}\leq C_g \bar\nu^{n_{T}} e^{\const\param n_T},
\]
since, in the present case, the spectral radius $\theta$ of $\cL_{\bF_\param}$, restricted to the kernel of $\cO_{1,\param}$, is given by $\bar\nu e^{\const\param}$. 
Accordingly,
\[
\begin{split}
|h_\param(x)-h_\param(y)|&\leq |\oH_{\param, T}(g)(x)-\oH_{\param, T}(g)(y)|+C_g T^{-1+\const\param n_T}\\
&\leq C_g|x-y|T^{1+\const\param}+C_g T^{-1+\const\param},
\end{split}
\]
and the claim follows by choosing $T=|x-y|^{-\frac 12}$.
\end{proof}

To complete the proofs of Lemmata \ref{thm:lip} and \ref{lem:fuck} we must first identify explicitly $\cO_{2,\param}$. To this end we start by studying the operators \eqref{eq:new-new} that, in the present context, read (recall \eqref{eq:new-to})
\begin{equation}\label{eq:op-last}
\begin{split}
&\cL_{\param,A}\bg=\Xi\, \theta^*\cL_{\bF_\param,\hAcc} (\theta^*)^{-1}\,\Xi^{-1}\bg=(1-s)^{-1}\,\bg\circ \widehat\bF_\param^{-1}(x,y,s)\\
&\left[\widehat\cL_{\widehat\bF_\param} \bgg\right]_{(x,y,s)}=\left[\Xi\, \theta^*\widehat\cL_{\bF_\param}  (\theta^*)^{-1}\,\Xi^{-1}\bgg\right]_{(x,y,s)}= (1-s) \left[(\widehat\bF_\param^{-1})^*\bgg\right]_{(x,y,s)}.
\end{split}
\end{equation}
As in Subsection \ref{subsec:cob} we can write $\widehat\cL_{\widehat\bF_\param}\langle \balpha, (dx,dy)\rangle=\langle \cL_\param\balpha, (dx,dy)\rangle$ where
\begin{equation}\label{eq:uffa}
\left[\,\cL_\param\balpha\right](x,y,s)=(1-s) \left[(D F_\param^*)^{-1}\balpha\right]\circ \widehat \bF_\param^{-1}(x,y,s).
\end{equation}
\begin{rem}\label{rem:apply} 
Note that we are interested in applying the above to the case in which $\bgg=\bpi^*dg$ (recall that $\bpi(x,v)=x$) for some function $g\in\cC^r(\bT^2,\bR)$. This means that, in \eqref{eq:uffa} we are interested in $\balpha=(\nabla g)\circ \bpi$.
\end{rem}

\subsubsection{\bfseries Some preliminary facts}\ \newline  
We are thus left with the task of studying the spectrum of $\cL_{\param,A}, \cL_\param$, for small $\param$. We will use the perturbation theory developed in \cite[Section 8]{GouezelLiverani06} which shows, in particular, that all the spectral data are differentiable in $\param$, but first we need to establish some facts and notations.

By direct computation we have, setting $\pi_x(x,y,s)=x$,
\begin{equation}\label{eq:bF-deriv}
\begin{split}
&\partial_\param\widehat\bF_\param(x,y,s)=-(\vf(x), \vf(x),\vf'(x) (2-\param\vf'(x)+s)^{-2})\\
&\phantom{\partial_\param\widehat\bF_\param(x,y,s)}
=-(\vf(x), \vf(x), \vf'(x)\partial_s\psi_\param(x,s))\\
&(D\widehat\bF_\param)^{-1}=\begin{pmatrix}
1&-1&0\\
-1+\param\vf'&2-\param\vf'&0\\
-\partial_x\psi_\param/\partial_s\psi_\param&\partial_x\psi_\param/\partial_s\psi_\param&1/\partial_s\psi_\param
\end{pmatrix}\\
&(D\widehat\bF_\param)^{-1}\partial_\param\widehat\bF_\param(x,y,s)=-(0,\vf(x),\vf'(x))\\
&\partial_\param\widehat\bF^{-1}_\param(x,y,s)=-\left[(D\widehat\bF_\param)^{-1}\partial_\param\widehat\bF_\param\right]\circ\widehat\bF_\param^{-1}(x,y,s)\\
&\phantom{\partial_\param\widehat\bF^{-1}_\param(x,y,s)}
=(0,\vf,\vf')\circ \pi_x\circ \widehat\bF_{\param}^{-1}(x,y,s)\,.
\end{split}
\end{equation}
In alternative, one can compute the above noticing that
\begin{equation}\label{eq:inver-F}\begin{split}
\bF_\param^{-1}(x,y,s)&=\left(x-y, 2y-x+\param\vf(x-y),\psi^-_\param(x,y,s)\right)\\
\psi^-_\param(x,y,s)&=\frac{2s-1+\param\vf'(x-y)(1-s)}{1-s}\;;\quad \partial_s \psi^-_\param(x,y,s)=(1-s)^{-2}.
\end{split}
\end{equation}
Hence, for each $\balpha=(\eta_1,\eta_2)\in\cC^\infty(\bT^2\times (-\beta_1,-\beta_2),\bR^2)$,
\begin{equation}\label{eq:Lderivative}
\begin{split}
\partial_\param \cL_\param\balpha&=(1-s)\left[(D F_\param^*)^{-1}\Upsilon\balpha\right]\circ \widehat\bF_{\param}^{-1}=\cL_\param\Upsilon\balpha\\
\Upsilon\balpha&=\left[\vf\circ \pi_x\partial_y\balpha+\vf'\circ \pi_x\partial_s\balpha+\vf'\circ \pi_x \langle e_2,\balpha\rangle e_1)\right].
\end{split}
\end{equation}

We are interested in the maximal eigenvalue $\mu_\param$ of $\cL_\param$, which we know to be simple, and the associated left and right eigenvectors $\bell_\param,\bh_\param$. 
Thus, for all $\balpha\in\cC^\infty(\bT^2\times (-\beta_1,-\beta_2),\bR^2)$, 
\begin{equation}\label{eq:hate} 
\cL_\param\bh_\param=\mu_\param \bh_\param\;\quad \bell_\param(\cL_\param\balpha)=\mu_\param\bell_\param(\balpha)
\end{equation}
and, recalling \eqref{eq:f-product}, 
\begin{equation}\label{eq:zero-status}
\begin{split}
&\mu_0=\bar\nu^{-2}\\
&\bh_0= V f(s)\\
&\bell_0(\balpha)=\int_{\bT^2}\langle V,\balpha(x,y,\bar s)\rangle dx dy.
\end{split}
\end{equation}
Moreover, calling $V_{+,\param}$, $\|V_{+,\param}\|=1$, the unstable distribution of $F_\param$, we have
\begin{equation}\label{eq:unstable-shit}
DF_\param^n V_{+,\param}=\nu^u_{\param,n} V_{+,\param}\circ F^n_\param.
\end{equation}
Since
\[
\langle DF_\param^* V_{+,\param}^\perp\circ F_\param , V_{+,\param}\rangle=\langle V_{+,\param}^\perp\circ F_\param, DF_\param V_{+,\param}\rangle=0,
\]
we have $DF_\param^* V_{+,\param}^\perp\circ F_\param=\tau  V_{+,\param}^\perp$ and, multiplying the latter relation by $ V_{\param}$ yields
\[
DF_\param^* V_{+,\param}^\perp\circ F_\param=\frac{\nu_{\param,1}\langle V_\param,  V_{+,\param}^\perp\rangle\circ F_\param}{\langle V_\param,  V_{+,\param}^\perp\rangle}V_{+,\param}^\perp.
\]
Thus
\begin{equation}\label{eq:FnT}
(DF_\param^{*})^{-1} V_{+,\param}^\perp=\frac{\langle V_\param^\perp, V_{+,\param}\rangle}{\nu_{\param,1}\langle V_\param^\perp, V_{+,\param}\rangle\circ F_\param}V_{+,\param}^\perp\circ F_\param.
\end{equation}
Analogously,
\[
(DF_\param^{*})^{-1} V_{\param}^\perp=\frac{\langle V_\param^\perp, V_{+,\param}\rangle}{\nu^u_{\param,1}\langle V_\param^\perp, V_{+,\param}\rangle\circ F_\param}V_{\param}^\perp\circ F_\param.
\]
Note that we can always write  $\balpha=\frac1{\langle V_\param^\perp, V_{+,\param}\rangle}\left[a V_{\param}^\perp+b V_{\param,+}^\perp\right]$,
hence we have
\[
\cL_\param\balpha= \frac1{\langle V_\param^\perp, V_{+,\param}\rangle}\left[ V_{\param,+}^\perp\frac{1-s}{\nu_{\param,1}}b\circ \bF_\param^{-1}+V_{\param}^\perp \frac{1-s}{\nu^u_{\param,1}}a\circ \bF_\param^{-1}\right].
\]
Thus, setting\footnote{ The operators are well defined on $\cB^{1,\alpha}$, which suffices for the present purpose to characterise the maximal eigenvectors.}
\[
\begin{split}
&\cL_{0,\param} f=\frac{1-s}{\nu^u_{\param,1}}f\circ \bF_\param^{-1}\\
&\cL_{2,\param} f=\frac{1-s}{\nu_{\param,1}}f\circ \bF_\param^{-1},
\end{split}
\]
we have, for each $n\in\bN$,
\[
\cL_\param^n\balpha= \frac1{\langle V_\param^\perp, V_{+,\param}\rangle}\left[ V_{\param,+}^\perp\cL_{2,\param}^nb+V_{\param}^\perp \cL_{0,\param}^n a\right].
\]
Note that the maximal eigenvalue of $\cL_{0,0}$ is one, while the maximal eigenvalue of $\cL_{2,0}$ is 
$\mu_0=\bar\nu^{-2}$. Also note that the operators $\cL_{2,\param}$ are well defined on $\cB^{1,\alpha}$ and, on such spaces, it has essential spectral radius bounded by $\mu_\param\bar\nu^{\alpha} e^{\const\param}$. Thus, by perturbation theory, $\sigma_{\cB^{1,\alpha}}(\cL_{2,\param})\subset \{\mu_\param\}\cup\{z\in\bC\;:\;|z|\leq \mu_\param\bar\nu^{\alpha} e^{\const\param}\}$.

Let us set $\bar s_\param\in\cC^{1+\alpha}(\bT^2,\bR_{<})$ so that $V_\param=(1,\bar s_\param)[1+\bar s_\param]^{-\frac 12}$ and, for each $\psi\in\cC^0(\Omega,\bR)$, we define 
\begin{equation}\label{eq:barpi}
\bar\pi_\param \psi(x,y)=\psi(x,y,\bar s_\param(x,y)).
\end{equation}
 Then, 
\[
\|\cL_{2,\param}^n \psi-\cL_{2,\param}^n \bar\pi_\gamma \psi\|_{\cC^0}\leq \Const(\|\psi\|_{\cC^{1}} e^{\const \param n}).
\]
By the above considerations it follows
\[
\begin{split}
\bh_\param\bell_\param(\balpha)&=\frac1{\langle V_\param^\perp, V_{+,\param}\rangle}\mu_\param^{-n}V_{\param,+}^\perp\cL_{2,\param}^n(b)+\mu_\param^{-n}\cO(\|\balpha\|_{\cC^0} e^{\const \param n})\\
&=\frac1{\langle V_\param^\perp, V_{+,\param}\rangle}\mu_\param^{-n}V_{\param,+}^\perp\cL_{2,\param}^n(\bar\pi_\param b)+\mu_\param^{-n}\cO(\|\balpha\|_{\cC^0} e^{\const \param n})\\
&= \frac1{\langle V_\param^\perp, V_{+,\param}\rangle}h_{2,\param} V_{\param,+}^\perp\ell_{2,\param}(\bar\pi_\param b)+\bar\nu^{\alpha n}\cO(\|\balpha\|_{\cC^{1+\alpha}} e^{\const \param n}),
\end{split}
\]
where $h_{2,\param}, \ell_{2,\param}$ are the right and left maximal eigenvectors of $\cL_{2,\param}$, respectively normalised so that $\ell_{2,\param}(1)=\ell_{2,\param}(h_{2,\param})=1$. Taking the limit $n\to \infty$ it follows
\begin{equation}\label{eq:noncisicrede}
\begin{split}
& \bell_\param(\balpha)=\ell_{2,\param}\left(\bar\pi_\param \langle V_\param,\balpha\rangle\right)\\
&\bh_\param=\frac1{\langle V_\param, V_{+,\param}^\perp\rangle}V_{\param,+}^\perp h_{2,\param}.
 \end{split}
\end{equation}
\begin{rem}\label{rem:measure0}
Note that $\ell_{2,\param}$ is the maximal left eigenvalue of a transfer operator with a $\cC^{1+\alpha}$ potential. It follows from \cite[Lemma 4.9, 4.10]{GouezelLiverani08} (by applying them to the case $\iota=0$, $r=2+\alpha$, $p=1$ and $q=\alpha$) that $h_{2,\param}, \ell_{2,\param}$ are measures, thus the same is true for $\bh_\param, \bell_{\param}$.
\end{rem}
The above provides some global information, next we need to compute the derivatives with respect to $\param$ of the various object of interest. Recall that we have seen in Subsection \ref{subsec:cob} that, outside of the essential spectrum, $\sigma(\widehat\cL_{0})= \{\bar\nu^{2k-2}\}_{k\in \bN}$. On the other hand the spectrum of $\cL_{0,A}$ can be computed as in section \ref{sec:ergo} yielding $\sigma(\cL_{0,A})= \{\bar\nu^{2k+1}\}_{k\in \bN}$.

\subsubsection{\bfseries Perturbation theory}\label{sub:perturb}\ \newline 
Differentiating equations \eqref{eq:hate} and remembering \eqref{eq:Lderivative} yields
\begin{equation}\label{eq:first-order}
\begin{split}
&\partial_\param\mu_\param=\bell_\param\left(\partial_\param \cL_\param \bh_\param\right)\\
&\partial_\param\bh_\param=\sum_{k=0}^\infty \mu_\param^{-k-1}\cL_\param^k\left[\partial_\param\cL_\param \bh_\param-\bell_\param(\partial_\param\cL_\param \bh_\param)\bh_\param\right]\\
&\partial_\param\bell_\param(\balpha)=\bell_\param\left(\mu_\param^{-1}\partial_\param\cL_\param(\Id-\mu_\param^{-1}\cL_\param)^{-1}[\balpha-\bh_\param\bell_\param(\balpha)]\right)\\
&\phantom{\partial_\param\bell_\param(\balpha)}
=\sum_{k=0}^\infty \mu_\param^{-k}\bell_\param\left(\Upsilon \cL_\param^{k}[\balpha-\bh_\param\bell_\param(\balpha)]\right) .
\end{split}
\end{equation}
Remembering equation \eqref{eq:Lderivative}, \eqref{eq:zero-status} and since $f(\bar s)=1$ (see the line before \eqref{eq:f-product}), the first of \eqref{eq:first-order} yields 
\[
\partial_\param\mu_\param|_{\param=0}=\bell_0\left(\partial_\param \cL_\param|_{\param=0} \bh_0\right)
=\int_{\bT^2}\langle V,[\cL_0 \Upsilon Vf](x,y,\bar s)\rangle dx dy.
\]
Next, we can use the definition \eqref{eq:uffa} and \eqref{eq:Lderivative} again to obtain
\begin{equation}\label{eq:derivmu}
\partial_\param\mu_\param|_{\param=0}=\bar\nu^{-2}\int_{\bT^2}[\langle V, e_1\rangle\langle e_2,V\rangle +f'(\bar s)]\vf'(x) dx dy=0.
\end{equation}
Next, \eqref{eq:Lderivative} implies
\[
\partial_\param \cL_\param \bh_\param|_{\param=0}=\cL_\param\left[f'(\bar s)\vf' V+\vf' f(\bar s)\langle e_2, V\rangle A(1,-1)\right]\circ \widehat \bF^{-1}(x,y,s).
\]
Hence
\begin{equation}\label{eq:h1-palle}
\partial_\param \bh_\param|_{\param=0}=\sum_{k=0}^\infty\bar\nu^{-2k-2}\cL_0^{k+1}\left[f'(\bar s)\vf' V+\vf' \langle e_2, V\rangle e_1\right].
\end{equation}
On the other hand, for each $\balpha\in\cC^\infty(\Omega^*,\bR^2)$, $\Omega^*=\bT^2\times (-\beta_1,-\beta_2)$, such that $\bell_0(\balpha)=0$,
\[
\begin{split}
\partial_\param\bell_\param(\balpha)|_{\param=0}=&\sum_{n=0}^\infty\int_{\bT^2}\vf(x)\langle V, \partial_y(\bar\nu^{2n}\cL_0^n(\balpha))(x,y,\bar s)\rangle\\
&+\sum_{n=0}^\infty\int_{\bT^2}\vf'(x)\langle V, \bar\nu^{2n}\partial_s(\cL_0^n(\balpha))(x,y,\bar s)\rangle\\
& +\sum_{n=0}^\infty\int_{\bT^2}\vf'(x)\langle e_2,\bar\nu^{2n}\cL_0^n(\balpha)(x,y,\bar s)\rangle \langle V,e_1\rangle .
\end{split}
\]
Note that the terms in the first sum are all identically zero by integration by parts with respect to $y$. 
We must compute the $s$-derivative in the second term of the above equation
\begin{equation}\label{eq:s-deriv}
\begin{split}
&\partial_s\widehat\bF^{-1}_\param(x,y,s)=(0,0,(\partial_s\psi_\param)^{-1}\circ \widehat\bF^{-1}_\param)=(0,0,(1-s)^{-2})\\
&\partial_s(\cL_\param^n\balpha)=-(1-s)^{-1}\cL_\param^n\balpha+(1-s)^{-2}\cL_\param \partial_s(\cL_\param^{n-1}\balpha)\\
&\phantom{\partial_s(\cL_\param^n\balpha)}
=-\sum_{k=0}^{n-1}(\nu_*^{2}\cL_\param)^k\nu_*\cL_\param^{n-k}\balpha+(\nu_*^{2}\cL_\param)^n\partial_s\balpha,
\end{split}
\end{equation}
where we have set $\nu_*(s)=(1-s)^{-1}$ and the last formula can be checked by induction.
Accordingly,
\[
\begin{split}
\partial_\param\bell_\param(\balpha)|_{\param=0}
=&-\sum_{n=0}^\infty\frac{1-\bar\nu^{2n}}{1-\bar\nu^2}\bar \nu
\int_{\bT^2}\vf'(x)\langle V, \balpha\circ \bF_0^{-n}(x,y,\bar s)\rangle\\
&+\sum_{n=0}^\infty\bar \nu^{2n}\int_{\bT^2}\vf'(x)\langle V, (\partial_s\balpha)\circ \bF_0^{-n}(x,y,\bar s)\rangle\\
& +\sum_{n=0}^\infty\int_{\bT^2}\bar\nu^{1+n}\vf'(x)\langle A^{-n}e_2,\balpha\circ \bF_0^{-n}(x,y, \bar s)\rangle \langle V,(1,-1)\rangle.
 \end{split}
\]
In the following we are interested only in the case when the one form is given by $dg$, $g\in\cC^r(\bT^2,\bR)$. This yields  $\balpha(x,y,s)=\varsigma(s)\nabla g(x,y)$, for $\varsigma(s)=\sqrt{1+s^2}$. Thus the last term of the above formula becomes
\[
\begin{split}
&\sum_{n=0}^\infty\int_{\bT^2}\bar\nu^{1+n}\vf'(x)\varsigma(\bar s)\langle A^{-n}e_2,(\nabla g)\circ F_0^{-n}(x,y)\rangle \langle V,(1,-1)\rangle\\
&=\sum_{n=0}^\infty\int_{\bT^2}\bar\nu^{1+n}\vf'(x)\varsigma(\bar s)\langle e_2,\nabla (g\circ F_0^{-n})(x,y)\rangle \langle V,(1,-1)\rangle\\
&=\sum_{n=0}^\infty\int_{\bT^2}\bar\nu^{1+n}\vf'(x)\varsigma(\bar s)\partial_y (g\circ F_0^{-n})(x,y)\rangle \langle V,(1,-1)\rangle
\end{split}
\]
which is again zero by integration by part (in the $y$ variable). We are left with
\begin{equation}\label{eq:atlast}
\begin{split}
\partial_\param\bell_\param(\varsigma\nabla g)|_{\param=0}=&-\sum_{n=0}^\infty\frac{1-\bar\nu^{2n}}{1-\bar\nu^2}\bar\nu\varsigma(\bar s)\int_{\bT^2}\vf'\circ F_0^n(x,y)\langle V, \nabla g(x,y)\rangle\\
&+\sum_{n=0}^\infty\bar \nu^{2n}\varsigma'(\bar s)\int_{\bT^2}\vf'(x)\langle V, (\nabla g)\circ F_0^{-n}(x,y)\rangle\\
=&-\sum_{n=0}^\infty\bar\nu^{n}\langle V, e_2\rangle\langle V, e_1\rangle\int_{\bT^2}\vf''\circ F_0^n\cdot g,
\end{split}
\end{equation}
where, in the last line, we have used the algebraic identities $\bar\nu\varsigma(\bar s)=-(1-\bar\nu^2)\varsigma'(\bar s)$ and $\varsigma'(\bar s)=\langle V, e_2\rangle$.

\subsubsection{\bfseries Leading obstruction: an explicit formula}\ \newline \label{sec:existence}
We assume that $g$ is of zero average with respect to the invariant measure of $\phi_{\param,t}$. Accordingly, by Lemma \ref{lem:invar-mea}, $\varsigma g\circ\bpi$ belongs to the kernel of the largest left eigenvector $\ell_\param$ of $\cL_{\hat \bF_\param}$. Hence, by perturbation theory we have
\begin{equation}\label{eq:L-growth1}
\|\cL_{\hat \bF_\param}^ng\circ\bpi\|_{p,q}\leq \Const \bar\nu^n e^{\const n\param}\|g\|_{\cC^r}.
\end{equation}
Also, again by perturbation theory, for each $\bg\in\cB^{p,q}$,
\begin{equation}\label{eq:L-growth2}
\|\cL_{\param, A}^n\bg\|_{p,q}\leq \Const \bar\nu^n e^{\const n\param}\|\bg\|_{\cC^r}.
\end{equation}
Using the above facts, recalling \eqref{eq:deriv-second}, \eqref{eq:preliminary2} and Lemma \ref{lem:test-function}, and using perturbation theory on the second largest eigenvalues of $\hL_{\bF_\param}$, we can write
\begin{equation}\label{eq:akbar}
\begin{split}
 \langle \bv (x),& \nabla \oH_{\param, T}(g)(x) \rangle 
=\sum_{l=1}^{K_T}\bH^1_{F_\param^{l n_*}(x),\chi_*\bF^{l n_*}_{\param,*}(\bv,0)}\bigg[-\hL_{\bF_\param}^{l n_*}\bpi^*dg \\
&+ \sum_{j=0}^{l n_*-1}\hL_{\bF_\param}^{l n_*-j}((\cL_{\bF_\param}^j \bg)\cdot \hbomega_B)  \\
&-\sum_{j=0}^{l n_*-1}\sum_{m=0}^{l n_*-j-1} \hL_{\bF_\param}^{l n_*-j-m}(\cL_{\bF_\param,\hAcc}^m \hDcc \cL_{\bF_\param}^{j} \bg)\cdot \hbomega_\Gamma \bigg] \\
&+\cO(\|g\|_{\cC^r}\|\bv\|e^{\const \param n_T}).
\end{split}
\end{equation}

Next, note that, by \eqref{eq:new-new-lin}, \eqref{eq:uffa}, we have
\[
\hL_{\bF_\param}^{n}\bpi^*dg=\langle \left[ \frac 1\varsigma\cL_\param^n(\varsigma\nabla g)\right]\circ \theta^{-1},(dx,dy)\rangle.
\]
Accordingly, recalling \eqref{eq:hate}, \eqref{eq:L-growth1}, \eqref{eq:L-growth2} and setting $\bh_\param^*=\langle\left[\frac 1\varsigma \bh_\param\right]\circ\theta^{-1}, (dx,dy)\rangle$, we have
\begin{equation}\label{eq:akbar1}
\begin{split}
&\langle \bv (x), \nabla \oH_{\param, T}(g)(x) \rangle =
\sum_{l=1}^{K_T}\bH^1_{F_\param^{l n_*}(x),\chi_*\bF^{l n_*}_{\param,*}(\bv,0)}\left(\bh_\param^*\right)\mu_\param^{l n_*}\Theta_{l,\param}\\
&\phantom{\langle \bv (x), \nabla \oH_{\param, T}(g)(x) \rangle =}
+\cO(\|g\|_{\cC^r}\|\bv\|\bar\nu^{-n_T}e^{\const \param n_T})\\
&\Theta_{l,\param}=-\bell_\param(\varsigma (\nabla g)\circ\bpi)+ \sum_{j=0}^{l n_*-1}\mu_\param^{-j} \bell_\param(\hBcc\circ \theta(\cL_{\widehat \bF_\param}^j \varsigma g\circ\bpi)) \\
&\phantom{\Theta_{l,\param}=}
-\sum_{j=0}^{l n_*-1}\sum_{m=0}^{l n_*-j-1} \mu_\param^{-j-m} \bell_\param(\Gamma\circ \theta\cL_{\widehat\bF_\param,\hAcc\circ \theta}^m \hDcc\circ \theta \cL_{\widehat\bF_\param}^{j} \varsigma g\circ \bpi))).
\end{split}
\end{equation}
It is educative to write $\Theta_{l,\param}$ differently. This is done by using the identity, for $n\in\bN$,
\begin{equation}\label{eq:deriv-id}
\begin{split}
\nablas\cL_{\widehat\bF_\param}^n(\varsigma g)=&\cL_\param^n (\nablas (\varsigma g))+\param\sum_{k=0}^{n-1}\cL_\param^k\eta\cL_{\widehat\bF_\param}\vf''(\nu_*^2\cL_{\widehat\bF})^{n-1-k}\partial_s(\varsigma g)\\
&-\param\sum_{k=0}^{n-2}\sum_{j=k+1}^{n-1}\cL_\param^k\eta\cL_{\widehat \bF_\param}\vf''(\nu_*^2\cL_{\widehat\bF_\param})^{n-1-j}\nu_*\cL_{\widehat\bF_\param}^{j-k}\varsigma g, 
\end{split}
\end{equation}
where $\nablas g=(\partial_x g,\partial_y g)$ and $\bw=(1,-1)$. The identity \eqref{eq:deriv-id} can be checked by induction using \eqref{eq:s-deriv} and recalling \eqref{eq:uffa}, $\nu_*(s)=(1-s)^{-1}$. Applying $\bell_\param$ to \eqref{eq:deriv-id} yields, for all $n\in\bN$,
\[
\begin{split}
\bell_\param(\varsigma (\nabla g)\circ\bpi)=&\mu_\param^{-n}\bell_\param(\nablas\cL_{\widehat\bF_\param}^n(\varsigma g))-\param\sum_{k=0}^{n-1}\mu_\param^{-k-1}\bell_\param\left(\bw\cL_{\widehat\bF_\param}\vf''(\nu_*^2\cL_{\widehat\bF_\param})^{k}\varsigma' g\right)\\
&+\param\sum_{k=0}^{n-1}\sum_{j=0}^{k-1}\mu_\param^{-k-1}\bell_\param(\bw\cL_{\widehat \bF_\param}\vf''(\nu_*^2\cL_{\widehat\bF_\param})^{j}\nu_*\cL_{\widehat\bF_\param}^{k-j}\varsigma g)
\end{split}
\]
We can then write
\begin{equation}\label{eq:slick}
\begin{split}
&\Theta_{l,\param}=-\mu_\param^{-l n_*}\bell_\param(\nablas\cL_{\widehat\bF_\param}^{l n_*}(\varsigma g))+\Theta_{l,\param}^*\\
&\Theta_{l,\param}^*=\param\sum_{k=0}^{l n_*-1}\mu_\param^{-k-1}\bell_\param\left(\bw\cL_{\widehat\bF_\param}\vf''(\nu_*^2\cL_{\widehat\bF_\param})^{k}\varsigma' g\right)\\
&\phantom{\Theta_{l,\param}^*=}-\param\sum_{k=0}^{l n_*-1}\sum_{j=0}^{k-1}\mu_\param^{-k-1}\bell_\param(\bw\cL_{\widehat \bF_\param}\vf''(\nu_*^2\cL_{\widehat\bF_\param})^{j}\nu_*\cL_{\widehat\bF_\param}^{k-j}\varsigma g)\\
&\phantom{\Theta_{l,\param}=}
+ \sum_{j=0}^{l n_*-1}\mu_\param^{-j} \bell_\param(\hBcc\circ \theta(\cL_{\widehat \bF_\param}^j \varsigma g\circ\bpi)) \\
&\phantom{\Theta_{l,\param}^*=}
-\sum_{j=0}^{l n_*-1}\sum_{m=0}^{l n_*-j-1} \mu_\param^{-j-m} \bell_\param(\Gamma\circ \theta\cL_{\widehat\bF_\param,\hAcc\circ \theta}^m \hDcc\circ \theta \cL_{\widehat\bF_\param}^{j} \varsigma g\circ \bpi))).
\end{split}
\end{equation}
Note that
\[
| \Theta_{\infty,\param}^*-\Theta_{l,\param}^*|\leq \Const \|g\|_{\cC^r}\mu_\param^{l n_*}\bar\nu^{l n_*}e^{\const\gamma l n_*}
\]
Thus we can finally identify $\cO_{2,\param}$ and rewrite \eqref{eq:akbar1} as
\begin{equation}\label{eq:akbar2}
\begin{split}
&\langle \bv (x), \nabla \oH_{\param, T}(g)(x) \rangle =\param
\sum_{l=1}^{K_T}\bH^1_{F_\param^{l n_*}(x),\chi_*\bF^{l n_*}_{\param,*}(\bv,0)}\left(\bh_\param^*\right)\mu_\param^{l n_*}\cO_{2,\param}\\
&\phantom{\langle \bv (x), \nabla \oH_{\param, T}(g)(x) \rangle =}
+\cO(\|g\|_{\cC^r}\|\bv\|\bar\nu^{-n_T}e^{\const \param n_T})\\
&\cO_{2,\param}(g)=\param^{-1}\Theta_{\infty,\param}^*.
\end{split}
\end{equation}
\begin{rem}\label{rem:measure}
Note that, by Remark \ref{rem:measure0}, $\cO_{2,\param}$ is a measure for each $\param$. We will see shortly that it does not blow up for $\gamma\to 0$.
\end{rem} 
Equation \eqref{eq:akbar1} shows the form of the obstruction. However one must rule out a nasty possibility: $\bH^1_{F^{l n_*}(x),\chi_*\bF^{l n_*}_*(\bv,0)}(\bh^*_\param)$ could be identically zero. To dismiss such a conspiracy we need to better understand the flow derivative. We take the opportunity for an interesting and useful digression.

\subsubsection{\bfseries Interlude: Truly parabolic}\label{sec:trulyp}\ \newline  
The next result, a refinement of Lemma \ref{lem:parabolic} adapted to the present context, shows that our flows are typically not elliptic. 
\begin{lem}\label{lem:true-parabolic} 
For each $t\in\bRp$ we have\footnote{ Here, and in the following, we use the quantum mechanical notation $|v\rangle\langle w|$ for the tensor product $v\otimes w$ as we find it more convenient.}
\begin{equation}\label{eq:flowderrep}
D_\xi\phi_{\param,t}=|V_\param\circ \phi_{\param,t}\rangle\langle V_\param|+a_\param(\xi,t)|V_\param\circ \phi_{\param,t}\rangle\langle V_\param^\perp|+ b_\param(\xi,t)|V^\perp_\param\circ\phi_{\param,t}\rangle\langle V_\param^\perp|
\end{equation}
where, for some $C_0,b_*, t_*>0$ and for all $t\geq t_*$,
\[
\begin{split}
&C_0 t^{-b_*\param}\leq b_\param(\xi,t)\leq C_0^{-1} t^{b_*\param}\\
& |a_\param(\xi,t)|\leq \param C_0^{-1} t^{1+b_* \param}. 
\end{split}
\]
Moreover, provided $\param$ is small enough, generically\footnote{ In the $\cC^2$ topology in the set $\|\vf\|_{\cC^2}<1$.} there exists $c_*>0$ such that
\[
\sup_x\limsup_{t\to\infty} b_\param(x,t) t^{-c_*\param^2}=\infty.
\]
\end{lem}
\begin{proof}
We follow the logic of Lemma \ref{lem:parabolic} but using the special properties of the flows under consideration.
To this end note that since $\det DF_\param=1$, for all $m\in \bN$,
\begin{equation}\label{eq:derivativeF}
\begin{split}
D F_{\param}^m=&\nu_{\param,m}|V_{\param}\circ F_{\param}^m\rangle\langle V_{\param}|+\langle V_{\param}\circ F_{\param}^{m},D F_{\param}^{m} V_\param^\perp\rangle  |V_{\param}\circ F_{\param}^{m}\rangle\langle V_{\param}^\perp|\\
&+\nu_{\param,m}^{-1}|V_{\param}^\perp\circ F_{\param}^{m}\rangle\langle V_{\param}^\perp|\\
D F_\param^{-m}\circ F_\param^m=&\nu_{\param,m}^{-1}|V_{\param} \rangle \langle V_{\param}\circ F_{\param}^m|-\langle V_{\param}\circ F_{\param}^{m},D F_{\param}^{m} V_\param^\perp\rangle  |V_{\param}\rangle\langle V_{\param}^\perp\circ F_{\param}^{m}|\\
&+\nu_{\param,m}|V_{\param}^\perp\rangle\langle V_{\param}^\perp\circ F_{\param}^{m}|.
\end{split}
\end{equation}
By  \eqref{eq:uff} we can write\footnote{ Where we have used \eqref{eq:derivativeF} and that $D_{F_\param^m(x)}\phi_{\param,\tau_{m}(\xi,t)}V_\param\circ F_\param^m(\xi)=V_\param(F_\param^m(\phi_{\param,t}(\xi))$.}
\begin{equation}\label{eq:parabolic}
\begin{split}
&b_\param(\xi,t)=\nu_{\param,m}\circ\phi_{\param,t}\langle V_\param^\perp(F_\param^m\circ\phi_{\param,t}(\xi)), D_{F_\param^m(x)}\phi_{\param,\tau_{m}(\xi,t)}\cdot D_\xi F_\param^mV_\param^\perp(\xi)\rangle\\
&\phantom{b}=\frac{\nu_{\param,m}\circ \phi_{\param,t} (\xi)}{\nu_{\param,m}(\xi)}\langle V_\param^\perp( F_\param^m\circ\phi_{\param,t}(\xi)), D_{F_\param^m(\xi)}\phi_{\param,\tau_{m}(\xi,t)}V_\param^\perp(F_\param^m(\xi))\rangle.
\end{split}
\end{equation}
By the arbitrariness of $m$ it follows
\begin{equation}\label{eq:b-def}
b_\param(\xi,t)=\lim_{m\to\infty}\frac{\nu_{\param,m}\circ \phi_{\param,t} (\xi)}{\nu_{\param,m}(\xi)}.
\end{equation}
Also, for future use, note that \eqref{eq:flowderrep} implies
\[
b_\param(\xi,t)=\det(D\phi_t).
\]
For $m_t=\Const\frac{\ln t}{\htop}$ and $s\in [0,t]$ we have, by  \cite[Lemma C.3]{GLP13},\footnote{ Recall that, by structural stability, the topological entropy is constant for the family $F_\param$.} 
\[
\|F_\param^{m_t}(\xi)-F_\param^{m_t}\circ \phi_{\param,s}(\xi)\|\leq\Const st^{-3}.
\]
Accordingly, provided $s\in [0,t]$, $t\geq 1$,
\begin{equation}\label{eq:b-asy}
b_\param(\xi,s)=\prod_{j=0}^{\infty}\frac{\nu_{\param,1}\circ F_\param^j\circ \phi_{\param,s} (\xi)}{\nu_{\param,1}(F_\param^j\xi)}=\frac{\nu_{\param,m_t}\circ \phi_{\param,s} (\xi)}{\nu_{\param,m_t}(\xi)}\left[1+\cO(t^{-2})\right].
\end{equation}
The above formula implies, for $\param$ small enough,
\begin{equation}\label{eq:b-bound}
\Const t^{-\const \param/\htop}\leq |b_\param(\xi,t)|\leq e^{\const \param m_t}\leq\Const t^{\const \param/\htop}\leq \Const \sqrt t.
\end{equation}
One might expect that $b_\param$ oscillates in time, however it cannot be always small. To see this consider
\[
\cE^*_t=\int_0^tb_\param(\xi,s)=\int_0^t\frac{\nu_{\param,m_t}\circ \phi_{\param,s} (\xi)}{\nu_{\param,m_t}(\xi)} ds+\cO(1/\sqrt t).
\]
 For each $\xi\in\bT^2$ let $m^*_t(\xi)$ be an integer so that $\tau_{m^*_t(\xi)}(\xi,t)\in (c_*, c_*^{-1})$. Since \cite[Lemma C.3]{GLP13} implies that $e^{\htop m^*_t(\xi)}\geq \Const t$,\footnote{ Here $\htop$ is the topological entropy of $F_\param^{-1}$, which coincides with the topological entropy of $F_\param$ (since two trajectories are $\ve$-separated for $F_\param$ iff they are $\ve$-separated for $F_{\param}^{-1}$).} it follows that we can choose $m^*_t(\xi)=m^*_t$, independent on $\xi$, provided $c_*$ has been chosen small enough. Then, recalling \eqref{eq:param-tau},
\[
\begin{split}
\cE^*_t(\xi)&\geq e^{-\const}\int_0^t \frac{\nu_{\param,m^*_t}\circ \phi_{\param,s}(\xi)}{\nu_{\param,m^*_t}(\xi)}ds+\cO(1)\geq \frac{\Const }{\nu_{\param,m^*_t}(\xi)}+\cO(1)\\
&\geq \Const t\frac{e^{-\htop m^*_t}}{\nu_{\param,m^*_t}(\xi)}+\cO(1).
\end{split}
\]
On the other hand, if $\overline\mu_{e,\param}$ is the (invariant) measure of maximal entropy of $F_\param$ generically, for $\param>0$, it will not be the SRB measure and hence, by Ruelle inequality, \cite[Theorem 1.5]{LY85} and structural stability we have\footnote{\label{foo:nuu} The first equality follows from the fact that $F_\param$ is area preserving and the invariance of $\overline\mu_{e,\param}$. Indeed, let $\theta_\param$ be the angle between $V_\param, V_{+,\param}$, that is $\|V_\param\wedge V_{+,\param}\|=|\sin \theta_\param|$, then
\[
|\sin\theta_\param|=\|F_\param^*( V_{\param}\circ F_{\param}^{-1}\wedge V_{+,\param}\circ F_{\param}^{-1})\|=\|(DF_\param V_{\param})\wedge (DF_\param V_{+,\param})\|=\nu_{\param,1}\nu^u_\param|\sin\theta_\param\circ F_\param|.
\]
}
\[
\overline\mu_{e,\param}(\ln\nu_{\param,1}^{-1})=\overline\mu_{e,\param}(\ln\nu^u_{\param})>\htop=\overline\mu_{h,0}(\ln\nu_{0,1}^{-1}).
\]
In fact, using perturbation theory (as in the proof of \cite[Proposition 8.1]{GouezelLiverani08}, but doing a painful second order computation), one can get the following more precise result which (rather lengthy) proof we momentarily postpone.
\begin{sublem}\label{sublem:second} 
There exists $\param_0,\bar c>0$ such that, for $|\param|\leq \param_0$,
\[
\overline\mu_{e,\param}(\ln\nu_{\param,1}^{-1})\geq \htop+\bar c\param^2.
\]
\end{sublem}
\noindent Then, by Birkhoff theorem, for  $\overline\mu_{e,\param}$-almost all $\xi\in\bT^2$, provided $t$ is large enough,
\begin{equation}\label{eq:realgrowth}
\cE^*_t(\xi)\geq \Const t^{1+\const \param^2},
\end{equation}
from which the last statement of the Lemma follows.
To study $a_\param$, let us set $\zeta(t)=D\phi_{\param,t}$. By the smooth dependence with respect to the initial conditions and recalling \eqref{eq:V-deriv}, we have 
\begin{equation}\label{eq:deriv-cont}
\dot\zeta(t)=\left[V_\param^\perp\otimes p_\param\right]\circ\phi_{\param,t}\, \zeta(t),
\end{equation}
hence $\langle V_\param,\dot \zeta V_\param^\perp\rangle=0$ and, differentiating our representation of $\zeta$,
\begin{equation}\label{eq:aderivative}
\dot a_\param(t)=\langle V_\param, p_\param\rangle\circ \phi_{\param,t} \,b_\param(t),
\end{equation}
from which the wanted bound follows by \eqref{eq:b-bound} and integrating.
\end{proof}
\begin{proof}[{\bf Proof of Sub-Lemma \ref{sublem:second}}]
By \cite[Theorem 6.4]{GouezelLiverani08} it follows that the transfer operator associated to the measure of maximal entropy is exactly $\cL_{\bF_\param} $.\footnote{ Recall that the definition of the space here is a bit different with respect to \cite{GouezelLiverani08}. However the two spaces, and hence the two transfer operators, are related by the continuous isomorphism \eqref{eq:conjugate}.} Hence, recalling \eqref{eq:L-reduced},  and calling $\ell_\param$, $h_\param$ the left and right eigenvector of $\cL_{\widehat\bF_\param}$ associated to the maximal eigenvalue $\tilde\mu_\param$ we have, for all $g\in\cC^0(\bT^2,\bR)$,
\[
\overline\mu_{e,\param}(g)=\ell_\param(g h_\param).
\]
In addition \footnote{ The second equality follows from structural stability.} 
\begin{equation}\label{eq:strut}
\tilde\mu_\param=e^{\htop(F_\param)}=e^{\htop(F_0)}=\bar\nu^{-1}=:\tilde \mu.
\end{equation}
 Also, by \cite[equation (8.10)]{GouezelLiverani08} and using the relation between the eigenvectors of $\cL_{\bF_\param}$ and $\cL_{\widehat\bF_\param}$, we have
\begin{equation}\label{eq:tocompute}
\begin{split}
\overline\mu_{e,\param}(\ln\nu_{\param,1}^{-1})&=\ell_\param(h_\param\ln(1-s))+\ell_\param(h_\param\ln\varsigma\circ\widehat\bF_\param^{-1})-\ell_\param(h_\param\ln\varsigma)\\
&=\ell_\param(h_\param\ln(1-s)),
\end{split}
\end{equation}
where we have used the invariance of the measure $\ell_\param(h_\param\,\cdot)$ with respect to $\widehat\bF_\param$.
Recall, from section \ref{sec:ergo}, that 
\begin{equation}\label{eq:recap0}
\begin{split}
&h_0(x,y,s)=f(s)=\prod_{k=0}^\infty \bar\nu (1-\bar s+\psi^{-k}(\bar s)-\psi^{-k}(s))\\
&\ell_0(\vf)=\int\vf(x,y, \bar s) dx dy\\
&\psi^{-1}(s)=\frac{ 2s-1}{1-s}\,; \quad1-\bar s=\bar\nu^{-1}.
\end{split}
\end{equation}
Clearly $f(\bar s)=1$, also, for further reference, note that a direct computation yields\footnote{ Note that
\[
f'(s)=-f(s)\sum_{k=0}^\infty\frac{\bar\nu\prod_{j=0}^{k-1}[\nu_*\circ \psi^{-j}(s)]^2}{1+\bar\nu\psi^{-k}(\bar s)-\bar\nu\psi^{-k}(s)}.
\]
}
\begin{equation}\label{eq:f-deriv}
\begin{split}
f'(\bar s)&= -\frac{\bar\nu}{1-\bar\nu^2} \\
f''(\bar s)&=f'(\bar s)^2-\sum_{k=0}^\infty\left\{\sum_{j=0}^{k-1} 2\bar\nu^{2k+2j+2}+\bar\nu^{4k+2}\right\}=0.
\end{split}
\end{equation}
Then, $\ell_0(h_0\ln(1-s))=-\ln\bar\nu=\htop$. As the expression on the right hand side of \eqref{eq:tocompute} is smooth in $\param$ (see \cite[Theorem 2.7]{GouezelLiverani06}), it suffices to compute the derivatives in zero.
Note that we can normalize $h_\param,\ell_\param$ so that $\ell_\param(h_\param)=1$ and $\ell_\param(\partial_\param h_\param)=0$. It follows $[\partial_\param\ell_\param](h_\param)=0$. 

We start with the analogous of \eqref{eq:first-order}:\footnote{ The last equality in the first line follows from \eqref{eq:strut}.}
\begin{equation}\label{eq:first-order-bis}
\begin{split}
&\partial_\param\tilde\mu_\param=\ell_\param(\partial_\param\cL_{\widehat\bF_\param} h_\param)=0\\
&\partial_\param h_\param=(\tilde\mu_\param-\cL_{\widehat\bF_\param})^{-1}\partial_\param\cL_{\widehat\bF_\param}h_\param\\
&\partial_\param\ell_\param(g)=\ell_\param\left(\partial_\param\cL_{\widehat\bF_\param}(\tilde\mu_\param-\cL_{\widehat\bF_\param})^{-1}[g-h_\param\ell_\param(g)]\right).
\end{split}
\end{equation}
Recalling  \eqref{eq:L-reduced}, \eqref{eq:bF-deriv} we have
\begin{equation}\label{LF-deriv}
\begin{split}
&\partial_\param\cL_{\widehat\bF_\param}g=(1-s)\langle \nabla g,(0,\vf,\vf')\circ \pi_x\rangle\circ \widehat\bF_{\param}^{-1}(x,y,s)
=\cL_{\widehat\bF_\param}\widehat \Upsilon g\\
&\widehat \Upsilon g=\langle \nabla g,(0,\vf,\vf')\circ \pi_x\rangle.
\end{split}
\end{equation}
It follows that
\begin{equation}\label{derivative-zero}
\begin{split}
&\partial_\param h_\param|_{\param=0}=(\tilde\mu-\cL_{\widehat\bF_0})^{-1}\partial_\param\cL_{\widehat\bF_0}h_0\\
&\partial_\param\ell_\param|_{\param=0}(g)=\ell_0\left(\partial_\param\cL_{\widehat\bF_0}(\tilde\mu-\cL_{\widehat\bF_0})^{-1}[g-h_0\ell_0(g)]\right).
\end{split}
\end{equation}
Thus, setting $\Gamma(\param)=\ell_\param(\ln(1-s)h_\param)$, we have $\Gamma(0)=\ln\bar\nu^{-1}$ and
\begin{equation}\label{eq:first-palla}
\begin{split}
\Gamma'(\param)=&\tilde\mu\ell_\param\left(\langle \nabla (\tilde\mu-\cL_{\widehat\bF_\param})^{-1}\left[\ln(1-s)-\Gamma(\param)\right] h_\param, (0,\vf,\vf')\rangle\right)\\
&+\ell_\param\left(\ln(1-s)(\tilde\mu-\cL_{\widehat\bF_\param})^{-1}\cL_{\widehat\bF_\param}\langle\nabla h_\param, (0,\vf,\vf')\rangle\right).
\end{split}
\end{equation}
Accordingly,
\[
\begin{split}
\Gamma'(0)=&\sum_{k=0}^\infty\ell_0\left(\langle \nabla \bar\nu^k\cL_{\widehat\bF_0}^k\left[\ln(1-s)+\ln\bar\nu\right] f(s),(0,\vf,\vf')\rangle\right)\\
&+\sum_{k=0}^\infty\ell_0\left(\ln(1-s)\bar\nu^{k+1}\cL_{\widehat\bF_0}^{k+1}\langle\nabla h_0, (0,\vf,\vf')\rangle\right)\\
=&\sum_{k=0}^\infty\bar\nu^k\ell_0\left( \partial_s\left\{\cL_{\widehat\bF_0}^k\left[\ln(1-s) +\ln\bar\nu\right]f(s)\right\}\vf'\right).
\end{split}
\]
Note that, using \eqref{eq:alfamap} and computing as in \eqref{eq:s-deriv} we have (recall that  $\nu_*(s)=(1-s)^{-1}$), for each $g\in\cC^1$ we have
\begin{equation}\label{eq:s-deriv2}
\partial_s(\cL_{\widehat\bF_\param}^k g)=-\sum_{j=0}^{k-1}(\nu_*^{2}\cL_{\widehat\bF_\param})^j\nu_*\cL_{\widehat\bF_\param}^{k-j}g 
+(\nu_*^{2}\cL_{\widehat\bF_\param})^k\partial_sg.
\end{equation}
Thus, setting $\widehat g(x,y,s)=\left[\ln(1-s) +\ln\bar\nu\right]f(s)$,
\begin{equation}\label{eq:s-deriv-go}
\ell_0\left( \vf'\partial_s\cL_{\widehat\bF_0}^k\widehat g\right)=
-\sum_{j=0}^{k-1}\bar \nu^{2j-k+1}\widehat g(\bar s)\int_{\bT^2}\vf'(x)+\bar\nu^k\partial_s\widehat g(\bar s)\int_{\bT^2}\vf'(x)=0.
\end{equation}
It follows 
\begin{equation}\label{eq:first-order-entro}
\Gamma'(0)=0.
\end{equation}
We must then compute the second derivative.

Differentiating \eqref{eq:first-palla} we have\footnote{ To simplify notation we abuse it and use $\ell'_0$ to signify $\partial_\param\ell_\param|_{\param=0}$ and similarly for $h_\param$ and $\cL_{\widehat\bF_\param}$.}
\begin{equation}\label{eq:Gamma2}
\begin{split}
&\Gamma''(0)=\tilde\mu\ell'_0\left(\langle \nabla (\tilde\mu-\cL_{\widehat\bF_0})^{-1}\left[\ln(1-s) +\ln\bar\nu\right]h_0, (0,\vf,\vf')\rangle\right)\\
&+\sum_{k=0}^\infty\sum_{j=0}^{k-1}\tilde\mu^{-k}\ell_0\left(\langle \nabla\cL_{\widehat\bF_0}^{k-j-1}\cL_{\widehat\bF_0}'\cL_{\widehat\bF_0}^j\left[\ln(1-s) +\ln\bar\nu\right]h_0, (0,\vf,\vf')\rangle\right)\\
&+\tilde\mu\ell_0\left(\langle \nabla (\tilde\mu-\cL_{\widehat\bF_0})^{-1}\left[\ln(1-s) +\ln\bar\nu\right]h'_0, (0,\vf,\vf')\rangle\right)\\
&+\ell_0'\left(\ln(1-s)(\tilde\mu-\cL_{\widehat\bF_0})^{-1}\cL_{\widehat\bF_0}\langle\nabla h_0, (0,\vf,\vf')\rangle\right)\\
&+\sum_{k=1}^\infty\sum_{j=0}^{k-1}\tilde\mu^{-k}\ell_0\left(\ln(1-s)\cL_{\widehat\bF_0}^{k-j-1}\cL_{\widehat\bF_0}'\cL_{\widehat\bF_0}^j \langle\nabla h_0, (0,\vf,\vf')\rangle\right)\\
&+\ell_0\left(\ln(1-s)(\tilde\mu-\cL_{\widehat\bF_0})^{-1}\cL_{\widehat\bF_0}\langle\nabla h'_0, (0,\vf,\vf')\rangle\right).
\end{split}
\end{equation}
Next we must compute, one by one, the above six terms. We will use \eqref{derivative-zero}, \eqref{LF-deriv}, \eqref{eq:s-deriv2}, \eqref{eq:s-deriv-go} and \eqref{eq:recap0}. Let us proceed in the order in which the terms appear
\[
\begin{split}
&\tilde\mu\ell_0\left(\cL_{\widehat\bF_0}'(\tilde\mu-\cL_{\widehat\bF_0})^{-1}\langle \nabla (\tilde\mu-\cL_{\widehat\bF_0})^{-1}\left[\ln(1-s) +\ln\bar\nu\right]h_0, (0,\vf,\vf')\rangle\right)\\
&=\tilde\mu^2\ell_0\left(\vf'\partial_s\left\{(\tilde\mu-\cL_{\widehat\bF_0})^{-1}\langle \nabla (\tilde\mu-\cL_{\widehat\bF_0})^{-1}\left[\ln(1-s) +\ln\bar\nu\right]h_0, (0,\vf,\vf')\rangle\right\}\right)\\
&=-\sum_{k=0}^\infty\sum_{j=0}^{k-1}\bar\nu^{k+1}\ell_0\left(\vf' (\nu_*^{2}\cL_{\widehat\bF_0})^j\nu_*\cL_{\widehat\bF_0}^{k-j}
\langle \nabla (\tilde\mu-\cL_{\widehat\bF_0})^{-1}\left[\ln(1-s) +\ln\bar\nu\right]h_0, (0,\vf,\vf')\rangle\right)\\
&\phantom{=}+\sum_{k=0}^\infty\bar\nu^{k+1}\ell_0\left(\vf' (\nu_*^{2}\cL_{\widehat\bF_0})^k\partial_s
\langle \nabla (\tilde\mu-\cL_{\widehat\bF_0})^{-1}\left[\ln(1-s) +\ln\bar\nu\right]h_0, (0,\vf,\vf')\rangle\right)\\
&=-\sum_{k=0}^\infty\sum_{j=0}^{k-1}\sum_{i=0}^\infty\bar\nu^{2j+i+1}\ell_0\left(\vf'\circ F_0^k 
\langle \nabla \cL_{\widehat\bF_0}^{i}\left[\ln(1-s)+\ln\bar\nu\right] h_0, (0,\vf,\vf')\rangle\right)\\
&\phantom{=}+\sum_{k=0}^\infty\sum_{i=0}^\infty\bar\nu^{2k+i}\ell_0\left(\vf'\circ F_0^k\partial_s
\langle \nabla \cL_{\widehat\bF_0}^{i}\left[\ln(1-s) +\ln\bar\nu\right]h_0, (0,\vf,\vf')\rangle\right).
\end{split}
\]
To continue note that
\begin{equation}\label{eq:useful?}
\begin{split}
&\ell_0\left(\vf'\circ F_0^k \langle \nabla \cL_{\widehat\bF_0}^{i}\left[\ln(1-s)+\ln\bar\nu\right] h_0, (0,\vf,\vf')\rangle\right)\\
&=-\tilde\mu^i\ell_0\left(\left[\partial_y(\vf'\circ F_0^k \vf)\right]\circ F_0^i
\left[\ln(1-s)+\ln\bar\nu\right] h_0\rangle\right)\\
&\phantom{=}+\ell_0\left(\vf'\circ F_0^k \vf'\partial_s \cL_{\widehat\bF_0}^{i}\left[\ln(1-s)+\ln\bar\nu\right] h_0\right)\\
&=-\ell_0\left(\vf'\circ F_0^k \vf'\bar\nu^{2i+1} \cL_{\widehat\bF_0}^{i} h_0\right)=-\bar\nu^{i+1}\mu_{e,0}(\vf'\circ F_0^k\vf').
\end{split}
\end{equation}
Also,
\[
\begin{split}
&\ell_0\left(\vf'\circ F_0^k\partial_s\langle \nabla \cL_{\widehat\bF_0}^{i}\left[\ln(1-s) +\ln\bar\nu\right]h_0, (0,\vf,\vf')\rangle\right)\\
&=\ell_0\left(\left\{\partial_y\left[\vf\cdot\vf'\circ F_0^k\right]\right\}(\bar\nu^2 \cL_{\widehat\bF_0})^{i}\partial_s\left[\ln(1-s) +\ln\bar\nu\right]h_0\right)\\
&\phantom{=}+\ell_0\left(\vf'\vf'\circ F_0^k\partial_s^2 \cL_{\widehat\bF_0}^{i}\left[\ln(1-s) +\ln\bar\nu\right]h_0\right).
\end{split}
\]
Next, remark that if $g(x,y,\bar s)=0$, then
\begin{equation}\label{eq:second-s}
\partial_s^2 \cL_{\widehat\bF_0}^{i}g(x,y,\bar s)=\bar \nu^{4i} \left[\cL_{\widehat\bF_0}^{i}\partial_s^2 g\right](x,y,\bar s),
\end{equation}
hence
\[
\begin{split}
&\ell_0\left(\vf'\circ F_0^k\partial_s\langle \nabla \cL_{\widehat\bF_0}^{i}\left[\ln(1-s) +\ln\bar\nu\right]h_0, (0,\vf,\vf')\rangle\right)\\
&=\bar\nu^{2+3i}\frac{1+\bar\nu^{2}}{1-\bar\nu^2}\bar\mu_{e,0}\left(\vf'\vf'\circ F_0^k\right).
\end{split}
\]
To further simplify the above expression, note the $\bar\mu_{e,0}$, when applied to functions that do not depend on $s$, is just Lebesgue measure. In addition, for each zero average function $g\in\cC^1(\bT^2,\bR)$ and $g_1\in\cC^1(\bT^2,\bR)$ such that $\partial_y g=\partial_y g_1=0$ we have that, for all $k>0$,
\begin{equation}\label{eq:no-cobo}
\int_{\bT^2} g\circ F_0^k \cdot g_1 =0.
\end{equation}
Indeed, calling $G$ the primitive of $g$, we have
\[
0=\int_{\bT^2}g_1\partial_y(G\circ F_0^k)=  \langle DF^k_0 e_2,e_1\rangle\int_{\bT^2}g_1 \cdot g\circ F_0^k,
\]
from which the claim follows since $ \langle DF^k_0 e_2,e_1\rangle\neq 0$.

Collecting all the above it follows
\begin{equation}\label{eq:first-term}
\begin{split}
&\tilde\mu\ell_0\left(\cL_{\widehat\bF_0}'(\tilde\mu-\cL_{\widehat\bF_0})^{-1}\langle \nabla (\tilde\mu-\cL_{\widehat\bF_0})^{-1}\left[\ln(1-s) +\ln\bar\nu\right]h_0, (0,\vf,\vf')\rangle\right)\\
&=\frac{\bar\nu^2}{(1-\bar\nu^2)^2}\Leb((\vf')^2).
\end{split}
\end{equation}
We can now compute the second term of \eqref{eq:Gamma2}. Recalling \eqref{LF-deriv}, \eqref{eq:s-deriv2} we have
\[
\begin{split}
&\ell_0\left(\langle \nabla\cL_{\widehat\bF_0}^{k-j-1}\cL_{\widehat\bF_0}'\cL_{\widehat\bF_0}^j\left[\ln(1-s) +\ln\bar\nu\right]h_0, (0,\vf,\vf')\rangle\right)\\
&=\ell_0\left(\vf' \partial_s\cL_{\widehat\bF_0}^{k-j}\left[ \vf'\partial_s\cL_{\widehat\bF_0}^j\left[\ln(1-s) +\ln\bar\nu\right]h_0\right]\right)\\
&=-\sum_{i=0}^{k-j-1}\bar\nu^{2i+1}\ell_0\left(\vf'\circ F_0^{k-j}\vf'\partial_s\cL_{\widehat\bF_0}^j\left[\ln(1-s) +\ln\bar\nu\right]h_0\right)\\
&\phantom{= \;}+\bar\nu^{k-j}\ell_0\left(\vf'\circ F_0^{k-j}\vf'\partial_s^2\cL_{\widehat\bF_0}^{j}\left[\ln(1-s) +\ln\bar\nu\right]h_0\right)\\
&=\bar\nu^{j+2}\frac{1+\bar\nu^2}{1-\bar\nu^2}\Leb((\vf')^2)\delta_{kj},
\end{split}
\]
where in the last line we have used  \eqref{eq:second-s}. Thus
\begin{equation}\label{eq:second-term}
\begin{split}
\sum_{k=0}^\infty\sum_{j=0}^{k-1}\tilde\mu^{-k}\ell_0\left(\langle \nabla\cL_{\widehat\bF_0}^{k-j-1}\cL_{\widehat\bF_0}'\cL_{\widehat\bF_0}^j\left[\ln(1-s) +\ln\bar\nu\right]h_0, (0,\vf,\vf')\rangle\right)=0.
\end{split}
\end{equation}
The third term reads, setting $\eta=(0,\vf,\vf')$,
\begin{equation}\label{eq:third-term}
\begin{split}
&\tilde\mu\ell_0\left(\langle \nabla (\tilde\mu-\cL_{\widehat\bF_0})^{-1}\left[\ln(1-s) +\ln\bar\nu\right](\tilde\mu-\cL_{\widehat\bF_0})^{-1}\cL_{\widehat\bF_0}\langle\nabla h_0, \eta\rangle, \eta\rangle\right)\\
&=\sum_{k=0}^\infty\sum_{j=1}^\infty\bar\nu^{j+k}\ell_0\left(\vf'\partial_s\cL_{\widehat\bF_0}^{k}\left[\ln(1-s) +\ln\bar\nu\right]
\cL_{\widehat\bF_0}^j(f'\vf')\right)=0.
\end{split}
\end{equation}
The fourth term can be computed analogously yielding as well
\begin{equation}\label{eq:fourth-term}
\begin{split}
\ell_0'\left(\ln(1-s)(\tilde\mu-\cL_{\widehat\bF_0})^{-1}\cL_{\widehat\bF_0}\langle\nabla h_0, (0,\vf,\vf')\rangle\right)=0.
\end{split}
\end{equation}
For the fifth term we must compute
\begin{equation}\label{eq:fifth-term}
\begin{split}
&\tilde\mu^{-k}\ell_0\left(\ln(1-s)\cL_{\widehat\bF_0}^{k-j-1}\cL_{\widehat\bF_0}'\cL_{\widehat\bF_0}^j \langle\nabla h_0, (0,\vf,\vf')\rangle\right)\\
&=\bar\nu^j\ln \bar\nu^{-1} \ell_0(\vf'\partial_s\cL_{\widehat\bF_0}^j \vf'f')=\bar\nu^{j+2}\ln\bar\nu^{-1}\frac{1-\bar\nu^{2j}}{(1-\bar\nu^2)^2}\Leb(\vf'\vf'\circ F_0^j)=0,
\end{split}
\end{equation}
where we used \eqref{eq:f-deriv} and \eqref{eq:no-cobo}.

At last we consider the sixth term
\begin{equation}\label{eq:sixth-term}
\begin{split}
&\ell_0\left(\ln(1-s)(\tilde\mu-\cL_{\widehat\bF_0})^{-1}\cL_{\widehat\bF_0}\langle\nabla h'_0, (0,\vf,\vf')\rangle\right)\\
&=\ln\bar\nu^{-1}\sum_{k=1}^\infty\ell_0(\vf' \partial_s (\tilde\mu-\cL_{\widehat\bF_0})^{-1}\cL_{\widehat\bF_0}f'\vf')=0.
\end{split}
\end{equation}
Collecting equations \eqref{eq:first-term}, \eqref{eq:second-term}, \eqref{eq:third-term}, \eqref{eq:fourth-term}, \eqref{eq:fifth-term} and \eqref{eq:sixth-term} yields
\[
\overline\mu_{e,\param}(\ln\nu_{\param,1}^{-1})=\htop+ \frac{\bar\nu^2\int(\vf')^2}{(1-\bar\nu^2)^2}\param^2+\cO(\param^3)
\]
from which the claim follows.
\end{proof}

To conclude the section let us note a property of $a_\param$ which is useful to check the correctness of the subsequent computations.
\begin{lem}\label{lem:zeroav} For al $\param\in\bR$ small enough and  all $t\geq 0$, we have
\[
\int_{\bT^2}a_\param(t,\xi)d\xi=0.
\]
 \end{lem}
 \begin{proof}
By \eqref{eq:diff-reg}, there exists $\omega_\param$, with $\|\omega_\param\|_\infty\leq \Const$, such that
\begin{equation}\label{eq:p-omega}
p_\param=\param\omega_\param.
\end{equation}
Note that
\[
\int_{\bT^2}\langle V_\param, p_\param\rangle\circ \phi_{\param,t} \,b_\param(t) d\xi=\int_{\bT^2}\langle V_\param, p_\param\rangle\circ \phi_{\param,t}\det(D_\xi\phi_{\param,s}) d\xi=\int_{\bT^2}\langle V_\param, p_\param\rangle d\xi.
\]
On the other hand, recalling \eqref{eq:V-deriv},
\[
0=\int_{\bT^2}\partial_{\xi_k} V_{\param,j}(\xi) d\xi=\int_{\bT^2} V^\perp_{\param,j}(\xi) p_{\param,k}(\xi) d\xi
\]
implies
\[
\int_{\bT^2} V_{\param,j}(\xi) \omega_{\param,k}(\xi) d\xi=0
\]
for all $k,j$. Hence,
\[
\int_{\bT^2}\langle V_\param, p_\param\rangle\circ \phi_{\param,t} \,b_\param(t) d\xi=0.
\]
Hence, using \eqref{eq:aderivative}, we can write
\[
a_\param(\xi,t)=\param\int_0^t \langle V_\param, \omega_\param\rangle\circ \phi_{\param,s}(\xi) \,b_\param(\xi,s) ds,
\]
which, by Fubini, concludes the Lemma.
\end{proof}

\subsubsection{\bfseries Characterization of the obstruction: conclusion}\label{sec:obstruction}\ \newline
We can now  continue our estimate left at \eqref{eq:akbar2}. As already mentioned the first problem is to investigate the prefactor. Using  \eqref{eq:bH1-def} and \eqref{eq:noncisicrede}, we can write
\[
\begin{split}
&\bH^1_{F^{l n_*}(x),\chi_*\bF^{l n_*}_*(\bv,0)}(\bh^*_\param)=\int_{\bR}\langle \bh^*_\param(\phi_{\param,s}(z_l),\bar s_\param\circ \phi_{\param,s}(z_l)), D_{z_l}\phi_{\param,s}D_xF_\param^{l n_*}\bv\rangle \chi_*(z_l, s) ds\\
&=\int_{\bR}\frac{\bar h_{2,\param}}{\varsigma(\bar s_\param)\langle V_\param,V_{+,\param}^\perp\rangle}\circ\phi_{\param,s}(z_l))\langle V_{+,\param}^\perp(\phi_{\param,s}(z_l)), D_{z_l}\phi_{\param,s}D_xF_\param^{l n_*}\bv\rangle \chi_*(z_l, s) ds,
\end{split}
\]
where we used the notation $z_l=F^{n_*l}(x)$.
To continue, we need the following.
\begin{lem}\label{lem:derivativephi}
For each $t\in( 0,1)$ it holds true that
\[
\begin{split}
\langle V_{+,\param}^\perp\circ\phi_{\param,t}, D\phi_{\param,t}V_{+,\param}\rangle=&\param\langle \be, V^\perp\rangle\langle e_1,V\rangle\sum_{k=1}^\infty \bar\nu^{2k+1}\left[\vf'\circ F_\param^{-k}\circ \phi_{\param,t}-\vf'\circ F_\param^{-k}\right]\\
&+\cO(\param^2).
\end{split}
\]
\end{lem}
\begin{proof}
Since $\|V_{+,\param}\|=1$, we have $\partial_{x_k}V_{+,\param}=\param V_{+,\param}^\perp \omega^+_{\param,k}$ for some vector function $ \omega^+_{\param}$.
It turns out to be convenient to take the derivative in the flow direction. Doing so we have, recalling \eqref{eq:deriv-cont}, \eqref{eq:p-omega},
\[
\begin{split}
\frac{d}{ds}\langle V_{+,\param}^\perp\circ\phi_{\param,s}, D\phi_{\param,s}V_{+,\param}\rangle=&-\param\langle V_{+,\param}\circ\phi_{\param,s}, D\phi_{\param,s}V_{+,\param}\rangle \langle \omega^+_\param,V_\param\rangle \circ \phi_{\param,s}\\
&+\param\langle V_{+,\param}^\perp, V_{\param}^\perp\rangle\circ\phi_{\param,s}\langle \omega_\param\circ \phi_{\param,s},V_{+,\param}\rangle\\
=&-\param \langle \omega^+_\param,V_\param\rangle \circ \phi_{\param,s}+\cO(\param^2).
\end{split}
\]
We are left with the task of computing $\omega^+_{\param}$. This is done as in \eqref{eq:odio0}:
\[
\begin{split}
&(\partial_{x_k}DF_\param) V_{+,\param}+\param \omega^+_{\param,k} DF_\param \hat V_{+,\param}^\perp=\partial_{x_k} \nu^u_{\param,1}  V_{+,\param}\circ F_\param\\
&\phantom{(\partial_{x_k}DF) V+p_{k} DF \hat V^\perp=}
+\param \nu^u_{\param,1} \sum_j \partial_{x_k}F_{\param,j}\omega^+_{\param,j}\circ F_\param\hat V_{+,\param}^\perp\circ F_\param.
\end{split}
\]
Which, remembering footnote \ref{foo:nuu} and multiplying the above by $V_{+,\param}^\perp\circ F_\param$ , yields
\[
-\vf'' e_1\langle V_{+,\param}^\perp, e_1\rangle \langle \be, V_{+,\param}\rangle+\omega^+_\param\langle V_{+,\param}^\perp\circ F_\param, DF_\param V_{+,\param}^\perp\rangle=\nu_{\param,1}^{u}(DF_\param)^*\omega^+_\param\circ F_\param.
\]
Note that,
\[
\langle V_{+,\param}^\perp\circ F_\param, DF_\param V_{+,\param}^\perp\rangle=\frac{\nu_\param\langle V_{+,\param}^\perp,V_\param\rangle\circ F_\param}{\langle V_{+,\param}^\perp,V_\param\rangle}.
\]
Thus, 
\[
\begin{split}
&\omega^+_\param=\sum_{k=1}^\infty \frac{\nu_{\param,k}\circ F_\param^{-k}\langle V_{+,\param}^\perp,V_\param\rangle}{\nu_{\param,k}^u\circ F_\param^{-k}\langle V_{+,\param}^\perp,V_\param\rangle\circ F_\param^{-k}}
(DF_\param^{-k})^*\Xi\circ F_\param^{-k}\\
&\Xi=-(\nu_{\param,1}^u)^{-1}\langle V_{+,\param}^\perp,\nabla\vf'\rangle\langle \be, V_{+,\param}\rangle e_1=-\bar\nu\langle V_{\param},\nabla\vf'\rangle\langle \be, V^\perp\rangle e_1+\cO(\param).
\end{split}
\]
Accordingly, remembering \eqref{eq:flowderrep} and \eqref{eq:b-def},
\[
\begin{split}
\frac{d}{ds}\langle V_{+,\param}^\perp\circ\phi_{\param,s}, D\phi_{\param,s}V_{+,\param}\rangle=&\param\sum_{k=1}^\infty \bar\nu^{2k+1}\langle V_\param,\nabla (\vf'\circ F_\param^{-k})\rangle\circ \phi_{\param,s}\langle \be, V^\perp\rangle\langle e_1,V\rangle +\cO(\param^2)\\
=&\param\langle \be, V^\perp\rangle\langle e_1,V\rangle\sum_{k=1}^\infty \bar \nu^{2k+1}\frac d{ds}\vf'\circ F_\param^{-k}\circ \phi_{\param,s}+\cO(\param^2)
\end{split}
\]
from which the Lemma follows by integration.
\end{proof}
Using the Lemma \ref{lem:derivativephi} and recalling Remark \ref{rem:measure0} we have
\[
\begin{split}
&\bH^1_{F^{l n_*}(x),\chi_*\bF^{l n_*}_*(\bv,0)}(\bh^*_\param)=\cO(\nu_{\param,l n_*}\|\bv\|+\param^2\nu^{-1}_{\param,ln_*}\|\bv\|)+\param\frac{\langle \be, V^\perp\rangle\langle e_1,V\rangle\langle V^\perp,\bv\rangle}{\nu_{\param,ln_*}\varsigma(\bar s)}\\
&\times \sum_{k=1}^\infty \bar\nu^{2k+1}\int_{\bR}h_{2,\param}(\phi_{\param,s}(z_l))\left[\vf'\circ F_\param^{-k}\circ \phi_{\param,s}(z_l)-\vf'\circ F_\param^{-k}(z_l)\right]\chi_*(z_l, s) ds
\end{split}
\]
It is convenient to define
\begin{equation}\label{eq:Omega}
\begin{split}
\Omega^\dagger(z,\bv)=&\frac{\langle \be, V^\perp\rangle\langle e_1,V\rangle}{\varsigma(\bar s)}\sum_{k=1}^\infty \bar\nu^{2k+1}\int_{\bR}\Xi^*_{\param,k}(z,s)\chi_*(z, s) ds\\
&+\cO(\nu_{\param,l n_*}^2\|\bv\|+\param\|\bv\|)\langle V^\perp,\bv\rangle^{-1}\\
\Xi^*_{\param,k}(z,s)=&h_{2,\param}(\phi_{\param,s}(z))\left[\vf'\circ F_\param^{-k}\circ \phi_{\param,s}(z)-\vf'\circ F_\param^{-k}(z)\right].
\end{split}
\end{equation}
The function $\Omega^\dagger$ is generically not identically zero for large $l$. To see it just consider the case in which $x$ is a periodic point, say $x=0$ (hence $z_l=0$), then any perturbation that leaves $\vf'(0)$ invariant  but changes the value in a neighborhood will change the value of the integral. On the other hand, for $l\geq \Const \ln\param^{-1}$ the integral is the dominating term in the above expression.

We can finally write, for $\langle V^\perp,\bv\rangle\neq 0$,
\begin{equation}\label{eq:prefactor}
\bH^1_{F^{l n_*}(x),\chi_*\bF^{l n_*}_*(\bv,0)}(\bh^*_\param)=\param\nu^{-1}_{\param,ln_*}\langle V^\perp,\bv\rangle\Omega^\dagger(z_l,\bv).
\end{equation}
Hence \eqref{eq:akbar2} becomes
\begin{equation}\label{eq:akbar3}
\begin{split}
&\langle V^\perp, \nabla \oH_{\param, T}(g)(x) \rangle =\param^2\nu^{-1}_{\param,n_T}\mu_\param^{n_T}\Omega_*^\dagger(x)\cO_{2,\gamma}(g)\\
&\phantom{\langle \bv (x), \nabla \oH_{\param, T}(g)(x) \rangle =}
+\cO(\|g\|_{\cC^r}\|\bv\|\bar\nu^{-n_T}e^{\const \param n_T}),\\
&\Omega^\dagger_*(x)=\nu_{\param,n_T}\mu_\param^{- n_T}\sum_{l=1}^{K_T} \Omega^\dagger(F_\param^{n_*l}(x), V^\perp)(\nu^{-1}\mu_\param)^{ln_*}.
\end{split}
\end{equation}
Our last task it to show that the term $\cO_{2,\gamma}(g)=\param^{-1}\Theta_{\infty,\param}^*$ in equation \eqref{eq:akbar3} does not blow up when $\param\to 0$.

Recalling \eqref{eq:stop}, \eqref{eq:gamma-def} and setting $\pf(s)=(1+s)^2+(2+s)^2$ we have
\[
\begin{split}
&\hAcc\circ\theta\circ \widehat\bF_\param^{-1}(x,y,s)=\frac{1+s^2}{\| D_xF_\param^{-1} (1,s)\|^2}\\
&\phantom{\hAcc\circ\theta\circ \widehat\bF_\param^{-1}(x,y,s)}=\frac{1+s^2}{2-6s+5s^2}-\param\frac{2(1+s^2)(2s-1)(1-s)}{[2-6s+5s^2]^2}\vf'(x)+\cO(\param^2)\\
&\phantom{\hAcc\circ\theta\circ \widehat\bF_\param^{-1}(x,y,s)}=A_0(s)+\param A_1(s)\vf'(x)+\cO(\param)\\
&\hBcc\circ\theta(x,y,s)=-\param\vf''(x)\frac{3-2\param\vf'(x)+2s}{(1-\param\vf'(x)+s)^2+(2-\param\vf'(x)+s)^2}e_1\\
&\phantom{\hBcc\circ\theta(x,y,s)}=-\param\frac{\vf''(x)[3+2s]}{Q(s)}e_1-2\param^2\frac{(\vf'(x)^2)'[2+3s+s^2]}{Q(s)^2}e_1+\cO(\param^3)\\
&\phantom{\hBcc\circ\theta(x,y,s)}=\param B_0(s)\vf''(s)e_1+\param^2 B_1(s)(\vf'(x)^2)' e_1+\cO(\param^3)\\
&\hDcc\circ \theta(x,y,s)=\frac{3(s^2+s-1)}{Q(s)}+\gamma\vf'(x)\frac{4(s^4+3s^3+3s^2+3s+2)}{Q(s)^2}+\cO(\gamma^2)\\
&\phantom{\hDcc\circ\theta(x,y,s)}=E_0(s,\param)+\param \vf'(x)E_1(s)+\cO(\gamma^2)\\
&\Gamma\circ \theta(x,s)= -\param \vf''(x)\varsigma(s)^{-2}e_1.
\end{split}
\]
Note, for future reference, that $A_0(\bar s)=\bar \nu^2$, $E_0(\bar s, \param)=0$.

Thus, remembering \eqref{eq:slick},  our assumption that $\varsigma g\circ\bpi$ belongs to the kernel of $\ell_\param$ and equations \eqref{eq:derivmu}, \eqref{LF-deriv},\eqref{eq:barpi} we have
\begin{equation}\label{eq:pallagalattica0}
\Theta_{\infty,\param}^*=\param \Delta_1+\cO(\param^2\|g\|_{\cC^r})
\end{equation}
where
\begin{equation}\label{eq:pallagalattica1}
\begin{split}
\Delta_1=&\sum_{k=0}^{\infty}\bar\nu^{3k+1}\frac{(1-\bar s)\varsigma'(\bar s)}{\varsigma(\bar s)}\Leb\left(\vf''\circ F_0^{k} g\right)\\
&-\sum_{k=0}^{\infty}\bar\nu^{k+2}(1-\bar s)\frac{1-\bar \nu^{2k}}{1-\bar \nu^2}\Leb(\vf''\circ F_0^{k} g)\\
&-\sum_{j=0}^{\infty}\bar\nu^{k} \frac{3+2\bar s}{(1+\bar s)^2+(2+\bar s)^2}\Leb( \vf''\circ F_0^{k}  g).
\end{split}
\end{equation}
Note that, since $\varsigma'(\bar s)/\varsigma(\bar s)=-\bar\nu/(1-\bar \nu^2)$ and $\bar s^2+\bar s-1=0$, $\Delta_1=0$, hence 
\begin{equation}\label{eq:sobeit}
\cO_{2,\param}(g)=\cO(\param\|g\|_{\cC^r}).
\end{equation}
\begin{rem} We could have reached the same conclusion by using equations \eqref{eq:atlast}, \eqref{eq:akbar1}. Also, using the above formulae and \eqref{eq:atlast} one can check that also $\partial_\param\cO_{2,\gamma}|_{\param=0}$ is a measure (and, with a lot more work, one could compute it). This lends credibility to the possibility $\cO_{2,\param}\equiv 0$.
\end{rem}

%%%%%%%%%%%%%%%%%%%%%%%%%%%%%%%%%%%%%%%%%%%%%%%%%%%%%%%%
\appendix
\section{A little classification}\label{sec:classy}
Here we provide the proof of a partial classification of the flows that satisfy our conditions.
\begin{proof}[{\bf Proof of Lemma \ref{lem:classify}}]
The map $F$ is topologically conjugated to a linear automorphism \cite[Theorem 18.6.1]{KH}. Such conjugation shows that the flow is topologically orbit equivalent to a rigid rotation. Hence one can chose a global Poincar\`e section and the associated Poincar\`e map. Such a map will have a rotation number determined by the foliation of the total automorphism, which a straightforward computation shows to have the claimed property.

Conversely, if $\phi_t$ has no fix points nor periodic orbits, then  there exists a global section uniformly transversal to the flow (see \cite{Siegel52} for the original work, or \cite{GiorgilliMarmi10} for a brief history of the problem and references) and the associated Poincar\`e map is  a $\cC^{1+\alpha}$ map of the circle with irrational rotation number $\omega$. To claim that the Poincar\`e map is conjugated to a rigid rotation requires however some regularity. In particular, if $\alpha\geq 1$, then Denjoy Theorem \cite[Theorem 12.1.1]{KH} implies that the Poincar\`e map is topologically conjugated to a rigid rotation. If $\omega$ is Diophantine, then for $\alpha\geq 2$ it is possible to show that the conjugation is $\cC^{\beta}$ for all $\beta<\alpha$, \cite[Th\'eor\`em fundamental, page 8]{Herman}. Then, if $\omega$ satisfies property \eqref{eq:diofanto}, we can view a linear foliation
as the stable foliation of a toral automorphism. We then obtain a $\cC^{\beta}$ Anosov map with the wanted properties by conjugation.
\end{proof}

%%%%%%%%%%%%%%%%%%%%%%%%%%
\section{Anisotropic Banach spaces: distributions} \label{sec:norms}

In this section we first construct the Banach spaces used in Section \ref{sec:growth}, then we discuss the relation with the Banach spaces constructed in \cite{GouezelLiverani08}, finally, we prove Proposition \ref{th:base} and show that $\bH$ is a bounded functional.

The construction of the Banach spaces are based on the definition of appropriate norms. The Banach spaces are then obtained by closing $\cC^r(\Omega,\bC)$ with respect to such norms.\footnote{ We consider complex valued functions because we are interested in having  nice spectral theory.} The basic idea is to control not the functions themselves but rather their integrals along  curves close to the stable manifolds. Hence the first step is to define the set of relevant curves.\footnote{ In fact, in the simple case at hand, we could consider directly pieces of stable manifolds. We do not do it to make easier to use already existing results.} To do so we need to fix $\delta\in (0,1/2)$ and $K\in\bRp$.

\begin{defin}[Admissible leaves] \label{def:leaves}
Given $r\in\bRp$, an admissible leave $W\subset \bT^2$ is a $\cC^r$ curve with length in the interval $[\delta/2,\delta]$. We require that there exists a parametrization $\omega:[0,1]\to W$ of such a curve such that $\omega'(\tau)\in C^s(\omega(\tau))$, for all $\tau\in [0,1]$, and $\|\omega\|_{\cC^r([0,1],\bT^2)}\leq K$. Moreover we ask $(\omega(\tau),\omega'(\tau)\|\omega'(\tau)\|^{-1})\in\Omega$, that is the curves have all the chosen orientation.
We call $\Sigma$ the set of admissible curves where to any $W\in \Sigma$ is associated a parameterization $\omega_W$ satisfying the properties mentioned above.
\end{defin}
The above set is not empty as it contains pieces of stable manifolds, provided $K$ has been chosen large enough, since the stable manifolds are uniformly $\cC^r$, \cite{KH}.
The basic fact about admissible curves is that if $W\in \Sigma$, then, for each $n\in\bN$, $F^{-n}W\subset\cup_{i=1}^{N_n}W_i$ for some finite set $\{W_i\}_{i=1}^{N_n}\subset \Sigma$. This is quite intuitive but see \cite{GouezelLiverani06} for a detailed proof in a more general setting.

Next, we define the integral of an element $\bg\in \cC^r(\Omega,\bC)$ along an element $W\in\Sigma$ against any $\vf\in\cC^0(W,\bC)$:
\begin{equation}\label{eq:int-prelim}
\int_W\vf \bg:=\int_0^1 ds\, \vf\circ \omega_W(s)\cdot \bg(\omega_W(s),\omega_W'(s)\|\omega_W'(s)\|^{-1})\|\omega_W'(s)\|.
\end{equation}

Also, given  $W\in\Sigma$ and $\vf : W \to \bR$  we set, for all $s\leq r$, 
\begin{equation} \label{eq:normovercurve}
 \| \vf \|_{\cC^s(W,\bR)} \doteq \| \vf \circ \omega_W \|_{\cC^s([0,1],\bR)} .   
\end{equation}

We are now ready to define the relevant semi-norms:\footnote{ By $\cC^s_0(W,\bC)$ we mean the $\cC^s$ functions with support contained in $\textrm{Int}(W)$. The fact that the test functions must be zero at the boundary of $W$ is essential for the following arguments.}
\begin{equation}\label{eq:def-norms}
\|  \bg\|_{p,q}  := \sup_{W \in \Sigma} \sup_{|\alpha| \leq p} \; \sup_{\|\varphi\|_{ \cC_0^{q+|\alpha|}(W,\bC)}\leq 1}
 \int_{W}  \varphi \cdot \partial^\alpha (\bg),
\end{equation}
where $\alpha=(\alpha_1,\alpha_2,\alpha_3)$ is the usual multi--index and $1,2$ refer to the $x$ co-ordinate while $3$ refers to $v$.\footnote{ To be more explicit, if we choose a chart $v=(\cos\theta, \sin\theta)$, then $\alpha_3$ refers to the derivative with respect to $\theta$.} It is easy to check that the $\| \cdot \|_{p,q} $ are indeed semi-norms on $\cC^r(\Omega,\bC)$.
\begin{defin}[$\cB^{p,q}$ spaces]
Let $p \in \bN^*, q \in \bR$, $p+q\leq r$ and $q>0$.   We define $\cB^{p,q}$ to be the closure of $\cC^r(\Omega,\bC)$ with respect to the semi-norm $\|\cdot\|_{p,q}$.\footnote{ To be precise the elements of $\cB^{p,q}$ are the equivalence classes determined by the equivalence relation $h\sim \bar h$ if and only if $\|h-\bar h\|_{p,q}=0$.}
\end{defin}
\begin{rem} Note that $\|\bg\|_{p,q}\leq \|\bg\|_{\cC^{p+q}}$.
\end{rem}
The Banach spaces defined above are well suited for the tasks at hand but, unfortunately, they are not exactly the one introduced in \cite{GouezelLiverani08} where a more general theory is put forward. To avoid having to develop the theory from scratch, it is convenient to show how to relate the present setting to the one in \cite{GouezelLiverani08}. To this end let us briefly recall the construction in \cite{GouezelLiverani08}, then we will explain the relation with the present one. This will allow us to apply the general results in  \cite{GouezelLiverani08} to the present context. 

We start by recalling, particularizing them to our simple situation, the basic objects used in \cite{GouezelLiverani08}: the $r$ times differentiable sections $\cS^r$ of a line bundle over the Grassmannian of one dimensional subspaces. More precisely, let $\cG=\{(x,E)\}$ where $x\in\bT^2$ and $E\subset \bR^2$ is a linear one dimensional subspace, then $h\in \cS^r$ is a $\cC^r$ map $(x,E)\to E^*$.\footnote{ To be precise, since we are going to do spectral theory, we should consider the complex dual. We do not insist on this since the complexification is totally standard.} Note that there is a strict relation between $\cS^r$ and $\cC^r(\Omega,\bC)$: for each $(x,v)\in\Omega$ let $E_v=\{\mu v\}_{\mu\in\bR}$, then for each $h\in\cS^r$ define 
${\boldsymbol i}:\cS^r\to\cC^r(\Omega,\bC)$ by
\begin{equation}\label{eq:conjugate}
[{\boldsymbol i}h](x,v)= h(x,E_v)(v).
\end{equation}
The important fact is that the elements of $\cS^r$, when restricted to the tangent bundle of $W$, are volume forms on $W$, hence can be integrated. Let us be explicit: given $W\in\Sigma$, $h\in\cS^r$ and $\vf\in\cC^0(\bT^2,\bC)$, by \eqref{eq:int-prelim} and \cite[Section 2.2.1]{GouezelLiverani08} we have
\begin{equation}\label{eq:integrate-form}
\int_W\vf h:=\int_0^1 ds\, \vf\circ \omega_W(s) \,h(\omega_W(s),E_{\omega_W'(s)})(\omega_W'(s))=\int_W\vf\, {\boldsymbol i}h.
\end{equation}
Finally, note that the norm in \cite{GouezelLiverani08} is also given by integrals along curves in $\Sigma$. Accordingly, if $h,\tilde h\in\cS^r$ differ only for $(x,E)$ such that $E$ does not belong to $C^s(x)$, then any norm of the difference based on integrations along curves in $\Sigma$ will be zero. The readers can then check that the norms defined in \cite{GouezelLiverani08} are equivalent to $\|{\boldsymbol i} h\|_{p,q}$. Thus ${\boldsymbol i}$ extends, by density, to a Banach space isomorphism between the spaces defined in \cite{GouezelLiverani08} and the $\cB^{p,q}$ presently defined.\footnote{ Note that, not by chance, the Banach spaces in \cite{GouezelLiverani08} are named similarly: $\cB^{p,q,1}$. The superscript $1$ refers there to the fact that, as we will see briefly, in the present language we do not need to have a weight in the transfer operator.} Finally, we have to understand how the operator $\cL_\bF$ reads in the corresponding language of  \cite{GouezelLiverani08}.
To this end it is useful to introduce the operator $\Xi:\cC^r(\Omega,\bC)\to \cC^r(\Omega,\bC)$ defined by
\[
(\Xi\bg)(x,v):=\bg(x,v)\|V(x)\|.
\]
Note that, by the assumptions of Definition \ref{def:flowmapsetup}, $\Xi$ is invertible and both the operator and its inverse can be extended to a continuous operator on $\cB^{p,q}$. It then follows by equations \eqref{eq:int-prelim}, \eqref{eq:integrate-form}, \eqref{eq:new-to}, and \cite[Section 3.2]{GouezelLiverani08} that, for all $W\in\Sigma$ and $\vf\in\cC^0(\bT^2,\bC)$, we have
\[
\begin{split}
&\int_W\vf\, {\boldsymbol i}^{-1}\Xi^{-1}\cL_{\bF}\Xi{\boldsymbol i}h=\int_0^1 ds\,\vf(\omega_W(s))(\Xi^{-1}\cL_{\bF}\Xi{\boldsymbol i}h)(\omega_W(s), \widehat\omega'_W(s))\|\omega'_W(s)\|\\
&=\int_0^1 ds\,\vf(\omega_W(s))({\boldsymbol i}h)\left(F^{-1} \omega_W(s),\frac{D_{\omega_W(s)}F^{-1}\omega'_W(s)}{\|D_{\omega_W(s)}F^{-1}\omega'_W(s)\|}\right) \|D_{\omega_W(s)}F^{-1} \omega'_W(s))\|\\
&=\int_{F^{-1}W}\vf\circ F\, {\boldsymbol i}h=\int_{F^{-1}W}\vf\circ F\, h=\int_W\vf F_*h,
\end{split}
\]
where we used the notation $ \widehat\omega'_W(s)= \omega'_W(s)\|\omega'_W(s))\|^{-1}$. Hence we conclude that
\begin{equation}\label{eq:operatros-iso}
{\boldsymbol i}^{-1}\Xi^{-1}\cL_{\bF}\Xi{\boldsymbol i}h=F_*h:=(F^{-1})^*h,
\end{equation}
that is $\cL_\bF$ is conjugated to the push-forward of $F$ on $\cS^r$. 

\begin{proof}[{\bf Proof of Proposition \ref{th:base}}]
Since \eqref{eq:operatros-iso} states that our operator is conjugated to the push-forward $F_*$, all the spectral properties of $F_*$, acting on $\cB^{p,q,1}$, and $\cL_\bF$, acting on $\cB^{p,q}$, coincide.
It thus suffices to note that \cite[Proposition 4.4, Theorem 5.1, Theorem 6.4]{GouezelLiverani08}  state that, for $q\in\bRp$, $p\in \bNp$ and $p+q\leq r$, $F_*$ can be extended continuously to $\cB^{p,q,1}$, that the logarithm of the spectral radius of $F_*$ is given by the topological entropy (which is the maxim of the metric entropy), that the maximal eigenvalue is simple and $F_*$ has a spectral gap and the essential spectral radius is bounded by $e^{\htop}\lambda^{-\min\{p,q\}}$. 
\end{proof}
We have thus seen that the operator $\cL_{\bF}$ acts very nicely on the spaces $\cB^{p,q}$. The next important fact is  that the functionals we are interested in are well behaved on such spaces.
\begin{lem}\label{lem:boundedfunctional}
There exists $C>0$ such that, for each $x\in\bT^2$, $q\in\bR_>$, $p\in \bN_>$, $p+q\leq r$, and $\vf\in\cC_0^r(\bRpe,\bR)$, $\bg\in\cC^r(\Omega, \bR)$ we have
\[
\left|\bH_{x, \vf}(\bg)\right| \leq C|\supp\vf|\|\bg\|_{p,q}\|\vf\|_{\cC^{p+q}}.
\]
\end{lem}
\begin{proof}
Let us start by considering the case $\supp\vf\subset [a,a+ \delta]$, for some $a>0$. Then, $\{\phi_t(x)\}_{t\in [a,a+\delta]}$ is the re-parametrization of a curve $W$ in $\Sigma$, provided the constant $K$ in Definition \ref{def:leaves} has been chosen large enough.  To see it just consider the parametrization $\omega_W(s)=\phi_{a+\delta s}(x)$. Moreover, setting $\tilde \vf(\phi_{s}(x))=\vf(s)$,\footnote{ Note that, since the stable manifolds are uniformly $\cC^r$, \cite{KH}, $\|\tilde\vf\|_{\cC^r(W,\bR)}\leq \Const\|\vf\|_{\cC^r(\bRp,\bR)}$.} by \eqref{eq:int-prelim} and \eqref{eq:molli}
\[
\begin{split}
\int_W \tilde\vf \bg&=\int_0^1ds\; \vf\circ \phi_{a+\delta s}(x) \bg( \phi_{a+\delta s}(x), \hV( \phi_{a+\delta s}(x)))\|\hV( \phi_{a+\delta s}(x))\|\delta\\
&=\int_{\bR}ds\;\vf(s) (\Xi\bg)\circ \bphi_s(x,\hV(x))=\bH_{x,\vf}(\Xi \bg).
\end{split}
\]
Since the first quality on the left is exactly one of the functionals used in \eqref{eq:def-norms} to define the norm ($p=0$) and $\Xi^{-1}$ is a bounded operator on each space $\cB^{p,q}$, we have
\[
\left\|\bH_{x,\vf}(\bg)\right\|\leq \Const\|\Xi^{-1}\|_{0,p}\|\vf\|_{\cC^q} \|\bg\|_{0,q}\leq \Const\|\vf\|_{\cC^{p+q}} \|\bg\|_{p,q}.
\]
The Lemma follows then by using a partition of unity. 
\end{proof}
\section{Anisotropic Banach spaces: currents} \label{sec:norms-bis}
In this appendix we briefly describe the Banach spaces of currents used in our second results and sketch the needed facts. We will be much faster than in Appendix \ref{sec:norms},  we will omit several details as the construction is very similar to the previous one and no essentially new ideas are present. 

We consider the same set of admissible leaves detailed in Definition \ref{def:leaves}. For each $W\in\Sigma$, let $\cV^q$ be the set of $\cC^q$ vector fields compactly supported on $W$ and with $\cC^q$ norm bounded by one. Then, for each smooth one form $\bgg$  on $\Omega$ we define
\begin{equation}\label{eq:def-norms-bis}
\|  \bgg\|_{p,q}  := \sup_{W \in \Sigma} \sup_{|\alpha| \leq p}  \sup_{\varphi \in \cV^{q+|\alpha|}}
 \int_{W}  \left[\partial^\alpha (\bgg)\right](\varphi),
\end{equation}
where the integral is defined as in the previous section.

Note that there exists a standard isomorphism $\bi$ from vector fields to one forms, so that $\bgg(\vf)=\langle \bgg,\bi(\vf)\rangle$.\footnote{ See \cite{GLP13} for the relevant definition of scalar product between forms in the present context.} Thus the above norm is equivalent to the norm $\|\cdot\|_{p,q,1}$ used in \cite{GLP13}. Let $\cA$ be the set of $\cC^\infty$ one forms on $\Omega$ such that, for all $v\in\bR^2$, $\bgg((0,v))=0$. If we define $\widehat\cB^{p,q}$ as the closure of $\cA$ with respect to the above norm, we obtain a space isomorphic to a subspace of the space $\cB^{p,q,1}$ defined in \cite{GLP13}.

Unfortunately, the transfer operator used here differs from the one studied in \cite{GLP13} insofar it has a potential, which was absent in \cite{GLP13}. In principle, we should therefore prove the Lasota-Yorke inequality for our operator and compute the spectral radius for the present operator via a variational principle (as in \cite{GouezelLiverani08}). Since such a computation is completely standard but a bit lengthy, we just state a partial result that suffices for our goals (in particular we do not bother computing exactly the spectral radius). Such a result follows by copying the computations made in \cite{GLP13} to obtain the Lasota-Yorke inequality. Such computations are exactly the same, apart from the need to keep track of the potential, which can be done easily:

\begin{itemize}
\item The operator $\widehat\cL_{\bF}$ extends continuously on $\widehat\cB^{p,q}$, has spectral radius $\rho$ and essential spectral radius strictly bounded by $\lambda^{-\min\{p,q\}}\rho$.

\item For all $w\in\cC^r$ and $x\in\bT^2$ and $p+q\leq r$ we have\footnote{ This follows immediately from the definition of the norm.}
\[
\left|\bH^1_{x,w }(\bgg)\right| \leq \Const|\supp w|\|\bgg\|_{p,q}\|w\|_{\cC^{p+q}}.
\]
\end{itemize}
The above two facts are all we presently need.

\addcontentsline{toc}{section}{ References}
\bibliographystyle{amsplain}
\bibliography{bibliografia}

\def\cprime{$'$} \def\cprime{$'$} \def\cprime{$'$}
\providecommand{\bysame}{\leavevmode\hbox to3em{\hrulefill}\thinspace}
\providecommand{\MR}{\relax\ifhmode\unskip\space\fi MR }
\providecommand{\MRhref}[2]{%
  \href{http://www.ams.org/mathscinet-getitem?mr=#1}{#2}
}
\providecommand{\href}[2]{#2}
\begin{thebibliography}{10}

\bibitem{AA}
Alexander Adam, \emph{Generic non-trivial resonances for {A}nosov
  diffeomorphisms}, preprint arXiv:1605.06493. To appear in Nonlinearity
  (2016).

\bibitem{Baladi00}
V.~Baladi, \emph{Positive transfer operators and decay of correlations},
  Advanced Series in Nonlinear Dynamics, vol.~16, World Scientific Publishing
  Co. Inc., River Edge, NJ, 2000.

\bibitem{BaladiGouezel10}
V.~Baladi and S.~Gou{\"e}zel, \emph{Banach spaces for piecewise cone-hyperbolic
  maps}, J. Mod. Dyn. \textbf{4} (2010), no.~1, 91--137. \MR{2643889}

\bibitem{BaladiTsujii00}
V.~Baladi and M.~Tsujii, \emph{Anisotropic {H}\"older and {S}obolev spaces for
  hyperbolic diffeomorphisms}, Ann. Inst. Fourier, Grenoble \textbf{57} (2007),
  no.~1, 127--154.

\bibitem{BaladiTsujii08}
\bysame, \emph{Dynamical determinants and spectrum for hyperbolic
  diffeomorphisms}, Geometric and probabilistic structures in dynamics,
  Contemp. Math., vol. 469, Amer. Math. Soc., Providence, RI, 2008, pp.~29--68.
  \MR{MR2478465}

\bibitem{Baladi-Tsujii08}
\bysame, \emph{Spectra of differentiable hyperbolic maps}, Traces in number
  theory, geometry and quantum fields, Aspects Math., E38, Friedr. Vieweg,
  Wiesbaden, 2008, pp.~1--21. \MR{2427585 (2010e:37038)}

\bibitem{BaladiLiverani11}
Viviane Baladi and Carlangelo Liverani, \emph{Exponential decay of correlations
  for piecewise cone hyperbolic contact flows}, Comm. Math. Phys. \textbf{314}
  (2012), no.~3, 689--773. \MR{2964773}

\bibitem{ban-naud}
Oscar Bandtlow and Frederic Naud, \emph{Lower bounds for the {R}uelle spectrum
  of analytic expanding circle maps},  (2016), arXiv:1605.06247.

\bibitem{BKL02}
Michael Blank, Gerhard Keller, and Carlangelo Liverani,
  \emph{Ruelle-{P}erron-{F}robenius spectrum for {A}nosov maps}, Nonlinearity
  \textbf{15} (2002), no.~6, 1905--1973. \MR{1938476}

\bibitem{BCV16}
T.~Bomfim, A.~Castro, and P.~Varandas, \emph{Differentiability of
  thermodynamical quantities in non-uniformly expanding dynamics}, Adv. Math.
  \textbf{292} (2016), 478--528. \MR{3464028}

\bibitem{Bufetov}
A.~I. Bufetov and B.~Solomyak, \emph{The {H}\"older property for the spectrum
  of translation flows in genus two}, Preprint arXiv:1501.05150 [math.DS]
  (2016).

\bibitem{Bufetov14}
A.I. Bufetov, \emph{Finitely-additive measures on the asymptotic foliations of
  a markov compactum}, Mosc. Math. J. \textbf{14} (2014), no.~2, 205â--224.

\bibitem{Bufetov14bis}
Alexander~I. Bufetov, \emph{Limit theorems for translation flows}, Ann. of
  Math. (2) \textbf{179} (2014), no.~2, 431--499. \MR{3152940}

\bibitem{ButterleyLiverani07}
O.~Butterley and C.~Liverani, \emph{Smooth {A}nosov flows: correlation spectra
  and stability}, Journal of Modern Dynamics \textbf{1, 2} (2007), 301--322.

\bibitem{ButterleyLiverani13}
Oliver Butterley and Carlangelo Liverani, \emph{Robustly invariant sets in
  fiber contracting bundle flows}, J. Mod. Dyn. \textbf{7} (2013), no.~2,
  255--267. \MR{3106713}

\bibitem{Cosentino2005}
S.~{Cosentino}, \emph{{A note on H\"older regularity of invariant distributions
  for horocycle flows.}}, {Nonlinearity} \textbf{18} (2005), no.~6, 2715--2726
  (English).

\bibitem{DemersLiverani}
M.F. Demers and C.~Liverani, \emph{Stability of statistical properties in
  two-dimensional piecewise hyperbolic maps}, Trans. Amer. Math. Soc.
  \textbf{360} (2008), no.~9, 4777--4814. \MR{2403704 (2009f:37021)}

\bibitem{DemersZhang}
M.F. Demers and H.-K. Zhang, \emph{A functional analytic approach to
  perturbations of the {L}orentz gas}, Comm. Math. Phys. \textbf{324} (2013),
  no.~3, 767--830. \MR{3123537}

\bibitem{Dyatlov-Zworski}
S.~Dyatlov and M.~Zworski, \emph{Dynamical zeta functions for {A}nosov flows
  via microlocal analysis}, ArXiv (2013), http://arxiv.org/abs/1306.4203.

\bibitem{Faure-Tsujii13}
F.~Faure and M.~Tsujii, \emph{Band structure of the {R}uelle spectrum of
  contact {A}nosov flows}, C. R. Math. Acad. Sci. Paris \textbf{351} (2013),
  no.~9-10, 385--391. \MR{3072166}

\bibitem{FlaminioForni03}
L.~Flaminio and G.~Forni, \emph{Invariant distributions and time averages for
  horocycle flows}, Duke Mathematical Journal \textbf{119} (2003), no.~3,
  465--526.

\bibitem{FlaminioForni06}
\bysame, \emph{Equidistribution of nilflows and applications to theta sums},
  Ergodic Theory and Dynamical Systems \textbf{26} (2006), no.~2, 409--434.

\bibitem{FlaminioForni07}
\bysame, \emph{On the cohomological equation for nilflows}, Journal of Modern
  Dynamics \textbf{1} (2007), no.~1, 37--60.

\bibitem{Forni-Ulcigrai}
G.~Forni and C.~Ulcigrai, \emph{Time-changes of horocycle flows}, J. Mod. Dyn.
  \textbf{6} (2012), no.~2, 251--273. \MR{2968956}

\bibitem{Forni97}
Giovanni Forni, \emph{Solutions of the cohomological equation for
  area-preserving flows on compact surfaces of higher genus}, Ann. of Math. (2)
  \textbf{146} (1997), no.~2, 295--344. \MR{1477760}

\bibitem{Forni02}
\bysame, \emph{Deviation of ergodic averages for area-preserving flows on
  surfaces of higher genus}, Ann. of Math. (2) \textbf{155} (2002), no.~1,
  1--103. \MR{1888794 (2003g:37009)}

\bibitem{Furstenberg61}
H.~Furstenberg, \emph{Strict ergodicity and transformation of the torus}, Amer.
  J. Math. \textbf{83} (1961), 573--601. \MR{0133429 (24 \#A3263)}

\bibitem{ForniMatheus11}
G.~G.~Forni and C.~Matheus, \emph{Introduction to {T}eichm\"uller theory and
  its applications to dynamics of interval exchange transformations, flows on
  surfaces and billiards}, Bedlewo school ``Modern Dynamics and its Interaction
  with Analysis, Geometry and Number Theory'' (2011).

\bibitem{ghys93}
{\'E}.~Ghys, \emph{Rigidit\'e diff\'erentiale des groupes fuchsines},
  Pubblications Math\'ematiques de L'I.H.\'E.S. \textbf{78} (1993), 163--185.

\bibitem{GiorgilliMarmi10}
A.~Giorgilli and S.~Marmi, \emph{Convergence radius in the
  {P}oincar\'e-{S}iegel problem}, Discrete Contin. Dyn. Syst. Ser. S \textbf{3}
  (2010), no.~4, 601--621. \MR{2684066 (2012a:37098)}

\bibitem{GLP13}
P.~Giulietti, C.~Liverani, and M.~Pollicott, \emph{Anosov flows and dynamical
  zeta functions}, Ann. of Math. (2) \textbf{178} (2013), no.~2, 687--773.
  \MR{3071508}

\bibitem{GH55}
Walter~Helbig Gottschalk and Gustav~Arnold Hedlund, \emph{Topological
  dynamics}, American Mathematical Society Colloquium Publications, Vol. 36,
  American Mathematical Society, Providence, R. I., 1955. \MR{0074810}

\bibitem{GouezelLiverani06}
S.~Gouezel and C.~Liverani, \emph{Banach spaces adapted to {A}nosov systems},
  Ergodic Theory and Dynamical Systems \textbf{26, 1} (2006), 189--217.

\bibitem{GouezelLiverani08}
\bysame, \emph{Compact locally maximal hyperbolic sets for smooth maps: fine
  statistical properties}, Journal of Differential Geometry \textbf{79} (2008),
  433--477.

\bibitem{Hasselblatt97}
B.~Hasselblatt, \emph{Regularity of the {A}nosov splitting. {II}}, Ergodic
  Theory Dynam. Systems \textbf{17} (1997), no.~1, 169--172. \MR{1440773
  (98d:58135)}

\bibitem{Hennion93}
H.~Hennion, \emph{Sur un th\'eor\`eme spectral et son application aux noyaux
  lipchitziens}, Proc. Amer. Math. Soc. \textbf{118} (1993), no.~2, 627--634.

\bibitem{Herman}
Michael-Robert Herman, \emph{Sur la conjugaison diff\'erentiable des
  diff\'eomorphismes du cercle \`a des rotations}, Inst. Hautes \'Etudes Sci.
  Publ. Math. (1979), no.~49, 5--233. \MR{538680}

\bibitem{Kato}
Tosio Kato, \emph{Perturbation theory for linear operators}, Classics in
  Mathematics, Springer-Verlag, Berlin, 1995, Reprint of the 1980 edition.
  \MR{1335452 (96a:47025)}

\bibitem{KH}
A.~Katok and B.~Hasselblatt, \emph{Introduction to the modern theory of
  dynamical systems}, Encyclopedia of Mathematics and its Applications,
  vol.~54, Cambridge University Press, 1995.

\bibitem{KellerLiverani99}
G.~Keller and C.~Liverani, \emph{Stability of the sprectrum for transfer
  operators}, Annali della Scuola Normale Superiore di Pisa \textbf{28} (1999),
  141--152.

\bibitem{Kitaev99}
A.~Yu. Kitaev, \emph{Fredholm determinants for hyperbolic diffeomorphisms of
  finite smoothness}, Nonlinearity \textbf{12} (1999), no.~1, 141--179.

\bibitem{LY85}
F.~Ledrappier and L.-S. Young, \emph{The metric entropy of diffeomorphisms.
  {I}. {C}haracterization of measures satisfying {P}esin's entropy formula},
  Ann. of Math. (2) \textbf{122} (1985), no.~3, 509--539. \MR{819556}

\bibitem{Liverani95}
C.~Liverani, \emph{Decay of correlation}, Annals of Mathematics \textbf{142}
  (1995), 239--301.

\bibitem{Liverani04}
\bysame, \emph{On contact {A}nosov flows}, Annals of Mathematics. Second Series
  \textbf{159} (2004), no.~3, 1275--1312.

\bibitem{Li04}
Carlangelo Liverani, \emph{Birth of an elliptic island in a chaotic sea}, Math.
  Phys. Electron. J. \textbf{10} (2004), Paper 1, 13 pp. (electronic).
  \MR{2111295}

\bibitem{LM}
Carlangelo Liverani and Marco Martens, \emph{Convergence to equilibrium for
  intermittent symplectic maps}, Comm. Math. Phys. \textbf{260} (2005), no.~3,
  527--556. \MR{2182435}

\bibitem{Otal1998}
J-P. Otal, \emph{{Sur les fonctions propres du Laplacien du disque
  hyperbolique.}}, {C. R. Acad. Sci., Paris, S\'er. I, Math.} \textbf{327}
  (1998), no.~2, 161--166 (French).

\bibitem{Rudin}
Walter Rudin, \emph{Principles of mathematical analysis}, third ed.,
  McGraw-Hill Book Co., New York-Auckland-D\"usseldorf, 1976, International
  Series in Pure and Applied Mathematics. \MR{0385023}

\bibitem{ruelle-sullivan}
D.~Ruelle and D.~Sullivan, \emph{Currents, flows and diffeomorphisms}, Topology
  \textbf{14} (1975), no.~4, 319--327. \MR{0415679 (54 \#3759)}

\bibitem{Rugh96}
H.~H. Rugh, \emph{Generalized {F}redholm determinants and {S}elberg zeta
  functions for {A}xiom {A} dynamical systems}, Ergodic Theory Dynam. Systems
  \textbf{16} (1996), no.~4, 805--819.

\bibitem{Siegel52}
C.~L. Siegel, \emph{\"{U}ber die {N}ormalform analytischer
  {D}ifferentialgleichungen in der {N}\"ahe einer {G}leichgewichtsl\"osung},
  Nachr. Akad. Wiss. G\"ottingen. Math.-Phys. Kl. Math.-Phys.-Chem. Abt.
  \textbf{1952} (1952), 21--30. \MR{0057407 (15,222b)}

\bibitem{SinaiKhanin89}
Ya.~G. Sina{\u\i} and K.~M. Khanin, \emph{Smoothness of conjugacies of
  diffeomorphisms of the circle with rotations}, Uspekhi Mat. Nauk \textbf{44}
  (1989), no.~1(265), 57--82, 247. \MR{997684 (90i:58183)}

\bibitem{SinaiKhanin92}
\bysame, \emph{Mixing of some classes of special flows over rotations of the
  circle}, Funktsional. Anal. i Prilozhen. \textbf{26} (1992), no.~3, 1--21.
  \MR{1189019 (93j:58079)}

\bibitem{sli-ban-J}
Julia Slipantschuk, Oscar~F. Bandtlow, and Wolfram Just, \emph{Complete
  spectral data for analytic anosov maps of the torus},  (2016),
  arXiv:1605.02883.

\bibitem{Viana06}
M.~Viana, \emph{Ergodic theory of interval exchange maps}, Rev. Mat. Complut.
  \textbf{19} (2006), no.~1, 7--100. \MR{2219821 (2007f:37002)}

\end{thebibliography}

\end{document}